\documentclass[preprint,3p,11pt]{elsarticle}
\makeatletter
\def\ps@pprintTitle{%
  \let\@oddhead\@empty
  \let\@evenhead\@empty
  \let\@oddfoot\@empty
  \let\@evenfoot\@empty
}
\makeatother
\usepackage{amsmath}
\usepackage{amsfonts}
\usepackage{amssymb}
\usepackage{amsthm}
\usepackage{bbm}
\usepackage{siunitx}
\usepackage{graphicx}
\usepackage[justification=centering]{caption}
\usepackage{subcaption}
\usepackage{xspace}
\usepackage{bm}
\usepackage{threeparttable}
\usepackage[ruled,lined,linesnumbered]{algorithm2e}
\usepackage[english]{babel}
\usepackage{amsmath}
\usepackage{textcomp, gensymb}
\usepackage{float}
\usepackage{booktabs}
\usepackage{placeins}
\usepackage{stmaryrd}
\usepackage{soul}
\usepackage{cancel}
\usepackage{xcolor}
\usepackage{listings}
\usepackage{hyperref}
\hypersetup{colorlinks=true, citecolor=blue, linkcolor=blue, urlcolor=blue}
\usepackage{tikz}
\usetikzlibrary{matrix,backgrounds,shapes}
\usetikzlibrary{positioning}
\usetikzlibrary{fit}
\usetikzlibrary{plotmarks}
\usetikzlibrary{decorations.pathreplacing}
\usepackage{pgfplots}
\pgfplotsset{compat=1.18}
\usetikzlibrary{shapes}
\usetikzlibrary{positioning,arrows}
\usetikzlibrary{fadings,shapes.arrows,shadows}
\usetikzlibrary{arrows.meta}
\usepackage{orcidlink}

\definecolor{codegray}{rgb}{0.95,0.95,0.95}
\lstdefinestyle{pythonstyle}{language=Python, backgroundcolor=\color{codegray}, basicstyle=\ttfamily\small, keywordstyle=\bfseries, commentstyle=\itshape, stringstyle=\ttfamily, numbers=left, numberstyle=\tiny, numbersep=8pt, frame=single, breaklines=true, showstringspaces=false, tabsize=4, captionpos=b}

\newtheorem{theorem}{Theorem}[section]

\newtheorem{remark}{Remark}[section]

\newcommand\Rey{\mbox{\textit{Re}}}

\newcommand\Mach{\mbox{\textit{Ma}}}

\newcommand{\Tr}{\ensuremath{^{\mr{\top}}}}
\newcommand{\mr}[1]{\ensuremath{\mathrm{#1}}}
\newcommand{\fnc}[1]{\ensuremath{\mathcal{#1}}}

\newcommand{\mat}[1]{\ensuremath{\mathsf{#1}}}
\newcommand{\xm}[0]{\ensuremath{x_{m}}}
\newcommand{\xil}[0]{\ensuremath{\xi_{l}}}
\newcommand{\nxil}[0]{\ensuremath{n_{\xil}}}
\newcommand{\Q}[0]{\ensuremath{\bm{\fnc{Q}}}}
\newcommand{\W}[0]{\ensuremath{\bm{\fnc{W}}}}
\newcommand{\Fxm}[0]{\ensuremath{\bm{\fnc{F}}_{\xm}^{(I)}}}
\newcommand{\Fxmv}[0]{\ensuremath{\bm{\fnc{F}}_{\xm}^{(V)}}}
\newcommand{\FxmALE}[0]{\ensuremath{\bm{\fnc{F}}_{\xm}^{(ALE)}}}
\newcommand{\GB}[0]{\ensuremath{\bm{\fnc{G}}^{(B)}}}
\newcommand{\J}[0]{\ensuremath{\fnc{J}}}
\newcommand{\Gzero}[0]{\ensuremath{\bm{\fnc{G}}^{(0)}}}
\newcommand{\Uone}[0]{\ensuremath{\fnc{U}_{1}}}
\newcommand{\Utwo}[0]{\ensuremath{\fnc{U}_{2}}}
\newcommand{\Uthree}[0]{\ensuremath{\fnc{U}_{3}}}
\newcommand{\Ui}[0]{\ensuremath{\fnc{U}_{i}}}
\newcommand{\U}[0]{\ensuremath{\bm{\fnc{U}}}}

\newcommand{\Vi}[0]{\ensuremath{\fnc{V}_{i}}}
\newcommand{\Vm}[0]{\ensuremath{\fnc{V}_{m}}}
\newcommand{\Vg}[0]{\ensuremath{\bm{\fnc{V}}}}
\newcommand{\Vgi}[1]{\ensuremath{\fnc{V}_{#1}}}
\newcommand{\E}[0]{\ensuremath{\fnc{E}}}

\newcommand{\Cij}[2]{\ensuremath{\mat{C}_{#1#2}}}
\newcommand{\Chatij}[2]{\ensuremath{\hat{\mat{C}}_{#1#2}}}

\newcommand{\Fxisca}[1]{\ensuremath{\fnc{F}_{\ensuremath{x_{#1}}}}}
\newcommand{\avg}[1]{\left[\!\left[#1\right]\!\right]}
\newcommand{\logavg}[1]{\left[\!\left[#1\right]\!\right]^{\log}}

\begin{document}
\begin{frontmatter}

\title{Entropy-stable moving-wall boundary conditions for the ALE formulation of the  compressible Navier--Stokes equations}

\author[KAUST-CEMSE,POLIMI]{Luca Galimberti\orcidlink{0009-0001-1145-0539}}
\author[KAUST-CEMSE]{Roberto Nuca\orcidlink{0000-0002-9031-5668}}
\author[KAUST-CEMSE]{Lisandro Dalcin\orcidlink{0000-0001-8086-0155}}
\author[POLIMI]{Alberto Guardone\orcidlink{0000-0001-6432-2461}}
\author[KAUST-CEMSE,KAUST-PSE]{Matteo Parsani\orcidlink{0000-0001-7300-1280}}

\address[KAUST-CEMSE]{King Abdullah University of Science and Technology (KAUST), Computer Electrical and Mathematical Science and Engineering Division (CEMSE), Thuwal, Saudi Arabia}
\address[POLIMI]{Politecnico di Milano, Dipartimento di Scienze e Tecnologie Aerospaziali (DAER), Milan, Italy}

\address[KAUST-PSE]{King Abdullah University of Science and Technology (KAUST), Physical Science and Engineering Division (PSE), Thuwal, Saudi Arabia}

\begin{abstract}
We present a high-order entropy-stable framework for the compressible Euler and Navier–Stokes equations on moving domains.
The space–time mapping describing the domain motion is recast in an arbitrary Lagrangian Eulerian (ALE) formulation, in which the physical inviscid fluxes and the contributions induced by mesh motion are treated in a unified manner.

At the continuous level, we prove that the proposed moving-wall boundary conditions are entropy conservative for the Euler equations and entropy stable for the Navier–Stokes equations.
The no-slip condition is formulated in terms of the velocity relative to the moving wall, yielding a bounded inviscid contribution to the entropy balance, while the viscous terms contribute only entropy dissipation.
Using diagonal norm summation-by-parts (SBP) operators together with appropriate numerical fluxes, these properties are extended to the semi-discrete formulation, resulting in nonlinear stability in the $L^2$ sense.

The accuracy, robustness, and scalability of the proposed method are demonstrated in practice through an extensive set of numerical experiments, ranging from canonical two-dimensional verification cases to large-scale turbulent and supersonic simulations involving moving boundaries and fluid–structure interaction. The results confirm the suitability of high-order entropy-stable schemes for complex moving-domain problems across a broad range of flow regimes and multiphysics applications.

Because the analysis relies on the SBP property rather than on a particular discretization, the framework naturally extends to a broad class of methods based on diagonal-norm SBP operators, including finite volume, finite element, and flux reconstruction schemes.
\end{abstract}

\end{frontmatter}

\section{Introduction}

Arbitrary Lagrangian--Eulerian (ALE) formulations provide a natural framework for the numerical simulation of flows on moving and deforming domains. Such problems arise in a wide range of engineering applications, including configurations with moving or rotating components, fluid-structure interaction, and prescribed body motion. From a computational perspective, ALE methods are also attractive because mesh motion can preserve a body-fitted discretization and enable redistribution of degrees of freedom, also referred to as $r$-adaptation, thereby improving the distribution of resolution in regions of interest, particularly near solid walls and within boundary layers.

The ALE framework has a long history, beginning with the seminal work of Hirt and co-workers \cite{hirt1974ALE} and followed by important developments for incompressible flow, finite-element formulations, and fluid-structure interaction \cite{hughes1981ALE,donea1982ALE,donea1992TransientAdvection}. Subsequent work established conservation and geometric conservation law properties on moving meshes \cite{farhat2001DiscreteGCL}, as well as stability analyses for model problems \cite{donea2004ALEMethodsReview,formaggia2004ALELinearStability} and high-order discretizations \cite{nikkar2015EnergyStableDeforming,kopriva2016MovingHexahedral}. More recently, ALE formulations have been combined with high-order summation-by-parts (SBP) and discontinuous Galerkin (DG) discretizations, including entropy-conservative and entropy-stable formulations for the compressible Navier--Stokes equations \cite{yamaleev2019EntropyStableALE,schnucke2020EntropyStableALE,krais2020SplitFormALE}.

The growing interest in high-order/variable-order ALE methods is closely tied to current demands in computational fluid dynamics \cite{slotnick2014CFDVision2030}. Advances in computational power have made higher-fidelity simulations increasingly feasible, including large-eddy simulations and other scale-resolving computations of unsteady flows in geometrically complex configurations. In this setting, high-order methods are particularly attractive because they can deliver a given level of accuracy with fewer degrees of freedom (DOFs) than low-order methods and typically exhibit improved dispersion and dissipation properties. Their robustness, however, depends critically on the design of provably stable formulations, especially for under-resolved flows, discontinuities, sharp gradients, and in the presence of complex boundaries.

Among the available stabilization frameworks, entropy stability has emerged as a particularly appealing approach for compressible flows because, when positivity is satisfied, it provides a nonlinear stability mechanism consistent with the second law of thermodynamics \cite{tadmor2003ReviewEntropyAnalysis}. For the compressible Navier--Stokes equations, entropy-stable formulations have been developed for a variety of high-order discretizations, including SBP \cite{fisher2013SStabilityFD,fisher2012PhD} and DGSEM/DG-SBP schemes \cite{carpenter2014SSDC,friedrich2018SStabilityHP,fernandez2020SStabilityHPRefinement,fernandez2020SStabilityPRefinement,friedrich2018SStableSpaceTime,chan2018SStabilityDG,crean2018EntropySBPEulerCurved,parsani2016SSDCStaggered}. A recurring difficulty in this theory, however, is the treatment of boundary conditions \cite{svard2018SStabilitySolidWall,parsani2015SStabilitySolidWallBC,svard2014SStabilityEulerWall}. Boundary contributions play a decisive role in the entropy balance, and stable wall boundary conditions are therefore essential to the robustness of the overall discretization \cite{parsani2015SStabilitySolidWallBC,dalcin2019WallBC,chan2022SStableModalDGWallBC} -- solid boundaries is very often where the ``action" begins.

This issue becomes even more delicate in the ALE setting. On moving and deforming meshes, the entropy balance couples the physical fluxes, the grid velocity, and the metric terms introduced by the coordinate transformation. As a result, entropy-stable wall treatments developed for static grids do not directly extend to general ALE boundaries. Although entropy-conservative and entropy-stable ALE formulations for the compressible Navier--Stokes equations have been proposed in the literature, existing analyses either focus on periodic domains or do not provide an analysis and proof of the moving-wall boundary treatment \cite{yamaleev2019EntropyStableALE,schnucke2020EntropyStableALE}.

The present work addresses this gap. We first derive, at the continuous level, the entropy-conservative condition associated with moving-wall boundary conditions. The analysis is carried out for general moving and deforming ALE boundaries and yields physically relevant boundary conditions, including no-slip adiabatic moving walls and prescribed entropy heat flux. We then prove semi-discrete entropy conservation and entropy stability of the proposed wall treatment within a DG-SBP-SAT framework. Although the derivation is presented in that setting, the resulting construction is formulated in a way that is also relevant to other discretization families, including finite-volume, finite-difference, finite-element, flux-reconstruction methods, etc.. Finally, the theoretical results are supported by numerical experiments on challenging moving-boundary test cases.
The paper is organized as follows. Section \ref{sec:entropy_continuous} presents the entropy analysis at the continuous level. Section \ref{sec:semi-discrete} introduces the relevant SBP operators and derives the semi-discrete form of the equations. Entropy conservation for the moving-wall boundary condition is discussed in Section \ref{sec:moving_wall_bc}. A series of numerical test cases, convergence, accuracy, and more realistic applications are presented in Section \ref{sec:numerical-tests}. Finally, conclusions are discussed in Section \ref{sec:conclusions}.

\section{Continuous entropy stability theory}\label{sec:entropy_continuous}
\subsection{The compressible Navier--Stokes equations}\label{subsec:compressible_ns}
We consider the compressible Navier--Stokes equations written in terms of the conservative variables, $\Q$,:
\begin{equation}
\label{eq:compressible_ns}
\begin{split}
&\frac{\partial\Q}{\partial t}+\sum\limits_{m=1}^{3}\frac{\partial \Fxm}{\partial \xm} =\sum\limits_{m=1}^{3}\frac{\partial\Fxmv}{\partial\xm}, \quad
\forall \bm{x} \in\Omega,\quad t\ge 0,\\
&\mathcal{B}\left(\Q,\nabla\Q;\bm{x},t\right)=\GB\left(\bm{x},t\right),
\quad\forall \bm{x}\in\Gamma,\quad t\ge 0,\\
&\Q\left(\bm{x},0\right)=\Gzero\left(\bm{x},0\right), \quad
\forall \bm{x} \in\Omega.
\end{split}
\end{equation}
Here $\mathcal{B}$ denotes the boundary operator. Its precise form is specified later for moving solid walls.
\begin{equation*}
\Q = \left[\rho,\rho\,\Uone,\rho\,\Utwo,\rho\,\Uthree,\rho\,\E\,\right]\Tr,
\end{equation*}
\begin{equation*}
\Fxm = \left[\rho\,\fnc{U}_{m},\rho\,\fnc{U}_{m}\,\Uone+\delta_{m1}\fnc{P},\rho\,\fnc{U}_{m}\,\Utwo+\delta_{m2}\fnc{P},\rho\,\fnc{U}_{m}\,\Uthree+\delta_{m3}\fnc{P},\rho\,\fnc{U}_{m}\fnc{H}\right]\Tr,
\end{equation*}
where $\rho$ is the density, $\Ui$ is the velocity in the $i-$th component, $\fnc{P}$ is the pressure, $\fnc{H}$ is the specific total enthalpy, $\E$ is the specific total energy, $\delta_{i,j}$ is the Kronecker delta and $\Fxm$ is the inviscid flux. The viscous flux $\Fxmv$ is given as
\begin{equation}
\label{eq:Fv}
\Fxmv=\left[0,\tau_{1m},\tau_{2m},\tau_{3m},\sum\limits_{i=1}^{3}\tau_{im}\,\fnc{U}_{i} + \kappa\,\frac{\partial \fnc{T}}{\partial\xm}\right]\Tr,
\end{equation}
where $\fnc{T}$ is the temperature, $\kappa = \kappa(\fnc{T})$ is the thermal conductivity, and the viscous stresses are given by
\begin{equation*}
\tau_{ij} = \mu\left(\frac{\partial\fnc{U}_{i}}{\partial x_{j}}+\frac{\partial\fnc{U}_{j}}{\partial x_{i}}
-\delta_{ij}\frac{2}{3}\sum\limits_{n=1}^{3}\frac{\partial\fnc{U}_{n}}{\partial x_{n}}\right),
\end{equation*}
where $\mu = \mu(\fnc{T})$ is the dynamic viscosity.

The constitutive relations that close the system are:
\begin{equation*}
\fnc{H} = c_{\fnc{P}}\fnc{T}+\frac{1}{2}\,\U\Tr\U,\quad \fnc{P} = \rho \, R \, \fnc{T},\quad R = \frac{R_{u}}{M_{w}},
\end{equation*}
where $c_{\fnc{P}}$ is the specific heat at constant pressure, $\U = \left(\fnc{U}_{1}, \fnc{U}_{2} , \fnc{U}_{3}\right)\Tr$ is the velocity vector, $R_{u}$ is the universal gas constant, and $M_{w}$ is the molecular weight of the gas. Finally, the thermodynamic entropy is given as
\begin{equation*}
s=\frac{R}{\gamma-1}\log\left(\frac{\fnc{T}}{\fnc{T}_{\infty}}\right)-R\log\left(\frac{\rho}{\rho_{\infty}}\right),\quad \gamma=\frac{c_{\fnc{P}}}{c_{\fnc{P}}-R},
\end{equation*}
with $\fnc{T}_{\infty}$ and $\rho_{\infty}$ the reference temperature and density, respectively. The definition of the thermodynamic entropy, $s$, will be useful in the following sections.

\subsection{Geometric mapping}\label{subsec:mapping}
Having introduced the Navier--Stokes equations in Cartesian coordinates, we now account for grid motion. Let the physical domain be described by the space-time coordinates \mbox{$[\,t,\bm{x}\,] = [\,t,x_1,x_2,x_3\,]$} and the reference domain by $[\,\tau,\bm{\xi}\,] = [\,\tau,\xi_1,\xi_2,\xi_3\,]$. The mapping from the reference domain to the physical domain, denoted as $\bm{\chi}(\tau,\bm{\xi})$, is assumed to be invertible and, therefore, induces the inverse mapping $\bm{\xi}(t,\bm{x})=\bm{\chi}^{-1}(t,\bm{x})$. Consequently, one can equivalently compute either $\bm{x}=\bm{\chi}(\tau,\bm{\xi})$ or $\bm{\xi}=\bm{\chi}^{-1}(t,\bm{x})$. 

We first collect a set of differential identities associated with the mapping between the reference and physical domains that will be used in the derivation of the ALE formulation in the next section. Starting from the inverse mapping, we have
\begin{equation*}
    \bm{\xi} = \bm{\chi}^{-1}(t,\bm{x}) = \bm{\chi}^{-1}(t,\bm{\chi}(\tau,\bm{\xi})) = \bm{\xi}(t,\bm{\chi}(\tau,\bm{\xi}))
\end{equation*}
Differentiating this relation with respect to $\tau$ while holding the reference coordinate $\bm{\xi}$ fixed yields
\begin{equation*}
    0 = \frac{d}{d \tau}\bm{\chi}^{-1}\bigl(t,\bm{\chi}(\tau,\bm{\xi})\bigr)
    = \frac{\partial t}{\partial \tau} \frac{\partial \bm{\xi}}{\partial t}
    + \frac{\partial \bm{\xi}}{\partial \bm{x}} \frac{\partial \bm{x}}{\partial \tau}.
\end{equation*}
Since the mapping is invertible, this relation implies the identity
\begin{equation*}
    \frac{\partial \xi_l}{\partial t}
    = -\sum\limits_{m=1}^{3} \frac{\partial \xi_l}{\partial x_m} \Vm,
\end{equation*}
where the grid velocity in physical space is defined by
\begin{equation*}
    \Vg = \frac{\partial \bm{\chi}}{\partial \tau}.
\end{equation*}
In the remainder of this section, $\Vm$ denotes the $m$-th component of the grid velocity $\Vg$.

The temporal and spatial derivatives transform according to
\begin{align}
\frac{\partial}{\partial t}
&= \frac{\partial}{\partial \tau} + \sum\limits_{l=1}^{3} \frac{\partial \xi_l}{\partial t} \frac{\partial}{\partial \xi_l}
= \frac{\partial}{\partial \tau} - \sum\limits_{l,m=1}^{3} \frac{\partial x_m}{\partial \tau} \frac{\partial \xi_l}{\partial x_m} \frac{\partial}{\partial \xi_l}
= \frac{\partial}{\partial \tau} - \sum\limits_{l,m=1}^{3} \Vm \frac{\partial \xi_l}{\partial x_m} \frac{\partial}{\partial \xi_l},
\label{eq:transformation1}
\\
\frac{\partial}{\partial x_m} &= \sum\limits_{l=1}^{3} \frac{\partial \xi_l}{\partial x_m} \frac{\partial}{\partial \xi_l}.
\label{eq:transformation2}
\end{align}

Finally, we define the Jacobian of the mapping from the reference space-time coordinates to the physical space-time coordinates, $\J$, as
\begin{equation}
    \J = \left | \frac{\partial[\,t,\bm{x}\,]}{\partial[\,\tau,\bm{\xi}\,]} \right |, \qquad
\textrm{where} \qquad
    \frac{\partial[\,t,\bm{x}\,]}{\partial[\,\tau,\bm{\xi}\,]} =  
    \begin{bmatrix}
    1 & 0 & 0 & 0\\
    \frac{\partial x_1}{\partial \tau} & \frac{\partial x_1}{\partial \xi_1} & \frac{\partial x_1}{\partial \xi_2} & \frac{\partial x_1}{\partial \xi_3}\\
    \frac{\partial x_2}{\partial \tau} & \frac{\partial x_2}{\partial \xi_1} & \frac{\partial x_2}{\partial \xi_2} & \frac{\partial x_2}{\partial \xi_3}\\
    \frac{\partial x_3}{\partial \tau} & \frac{\partial x_3}{\partial \xi_1} & \frac{\partial x_3}{\partial \xi_2} & \frac{\partial x_3}{\partial \xi_3}\\
    \end{bmatrix}.
    \label{eq:jacobian}
\end{equation}
Fig.~\ref{fig:mapping} shows a graphical representation of the mapping between reference and physical space.
\begin{figure}[H]
    \centering
    \includegraphics[width=0.7\linewidth]{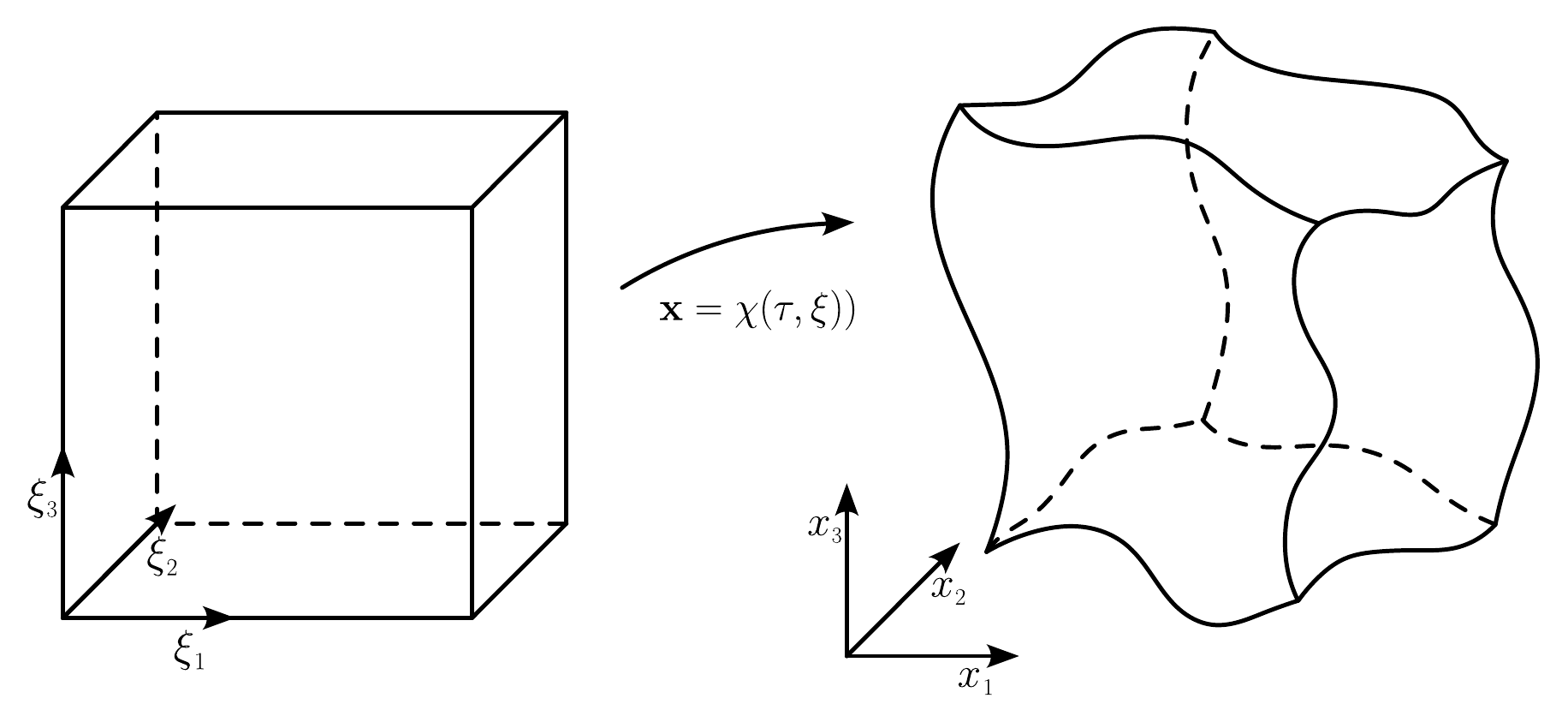}
    \caption{Schematic representation of the mapping from reference to physical space.}
    \label{fig:mapping}
\end{figure}

\subsection{ALE formulation}
Having introduced the compressible Navier--Stokes equations and the coordinate mapping, we now derive the transformed system used in the subsequent analysis. Starting from the compressible Navier--Stokes equations \eqref{eq:compressible_ns}, we multiply by the Jacobian $\J$ defined in \eqref{eq:jacobian} and apply the transformations given in \eqref{eq:transformation1} and \eqref{eq:transformation2} to obtain
\begin{equation*}
    \J \frac{\partial \Q }{\partial \tau} - \J \sum\limits_{l,m=1}^3 \Vm \frac{\partial \xi_l}{\partial \xm} \frac{\partial \Q}{\partial \xi_l} 
    + \J \sum_{l,m=1}^3 \frac{\partial \xi_l}{\partial \xm} \frac{\partial \Fxm}{\partial \xi_l} = \J \sum_{l,m=1}^3 \frac{\partial \xi_l}{\partial \xm} \frac{\partial \Fxmv}{\partial \xi_l}.
\end{equation*}

We seek to recast the equations in a form such that the differential operators act on both the Jacobian and the conservative variables. Likewise, the divergence operator in the reference coordinates acts on the inviscid fluxes and the grid-velocity contribution. By applying the chain rule to the temporal derivative and to the spatial derivative with respect to $\xi_l$, these terms can be written as follows:
\begin{align*}
    \frac{\partial \left( \J \Q \right)}{\partial \tau} 
    &- \sum\limits_{l,m=1}^{3} \frac{\partial}{\partial\xi_l} \left( \J \frac{\partial \xi_l}{\partial \xm} \Vm \Q \right)
    - \Q \left( \frac{\partial \J}{\partial \tau} 
    - \sum\limits_{l,m=1}^{3} \frac{\partial}{\partial\xi_l} \left( \J \frac{\partial \xi_l}{\partial \xm} \Vm \right) \right) \\
    &+ \sum_{l,m=1}^3 \frac{\partial}{\partial \xi_l} \left (\J \frac{\partial \xi_l}{\partial \xm} \left( \Fxm - \Fxmv  \right) \right)
    -  \sum_{m=1}^3 \left( \left( \Fxm - \Fxmv  \right) \sum_{l=1}^3 \frac{\partial}{\partial \xi_l} \left (\J \frac{\partial \xi_l}{\partial \xm} \right) \right) = 0,
\end{align*}
where we identify the following metric identities:
\begin{equation}
    \frac{\partial \J}{\partial \tau} 
    - \sum\limits_{l,m=1}^{3} \frac{\partial}{\partial\xi_l} \left( \J \frac{\partial \xi_l}{\partial \xm} \Vm \right)=0, 
    \label{eq:timeGCL}
\end{equation}
\begin{equation}
    \sum_{l=1}^3 \frac{\partial}{\partial \xi_l} \left (\J \frac{\partial \xi_l}{\partial \xm} \right)=0, \qquad m=1,2,3.
    \label{eq:spaceGCL}
\end{equation}
If the metric identities are satisfied, the compressible Navier--Stokes equations on moving curvilinear grids can be written in the following form:

\begin{equation}
    \frac{\partial \left( \J \Q \right)}{\partial \tau} + \sum\limits_{l,m=1}^3 \frac{\partial}{\partial \xi_l} \left (\J \frac{\partial \xi_l}{\partial \xm} \left( \FxmALE - \Fxmv  \right) \right) = 0,
    \label{eq:continuous_ALE_NS}
\end{equation}
where the ALE flux combines the inviscid flux and grid motion
\begin{equation*}
    \FxmALE = \underbrace{\Fxm}_{\textbf{INVISCID  \, FLUX}} - \underbrace{\Vm \Q}_{\textbf{GRID \, MOTION}}.
\end{equation*}

\subsection{Entropy variables, entropy fluxes, and potentials}
The foundations of entropy inequalities for hyperbolic systems of conservation laws were established in the seminal works of Godunov \cite{godunov1961InterestingClass} and Friedrichs and Lax \cite{friedrichs1971SystemOfConservationEquationsWithConvexExtension,lax1973HyperbolicSystemsCLaws}. The associated theory, including convex entropy functions, entropy solutions, and admissibility criteria for shock waves, was subsequently developed in depth by many leading researchers, including Harten \cite{harten1983SymmetricSystemCLaws} and Mock \cite{mock1980SystemCLaws}, and later synthesized in classical monographs such as Smoller’s \cite{smoller1994ShockWavesReactionDiffusion}.
Building on these continuous results, Tadmor introduced a discrete counterpart of the entropy inequality and thereby established the modern framework for entropy-conservative and entropy-stable numerical schemes \cite{tadmor1987SStability,tadmor2003ReviewEntropyAnalysis}. Together, these seminal contributions showed that the entropy inequality plays a central role in the nonlinear stability of hyperbolic systems and their numerical discretizations, while also ensuring the physical admissibility of weak solutions \cite{smoller1994ShockWavesReactionDiffusion,dafermos2016HyperbolicCLaws}.

In the context of the compressible Navier--Stokes equations, the mathematical entropy function, $\fnc{S}$, is defined starting from the thermodynamic entropy, $s$, as \cite{hughes1986FEMEntropy}:
\begin{equation}\label{eq:entropy_function}
\fnc{S}=-\rho s,
\end{equation}
Following the results presented in \cite{friedrichs1971SystemOfConservationEquationsWithConvexExtension,lax1973HyperbolicSystemsCLaws}, a scalar function $\fnc{S}=\fnc{S}(\Q)$ is an entropy function of system \eqref{eq:compressible_ns} if:
\begin{itemize}
    \item $\fnc{S}(\Q)$, when differentiated with respect to the conservative variables, simultaneously contracts the inviscid spatial fluxes:
    \begin{equation}
        \label{eq:contraction_inviscid_flux}
        \frac{\partial \fnc{S}}{\partial \Q}\frac{\partial\Fxm}{\partial \xm}
        = \frac{\partial \fnc{S}}{\partial \Q}\frac{\partial\Fxm}{\partial \Q}\frac{\partial\Q}{\partial \xm}
        = \frac{\partial\fnc{F}_{\xm}}{\partial \Q}\frac{\partial \Q}{\partial \xm}
        = \frac{\partial\fnc{F}_{\xm}}{\partial \xm},\qquad m=1,2,3,
    \end{equation}
    where $\fnc{F}_{\xm}$ is the entropy flux, and the entropy variables are defined as $\W=\frac{\partial \fnc{S}}{\partial \Q}$.
    \item The entropy variables, $\W$, symmetrize the system of equations~\eqref{eq:compressible_ns}:
    \begin{equation*}
        \frac{\partial\Q}{\partial\W}\frac{\partial\W}{\partial t}+\sum\limits_{m=1}^{3}\frac{\partial\Fxm}{\partial\W}\frac{\partial\W}{\partial \xm} 
        = \sum\limits_{m,j=1}^{3}\frac{\partial}{\partial\xm} \left(\Cij{m}{j}\frac{\partial\W}{\partial x_j}\right), 
    \end{equation*}
    where the symmetry conditions $\frac{\partial \Q}{\partial \W} = \left( \frac{\partial \Q}{\partial \W} \right)\Tr$ and $\frac{\partial \Fxm}{\partial \W} = \left( \frac{\partial \Fxm}{\partial \W} \right)\Tr$ hold. For the definition of the symmetric positive semi-definite matrices $\Cij{m}{j}$, see \cite{parsani2015SStabilitySolidWallBC,fisher2012PhD}.
\end{itemize}
In physical coordinates, the viscous fluxes can be written in entropy-variable form as
\begin{equation*}
    \Fxmv = \sum\limits_{j=1}^{3}\Cij{m}{j} \frac{\partial \W}{\partial x_j},
    \qquad m=1,2,3.
\end{equation*}
Under the ALE mapping, the contravariant viscous flux in the reference direction $\xi_l$ is
\begin{equation*}
    \sum\limits_{m=1}^{3} \J\frac{\partial \xi_l}{\partial x_m}\Fxmv 
    = \sum\limits_{n=1}^{3} \Chatij{l}{n} \frac{\partial \W}{\partial \xi_n},
    \qquad l=1,2,3,
\end{equation*}
where
\begin{equation}
    \Chatij{l}{n} = \J \sum\limits_{m,j=1}^{3} \frac{\partial \xi_l}{\partial x_m} \Cij{m}{j} \frac{\partial \xi_n}{\partial x_j}.
    \label{eq:mapped_viscous_matrix}
\end{equation}
In other words, if
\begin{equation*}
    \mathcal{M}_{l}^{m} = \J\frac{\partial \xi_l}{\partial x_m},
\end{equation*}
then,
\begin{equation}
    \Chatij{l}{n} = \J^{-1} \sum\limits_{m,j=1}^{3} \mathcal{M}_{l}^{m} \Cij{m}{j} \mathcal{M}_{n}^{j}.
    \label{eq:mapped_viscous_matrix_metrics}
\end{equation}
The block matrix formed by the blocks $\Chatij{l}{n}$ is symmetric positive semi-definite whenever the Cartesian viscous block matrix formed by \(\Cij{m}{j}\) is symmetric positive semi-definite and $\J>0$.
In the specific case of the compressible Navier--Stokes equations, the entropy variables, $\W$, are given by
\begin{equation}
    \label{eq:EntropyVariables}
    \W = \left[\frac{\fnc{H} - \U\Tr\U}{\fnc{T}} -s,\frac{\Uone}{\fnc{T}},\frac{\Utwo}{\fnc{T}},\frac{\Uthree}{\fnc{T}},-\frac{1}{\fnc{T}}\right]\Tr.
\end{equation}
The first component of $\W$ can also be written as
\[
\frac{\fnc{H}-\U\Tr\U}{\fnc{T}}-s
=
\frac{c_{\fnc{P}}\fnc{T}-\frac12\U\Tr\U}{\fnc{T}}-s,
\]
because $\fnc{H}$ denotes the specific total enthalpy.

A sufficient condition for the convexity of the mathematical entropy $\fnc{S}(\Q)$ is $\rho>0$ and $\fnc{T}>0$; see \cite{fisher2012PhD} for the proof. Under this positivity condition, the Hessian 
\begin{equation*}
    \frac{\partial^2 \fnc{S}}{\partial \Q^2} = \frac{\partial \W}{\partial \Q},
\end{equation*}
is symmetric positive definite, i.e.,
\begin{equation*}
    \bm{v}^{\mathrm{T}} \frac{\partial^2 \fnc{S}}{\partial \Q^2} \bm{v} > 0
    \qquad \forall\bm{v} \neq 0.
\end{equation*}
Due to the symmetry of $\partial \Q/\partial \W$ and $\partial \Fxm/\partial \W$, there exist scalar functions of the entropy variables whose Jacobians yield the conservative variables $\Q$ and the inviscid fluxes $\Fxm$, namely,
\begin{equation*}
  \Q\Tr=\frac{\partial\fnc{\varPhi}}{\partial\W},\qquad
  \left(\Fxm\right)^{\mr{\top}}=\frac{\partial\fnc{\varPsi}_{\xm}}{\partial\W}.
\end{equation*}
The function $\fnc{\varPhi}$ is referred to as the potential, while the functions $\fnc{\varPsi}_{\xm}$ are the potential fluxes in the $\xm$ direction. The pair $(\fnc{\varPhi},\fnc{\varPsi}_{\xm})$ is called the potential--potential flux pair \cite{tadmor2003ReviewEntropyAnalysis}. The close relation between the entropy pair and the potential--potential flux pair is stated in the following theorem, due to Godunov \cite{godunov1961InterestingClass}; see also \cite{harten1983SymmetricSystemCLaws}.

\begin{theorem}
\label{th:Godunov}
Let \(\fnc{S}(\Q)\) be a strictly convex entropy function and let
$\W=\frac{\partial \fnc{S}}{\partial \Q}$
be the entropy variables. Define the entropy potential by
\begin{equation*}
    \fnc{\varPhi}(\W) = \W\Tr\Q(\W) - \fnc{S}(\Q(\W)).
\end{equation*}
Then,
\begin{equation*}
    \Q\Tr=\frac{\partial\fnc{\varPhi}}{\partial\W}.
\end{equation*}
Moreover, if the entropy compatibility condition holds for the inviscid fluxes, then the entropy-flux potentials
\begin{equation}
    \label{eq:GodFpotential}
    \fnc{\varPsi}_{\xm}(\W) = \W\Tr \Fxm(\Q(\W)) - \fnc{F}_{\xm}(\Q(\W))
\end{equation}
satisfy
\begin{equation*}
    \left(\Fxm\right)^{\mr{\top}}=\frac{\partial\fnc{\varPsi}_{\xm}}{\partial\W}.
\end{equation*}
\end{theorem}

For the compressible Navier--Stokes equations, the entropy--entropy flux pair $(\fnc{S},\fnc{F}_{\xm})$ is given by
\begin{equation*}
    \fnc{S} = -\rho \, s, \qquad \fnc{F}_{x_i}=-\rho \, \Ui \, s, \qquad i=1,2,3,
\end{equation*}
while the potential--potential flux pair $(\fnc{\varPhi},\fnc{\varPsi}_{x_i})$ is
\begin{equation}
    \fnc{\varPhi} = \rho \, R, \qquad \fnc{\varPsi}_{x_i} = \rho \, \Ui \, R,
    \qquad i=1,2,3.
    \label{eq:PotentialPotentialFlux}
\end{equation}
A global estimate for a strictly convex mathematical entropy provides a nonlinear measure of stability for the conservative variables. In particular, as long as the numerical solution remains in a compact subset of the physically admissible set, with density and temperature bounded away from zero, the mapping between conservative and entropy variables remains one-to-one and the Hessian of the entropy is positive definite. Uniform convexity then allows a bound on the total entropy to be related to an $L^{2}$-type bound on the conservative state. The entropy estimate should therefore be interpreted not only as a discrete analogue of the second law of thermodynamics, but also as a mechanism for controlling the growth of the numerical solution. This argument relies on the preservation of positive density and temperature, which must generally be ensured by an additional positivity-preserving mechanism \cite{dafermos2016HyperbolicCLaws}.

\begin{remark}
    The stability analysis presented in this work applies to the spatial semi-discretization. Entropy conservation or entropy stability of the spatial operator does not automatically imply the corresponding property for the fully discrete method, since the temporal approximation can introduce additional entropy production or dissipation. As noted by Tadmor \cite{tadmor2003ReviewEntropyAnalysis}, implicit schemes may contribute numerical entropy dissipation, whereas a general explicit time-integration method can produce entropy even when coupled with an entropy-conservative spatial discretization. Consequently, a fully discrete entropy estimate depends on the combined behavior of the spatial and temporal operators. Exact entropy conservation, or a prescribed entropy inequality at the fully discrete level, requires an entropy-compatible time integrator, such as a suitable space--time discretization or a relaxation Runge--Kutta method \cite{ranocha2019RelaxationRK}.
\end{remark}

\subsection{Continuous moving-wall entropy balance}
We now derive the continuous entropy balance for moving-wall boundaries. Building on the relations established above, we contract the ALE system \eqref{eq:continuous_ALE_NS} with the entropy variables, $\W$:
\begin{equation*}
    \underbrace{\W\Tr \frac{\partial \left( \J \Q \right)}{\partial \tau}}_{\textbf{TIME}} = 
    - \underbrace{\sum\limits_{l,m=1}^3 \W\Tr \frac{\partial}{\partial \xi_l} \left[ \J \frac{\partial \xi_l}{\partial \xm} \left( \Fxm - \Vm \Q \right) \right] }_{\textbf{INVISCID + ALE}}
    + \underbrace{\sum\limits_{l,m=1}^3 \W\Tr \frac{\partial}{\partial \xi_l} \left (\J \frac{\partial \xi_l}{\partial \xm} \Fxmv \right)}_{\textbf{VISCOUS}}
\end{equation*}

The contribution of the temporal derivative can be rewritten as
\begin{align*}
    \textbf{TIME}
    &= \W\Tr \Q \frac{\partial \J}{\partial \tau}
       + \frac{\partial \fnc{S}}{\partial \Q}\frac{\partial \Q}{\partial \tau}\J \\
    &= \W\Tr \Q \frac{\partial \J}{\partial \tau}
       + \frac{\partial \fnc{S}}{\partial \tau}\J \\
    &= \W\Tr \Q \frac{\partial \J}{\partial \tau}
       + \frac{\partial (\J \fnc{S})}{\partial \tau}
       - \fnc{S}\frac{\partial \J}{\partial \tau} \\
    &= \frac{\partial (\J \fnc{S})}{\partial \tau}
       + \left( \W\Tr \Q - \fnc{S} \right)\frac{\partial \J}{\partial \tau}.
\end{align*}

The convective contributions, arising from the inviscid and grid-motion terms, can be simplified using the entropy compatibility relation \eqref{eq:contraction_inviscid_flux}:
\begin{align*}
    \textbf{INVISCID}
    &= \sum\limits_{l,m=1}^3 \left[ \J \frac{\partial \xi_l}{\partial \xm} \W\Tr \frac{\partial \Fxm}{\partial \xi_l} +
    \W\Tr \Fxm \frac{\partial}{\partial \xi_l} \left( \J \frac{\partial \xi_l}{\partial \xm} \right)\right] \\
    &= \sum\limits_{l,m=1}^3 \left[ \J \frac{\partial \xi_l}{\partial \xm} \frac{\partial \fnc{F}_{\xm}}{\partial \xi_l} +
    \W\Tr \Fxm \frac{\partial}{\partial \xi_l} \left( \J \frac{\partial \xi_l}{\partial \xm} \right)\right]  \\
    &= \sum\limits_{l,m=1}^3 \left[\frac{\partial }{\partial \xi_l} \left(\J \frac{\partial \xi_l}{\partial \xm} \fnc{F}_{\xm} \right) +
    \left(\W\Tr \Fxm - \fnc{F}_{\xm} \right) \frac{\partial}{\partial \xi_l} \left( \J \frac{\partial \xi_l}{\partial \xm} \right)\right],
\end{align*}
\begin{align*}
    \textbf{ALE}
    &= \sum\limits_{l,m=1}^3 \left[\J \frac{\partial \xi_l}{\partial \xm} \Vm \frac{\partial \fnc{S}}{\partial \Q} \frac{\partial \Q}{\partial \xi_l}
    + \W\Tr \Q \frac{\partial}{\partial \xi_l} \left( \J \frac{\partial \xi_l}{\partial \xm} \Vm \right)\right] \\
    &= \sum\limits_{l,m=1}^3 \left[\frac{\partial}{\partial \xi_l} \left( \J \frac{\partial \xi_l}{\partial \xm} \Vm \fnc{S} \right)
    - \fnc{S} \frac{\partial}{\partial \xi_l} \left( \J \frac{\partial \xi_l}{\partial \xm} \Vm \right) + \W\Tr \Q \frac{\partial}{\partial \xi_l} \left( \J \frac{\partial \xi_l}{\partial \xm} \Vm \right)\right] \\
    &= \sum\limits_{l,m=1}^3 \left[\frac{\partial}{\partial \xi_l} \left( \J \frac{\partial \xi_l}{\partial \xm} \Vm \fnc{S} \right)
    + \left(\W\Tr \Q -\fnc{S} \right) \frac{\partial}{\partial \xi_l} \left( \J \frac{\partial \xi_l}{\partial \xm} \Vm \right)\right].
\end{align*}

The viscous contribution can be written as
\begin{align*}
    \textbf{VISCOUS}
    &= \sum\limits_{l,n=1}^3
    \frac{\partial}{\partial\xi_l} \left(
    \W\Tr \Chatij{l}{n} \frac{\partial \W}{\partial \xi_n} \right)
    -
    \sum\limits_{l,n=1}^3 \frac{\partial \W\Tr}{\partial \xi_l} \Chatij{l}{n} \frac{\partial \W}{\partial \xi_n}.
\end{align*}

Finally, we obtain the differential form of the scalar entropy equation:
\begin{equation}
    \label{eq:entropy_equation}
    \begin{aligned}
        \frac{\partial \left( \J \fnc{S} \right )}{\partial \tau} = &-\sum\limits_{l,m=1}^{3}  \frac{\partial}{\partial\xi_l} \left[ \J\frac{\partial \xi_l}{\partial \xm} \left( \fnc{F}_{x_m} - \Vm\fnc{S} \right)\right] 
        + \sum\limits_{l,n=1}^{3} \left[ \frac{\partial}{\partial\xi_l} \left(
        \W\Tr \Chatij{l}{n} \frac{\partial \W}{\partial \xi_n} \right)
        - \frac{\partial \W\Tr}{\partial \xi_l} \Chatij{l}{n} \frac{\partial \W}{\partial \xi_n} \right] +\\
        &- \underbrace{\left( \W\Tr \Q - \fnc{S} \right)}_{\fnc{\varPhi}} \left[ \underset{Eq.\eqref{eq:timeGCL}}{\cancel{\frac{\partial \J}{\partial \tau} - \sum\limits_{l,m=1}^{3} \frac{\partial}{\partial\xi_l} \left( \J \frac{\partial \xi_l}{\partial \xm} \Vm \right)}} \right] 
        - \sum\limits_{m=1}^{3} \underbrace{\left(\W\Tr \Fxm - \fnc{F}_{\xm} \right)}_{\fnc{\varPsi}_{\xm}} \underset{Eq.\eqref{eq:spaceGCL}}{\cancel{\sum_{l=1}^3 \frac{\partial}{\partial \xi_l} \left (\J \frac{\partial \xi_l}{\partial \xm} \right)}}.
    \end{aligned}
\end{equation}
The last two terms vanish when the time and space metric identities \eqref{eq:timeGCL} and \eqref{eq:spaceGCL} are satisfied.

Integrating equation \eqref{eq:entropy_equation} over the domain $\Omega$ yields the global entropy conservation statement:
\begin{equation}
    \label{eq:continuous_entropy_estimate_discont}
    \begin{split}
        \int_{\Omega}\frac{\partial \left( \J\fnc{S} \right)}{\partial \tau}\mr{d}\Omega 
        = \frac{\mr{d}}{\mr{d}\tau}\int_{\Omega}\J\fnc{S}\mr{d}\Omega
        \leq \sum\limits_{l,m=1}^{3}\oint_{\Gamma}\left(\W\Tr \Fxmv - \J \frac{\partial \xi_l}{\partial \xm} \left( \fnc{F}_{\xm} -\Vm \fnc{S} \right) \right)\nxil\mr{d}\Gamma - DT,
    \end{split}
\end{equation}
where $\nxil$ is the $l$-th component of the outward-facing unit normal and
\[
DT = \sum\limits_{l,n=1}^{3}\int\limits_{\hat{\Omega}}\left(\frac{\partial \W}{\partial \xi_l}\right)^{\top} \Chatij{l}{n} \,
  \frac{\partial \W}{\partial \xi_n} \mr{d}\hat{\Omega} \ge 0.
\]

We note that viscous dissipation contributes a non-positive term to the evolution of the mathematical entropy, since $DT\ge 0$. Any increase of the mathematical entropy functional can therefore only arise from boundary contributions. For smooth flows, the inequality in \eqref{eq:continuous_entropy_estimate_discont} reduces to an equality.

We now isolate the boundary contribution for a moving reference element. For clarity, we write the boundary contribution on the reference unit cube, with grid velocity $\Vg$, as illustrated in Fig.~\ref{fig:cubeface}. Then, equation \eqref{eq:continuous_entropy_estimate_discont} gives
\begin{equation}
\label{eq:continuous_entropy_estimate_moving}
    \begin{aligned}
        & \frac{\mr{d}}{\mr{d}t} \int_{\hat{\Omega}} \J \fnc{S} \, \mr{d}\xi_1 \, \mr{d}\xi_2 \, \mr{d}\xi_3 =  -DT  \\
        & \:+\: \int_{\xi_1=0} \sum\limits_{m=1}^3 \left[
        + \J \frac{\partial \xi_1}{\partial \xm} \left( \Fxisca{m} - \Vm \, \fnc{S} \right) 
        - \, \W^{\top}\Cij{1}{m} \frac{\partial \W}{\partial \xi_m}
        \right] \mr{d}\xi_2 \, \mr{d}\xi_3 \\
        & \:+\: \int_{\xi_1=1} \sum\limits_{m=1}^3 \left[
        - \J \frac{\partial \xi_1}{\partial \xm} \left( \Fxisca{m} - \Vm \, \fnc{S} \right) 
        + \, \W^{\top}\Cij{1}{m} \frac{\partial \W}{\partial \xi_m}
        \right] \mr{d}\xi_2 \, \mr{d}\xi_3 \\
        & \:+\: \int_{\xi_2=0} \sum\limits_{m=1}^3 \left[
        + \J \frac{\partial \xi_2}{\partial \xm} \left( \Fxisca{m} - \Vm \, \fnc{S} \right) 
        - \, \W^{\top}\Cij{2}{m} \frac{\partial \W}{\partial \xi_m}
        \right] \mr{d}\xi_1 \, \mr{d}\xi_3 \\
        & \:+\: \int_{\xi_2=1} \sum\limits_{m=1}^3 \left[
        - \J \frac{\partial \xi_2}{\partial \xm} \left( \Fxisca{m} - \Vm \, \fnc{S} \right) 
        + \, \W^{\top}\Cij{2}{m} \frac{\partial \W}{\partial \xi_m}
        \right] \mr{d}\xi_1 \, \mr{d}\xi_3 \\
        & \:+\: \int_{\xi_3=0} \sum\limits_{m=1}^3 \left[
        + \J \frac{\partial \xi_3}{\partial \xm} \left( \Fxisca{m} - \Vm \, \fnc{S} \right) 
        - \, \W^{\top}\Cij{3}{m} \frac{\partial \W}{\partial \xi_m}
        \right] \mr{d}\xi_1 \, \mr{d}\xi_2 \\
        & \:+\: \int_{\xi_3=1} \sum\limits_{m=1}^3 \left[
        - \J \frac{\partial \xi_3}{\partial \xm} \left( \Fxisca{m} - \Vm \, \fnc{S} \right) 
        + \, \W^{\top}\Cij{3}{m} \frac{\partial \W}{\partial \xi_m}
        \right] \mr{d}\xi_1 \, \mr{d}\xi_2
    \end{aligned}
\end{equation}
The plus and minus signs in \eqref{eq:continuous_entropy_estimate_moving} reflect the orientation of the outward-facing normals on the six faces of the unit cube $\hat{\Omega}$.

\begin{figure}[H]
    \centering
    \includegraphics[width=0.3\linewidth]{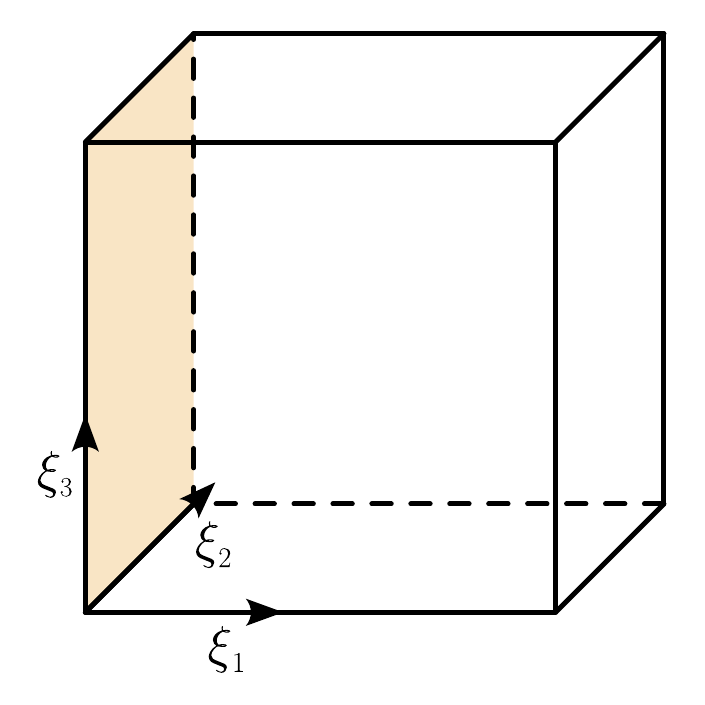}
    \caption{Representation of the reference element domain. We consider the face $\xi_1=0$ in the following proofs.}
    \label{fig:cubeface}
\end{figure}

To isolate the moving-wall contribution, we consider the reference face \(\xi_1=0\) and assume that the contributions from the remaining faces are entropy conservative. Then \eqref{eq:continuous_entropy_estimate_moving} reduces to
\begin{equation}
    \label{eq:continuous_entropy_estimate_1wall_moving}
    \frac{\mr{d}}{\mr{d}t} \int_{\Omega} \J \fnc{S} \, \mr{d}\xi_1 \, \mr{d}\xi_2 \, \mr{d}\xi_3 =  -DT
    \:+\: \int_{\xi_1=0} \sum\limits_{m=1}^3 \left[ \J \frac{\partial \xi_1}{\partial x_m} \left( \Fxisca{m} - \Vgi{m} \, \fnc{S} \right) 
    - \, \W^{\top} \left(\Cij{1}{m} \frac{\partial \W}{\partial \xi_m} \right)\right] \mr{d}\xi_2 \, \mr{d}\xi_3.
\end{equation}

Now, we identify wall conditions that control the boundary contribution in \eqref{eq:continuous_entropy_estimate_1wall_moving}. In particular, an entropy-conservative moving wall should make the ALE-inviscid wall contribution vanish.
\begin{theorem}
\label{th:continuous_noslip}
The ALE no-slip boundary conditions 
\begin{equation*}
    \Ui = \Vi, \qquad i=1,2,3,
\end{equation*}
bound the ALE-inviscid flux contribution to the time derivative of the entropy in equation \eqref{eq:continuous_entropy_estimate_1wall_moving}. The weaker moving no-penetration condition is sufficient for the inviscid contribution.
\end{theorem}
\begin{proof}
Substituting the entropy--entropy flux relation \eqref{eq:GodFpotential} into \eqref{eq:continuous_entropy_estimate_1wall_moving} yields
\begin{equation*}
    \begin{aligned}
        \fnc{F}_{\xm} &= \W\Tr \Fxm - \fnc{\varPsi}_{\xm} = -\rho \,\fnc{U}_m s + \rho R \,\fnc{U}_m - \rho R \,\fnc{U}_m = -\rho \,\fnc{U}_m s \\
        \fnc{S} &= - \rho s.
    \end{aligned}
\end{equation*}
Hence,
\begin{equation*}
    \frac{\mr{d}}{\mr{d}t} \int_{\Omega} \J \fnc{S} \, \mr{d}\xi_1 \, \mr{d}\xi_2 \, \mr{d}\xi_3 =  -DT
    \:+\: \int_{\xi_1=0} \sum\limits_{m=1}^3 - \J \frac{\partial \xi_1}{\partial x_m} \rho \left( \fnc{U}_m - \Vgi{m} \right) \, s 
    - \, \W^{\top} \left(\Cij{1}{m} \frac{\partial \W}{\partial \xi_m} \right) \mr{d}\xi_2 \, \mr{d}\xi_3.
\end{equation*}
When the no-slip boundary conditions $\fnc{U}_m = \Vgi{m}$, for $m=1,2,3$, are imposed, the entropy flux associated with the ALE inviscid contribution vanishes:
\[
\fnc{F}_{\xm} - \Vm \, S = - \rho \, \left( \fnc{U}_m - \Vm \right) s = 0.
\]
\end{proof}

The no-slip condition is stronger than what is required by the ALE-inviscid entropy contribution. The inviscid wall contribution vanishes if
\begin{equation*}
    \sum\limits_{m=1}^{3}
    \J\frac{\partial \xi_1}{\partial x_m}
    \left( \fnc{U}_m-\Vgi{m} \right) = 0
    \qquad \text{on} \quad \xi_1=0.
\end{equation*}
This is the moving no-penetration condition written on the reference face.

\begin{theorem}\label{th:continous_viscous_bound}
The viscous wall contribution is controlled by prescribing the heat entropy flow through the wall. With the viscous flux convention in \eqref{eq:Fv} and an outward physical unit normal $\bm n$,
\begin{equation*}
    \W^{\top} \ensuremath{\bm{\fnc{F}}^{(V)}} \cdot\bm n
    = -\frac{\kappa}{\fnc{T}} \frac{\partial \fnc{T}}{\partial n},
\end{equation*}
where
\begin{equation*}
    \frac{\partial \fnc{T}}{\partial n}
    = \nabla \fnc{T}\cdot\bm n.
\end{equation*}
In other words, if one defines the prescribed heat entropy flow by
\begin{equation*}
    \mathtt{g}(t)
    = -\W^{\top} \ensuremath{\bm{\fnc{F}}^{(V)}} \cdot\bm n,
\end{equation*}
then
\begin{equation*}
    \mathtt{g}(t)
    = \frac{\kappa}{\fnc{T}} \frac{\partial \fnc{T}}{\partial n}.
\end{equation*}
This boundary datum controls the viscous contribution to the entropy balance \eqref{eq:continuous_entropy_estimate_1wall_moving}.
\end{theorem}
\begin{proof}
    See Theorem 3.2 in \cite{parsani2015SStabilitySolidWallBC} or Theorem 2.3 in \cite{dalcin2019WallBC}.
\end{proof}

\section{\textbf{Semi-discrete equations}}
\label{sec:semi-discrete}

This section introduces the summation-by-parts operators used in the semi-discrete ALE compressible Navier--Stokes discretization \eqref{eq:SemidiscreteALENavierStokes}. Alternative derivations based on nodal discontinuous Galerkin formulations can be found, for example, in \cite{krais2020SplitFormALE,schnucke2020EntropyStableALE}. Here, we follow the diagonal-norm SBP framework and the associated flux-differencing formulation of \cite{carpenter2014SSDC,parsani2015SStabilitySolidWallBC,carpenter2016SStabilitySSDC,fernandez2020SStabilityPRefinement,yamaleev2019EntropyStableALE}.

At each time $\tau$, the physical domain $\Omega(\tau)$ is partitioned into non-overlapping hexahedral elements $\Omega_{\kappa}(\tau)$. Each physical element is mapped to the fixed reference element $\widehat{\Omega} = [-1,1]^{3}$, with computational coordinates $\boldsymbol{\xi}=(\xi_{1},\xi_{2},\xi_{3})^{T}$. The nodal distribution consists of $N$ Legendre--Gauss--Lobatto points in each computational direction, giving $N^{3}$ volume nodes per element. The corresponding polynomial degree is $p=N-1$.

\subsection{SBP operators}
Let $\mathbf{D}_{N}$ denote the one-dimensional differentiation matrix. A diagonal-norm SBP operator satisfies
\begin{equation*}
    \mathbf{D}_{N} = \mathbf{P}_{N}^{-1}\mathbf{Q}_{N},
    \qquad \mathbf{Q}_{N} + \mathbf{Q}_{N}^{T} = \mathbf{B}_{N},
\end{equation*}
where $\mathbf{P}_{N}$ is a diagonal positive-definite quadrature matrix and
\begin{equation*}
    \mathbf{B}_{N} = \mathbf{e}_{N}\mathbf{e}_{N}^{T} - \mathbf{e}_{1}\mathbf{e}_{1}^{T} = \operatorname{diag} \left( -1,0,\ldots,0,1 \right).
\end{equation*}
Here, $\mathbf{e}_{1}\mathbf{e}_{1}^{T}$ and $\mathbf{e}_{N}\mathbf{e}_{N}^{T}$ extract the left and right endpoints of the one-dimensional reference interval.
Tensor products construct the scalar three-dimensional volume quadrature matrix as
\begin{equation*}
    \widehat{\mathbf{P}} := \mathbf{P}_{N} \otimes \mathbf{P}_{N} \otimes \mathbf{P}_{N}.
\end{equation*}
The scalar differentiation matrices in the three computational directions are
\begin{equation*}
        \widehat{\mathbf{D}}_{\xi_{1}} := \mathbf{D}_{N} \otimes \mathbf{I}_{N} \otimes \mathbf{I}_{N}, \qquad
        \widehat{\mathbf{D}}_{\xi_{2}} := \mathbf{I}_{N} \otimes \mathbf{D}_{N} \otimes \mathbf{I}_{N}, \qquad
        \widehat{\mathbf{D}}_{\xi_{3}} := \mathbf{I}_{N} \otimes \mathbf{I}_{N} \otimes \mathbf{D}_{N}.
\end{equation*}
The corresponding multidimensional undivided-difference matrices are defined by
\begin{equation*}
    \widehat{\mathbf{Q}}_{\xi_l} := \widehat{\mathbf{P}}\, \widehat{\mathbf{D}}_{\xi_l}, \qquad l=1,2,3,
\end{equation*}
and satisfy
\begin{equation*}
    \widehat{\mathbf{Q}}_{\xi_l} + \widehat{\mathbf{Q}}_{\xi_l}^{T} = \widehat{\mathbf{B}}_{\xi_l}.
\end{equation*}
Explicitly,
\begin{equation*}
        \widehat{\mathbf{B}}_{\xi_{1}} = \mathbf{B}_{N} \otimes \mathbf{P}_{N} \otimes \mathbf{P}_{N}, \qquad
        \widehat{\mathbf{B}}_{\xi_{2}} = \mathbf{P}_{N} \otimes \mathbf{B}_{N} \otimes \mathbf{P}_{N}, \qquad
        \widehat{\mathbf{B}}_{\xi_{3}} = \mathbf{P}_{N} \otimes \mathbf{P}_{N} \otimes \mathbf{B}_{N}.
\end{equation*}

For the five-component compressible Navier--Stokes system, the scalar operators are extended componentwise through tensor product multiplication:
\begin{equation*}
    \begin{alignedat}{2}
        \mathbf{P}_{\kappa}
        &:= \widehat{\mathbf{P}} \otimes \mathbf{I}_{5}, &\qquad
        \mathbf{D}_{\xi_l}
        &:= \widehat{\mathbf{D}}_{\xi_l} \otimes \mathbf{I}_{5}, \\
        \mathbf{Q}_{\xi_l}
        &:= \widehat{\mathbf{Q}}_{\xi_l} \otimes \mathbf{I}_{5}, &\qquad
        \mathbf{B}_{\xi_l}
        &:= \widehat{\mathbf{B}}_{\xi_l} \otimes \mathbf{I}_{5},
        \qquad l=1,2,3.
    \end{alignedat}
\end{equation*}
The subscript $\kappa$ on the norm matrix is retained to indicate its association with element $\kappa$, although the reference element quadrature matrix is identical for all elements with the same polynomial degree.

Let $\widehat{\mathbf{R}}_{f}$ denote the scalar restriction operator from the volume nodes to the nodes on face $f$. For example, the restriction operators associated with the two faces orthogonal to $\xi_{1}$ are
\begin{equation*}
    \widehat{\mathbf{R}}_{\xi_{1}^{-}} := \mathbf{e}_{1}^{T} \otimes \mathbf{I}_{N} \otimes \mathbf{I}_{N}, \qquad
    \widehat{\mathbf{R}}_{\xi_{1}^{+}} := \mathbf{e}_{N}^{T} \otimes \mathbf{I}_{N} \otimes \mathbf{I}_{N}.
\end{equation*}
The restriction operators in the remaining coordinate directions are defined analogously. The scalar face norm is
\begin{equation*}
    \widehat{\mathbf{P}}_{f} := \mathbf{P}_{N} \otimes \mathbf{P}_{N},
\end{equation*}
while the corresponding system operators are
\begin{equation*}
    \mathbf{R}_{f} := \widehat{\mathbf{R}}_{f} \otimes \mathbf{I}_{5}, \qquad \mathbf{P}_{f} := \widehat{\mathbf{P}}_{f} \otimes \mathbf{I}_{5}.
\end{equation*}

The boundary matrix in direction $\xi_l$ can therefore be written in terms of the face restriction operators as
\begin{equation*}
    \mathbf{B}_{\xi_l} = \sum_{\substack{f\subset\partial\widehat{\Omega}\\l(f)=l}} n_{\xi_l,f}\, \mathbf{R}_{f}^{T} \mathbf{P}_{f} \mathbf{R}_{f}, \qquad n_{\xi_l,f}\in\{-1,+1\}.
\end{equation*}
Consequently, for arbitrary nodal vectors $\mathbf{u}$ and $\mathbf{v}$, the multidimensional SBP identity is
\begin{equation}
    \mathbf{v}^{T} \mathbf{P}_{\kappa} \mathbf{D}_{\xi_l}\mathbf{u} + \left( \mathbf{D}_{\xi_l}\mathbf{v} \right)^{T} \mathbf{P}_{\kappa}\mathbf{u} = \sum_{\substack{f\subset\partial\widehat{\Omega}\\l(f)=l}} n_{\xi_l,f} \left( \mathbf{R}_{f}\mathbf{v} \right)^{T} \mathbf{P}_{f} \left( \mathbf{R}_{f}\mathbf{u} \right).
    \label{eq:MultidimensionalDiscreteIntegrationByParts}
\end{equation}
Equation \eqref{eq:MultidimensionalDiscreteIntegrationByParts} is the discrete analog of multidimensional integration by parts and is the fundamental identity used in the semi-discrete entropy analysis.

Throughout the analysis, the $\widetilde{\;}$ symbol indicates that the quantity is evaluated in the reference element and contracted with metric coefficients.

Let $N_v=N^{3}$ and $N_f=N^{2}$ denote the number of volume and face nodes, for a single cell and face, respectively.
The state vector is defined node-wise as $\mathbf{q}_{\kappa}\in\mathbb{R}^{5N_v}$, while $\mathbf{J}_{\kappa}=\operatorname{diag}(\mathbf{j}_{\kappa})\otimes\mathbf{I}_{5}$ and $\mathbf{P}_{\kappa}=\widehat{\mathbf{P}}_{\kappa}\otimes\mathbf{I}_{5}$ are diagonal matrices in $\mathbb{R}^{5N_v\times 5N_v}$.
The symbol $\circ$ denotes the Hadamard product (i.e. the elementwise product between two size-compatible matrices).

For any scalar nodal quantity $a$, the arithmetic two-point average and the logarithmic average are defined respectively by
\begin{equation*}
    \avg{a}_{ij}:=\frac{a_i+a_j}{2}, \qquad
    \logavg{a}_{ij}:=\frac{a_i-a_j}{\log{a_i}-\log{a_j}},
\end{equation*}
both with the same definition applied componentwise to vector quantities.

For a face $f$, $\mathbf{R}_{f}\in\mathbb{R}^{5N_f\times5N_v}$ restricts volume data to the face, $\mathbf{P}_{f}\in\mathbb{R}^{5N_f\times5N_f}$ is the face quadrature matrix, and each penalty vector $\mathbf{g}_{f}^{(\cdot),q}\in\mathbb{R}^{5N_f}$ contains the corresponding interface or physical-boundary contribution.

\subsection{Semi-discrete compressible Navier--Stokes equations}
We can write the compressible Navier--Stokes equations in the ALE formulation for a single tensor-product cell $\kappa$ as:

\begin{equation}
    \begin{aligned}
        \frac{\mathrm{d}}{\mathrm{d}\tau} \left( \mathbf{J}_{\kappa}\mathbf{q}_{\kappa} \right) + \sum_{l=1}^{3} \left[ 2 \mathbf{D}_{\xi_l} \circ \widetilde{\mathbf{F}}_{l,\kappa}^{\,\mathrm{ALE,ec}} \,\mathbf{1} - \mathbf{D}_{\xi_l} \widetilde{\mathbf{f}}_{l,\kappa}^{\,\mathrm{V}} \right] &= \mathbf{P}_{\kappa}^{-1} \sum_{f\in\Gamma_{\kappa}^{\mathrm{I}}} \mathbf{R}_{f}^{T}\mathbf{P}_{f} \left[ \mathbf{g}_{f}^{(\mathrm{I,ALE}),q} + \mathbf{g}_{f}^{(\mathrm{I,V}),q} \right] \\
        &\quad+ \mathbf{P}_{\kappa}^{-1} \sum_{f\in\Gamma_{\kappa}^{\mathrm{B}}} \mathbf{R}_{f}^{T}\mathbf{P}_{f} \left[ \mathbf{g}_{f}^{(\mathrm{B,ALE}),q} + \mathbf{g}_{f}^{(\mathrm{B,V}),q} \right], \\
    \end{aligned}
    \label{eq:SemidiscreteALENavierStokes}
\end{equation}

where the discrete volume fluxes are defined as:

\begin{equation*}
    \left[ \widetilde{\mathbf{F}}_{l,\kappa}^{\,\mathrm{ALE,ec}} \right]_{ij} := \operatorname{diag} \left[ \sum_{m=1}^{3} \avg{J\frac{\partial\xi_l}{\partial x_m}}_{ij} \mathbf{F}_{x_m}^{\,\mathrm{ALE,ec}} \left( \mathbf{Q}_{\kappa,i}, \mathbf{Q}_{\kappa,j}; \avg{V_m}_{ij} \right) \right].
\end{equation*}

The ALE convective operator is written in nonlinear flux-differencing form,
\begin{equation*}
    2 \mathbf{D}_{\xi_l} \circ \widetilde{\mathbf{F}}_{l,\kappa}^{\,\mathrm{ALE,ec}} \mathbf{1},
\end{equation*}
where the block flux matrix is assembled from a symmetric and consistent two-point entropy-conservative ALE flux. The viscous divergence is approximated using the standard SBP derivative,
\begin{equation*}
    \mathbf{D}_{\xi_l} \widetilde{\mathbf{f}}_{l,\kappa}^{\,\mathrm{V}},
\end{equation*}
The differentiation operators $\mathbf{D}_{\xi_l}\in\mathbb{R}^{5N_v\times 5N_v}$ act componentwise, $\widetilde{\mathbf{f}}_{l,\kappa}^{\,\mathrm{V}}\in\mathbb{R}^{5N_v}$ contains the viscous flux evaluated at the volume nodes, and $\mathbf{1}\in\mathbb{R}^{5N_v}$ is the vector of ones. The nonlinear flux matrix $\widetilde{\mathbf{F}}_{l,\kappa}^{\,\mathrm{ALE,ec}}\in\mathbb{R}^{5N_v\times 5N_v}$ is composed of $5\times5$ diagonal blocks, with the $(i,j)$ block constructed from the two-point flux evaluated using the states at volume nodes $i$ and $j$.

Interface and physical boundary conditions are imposed weakly through SAT contributions of the form
\begin{equation*}
    \mathbf{P}_{\kappa}^{-1} \mathbf{R}_{f}^{T} \mathbf{P}_{f} \mathbf{g}_{f}.
\end{equation*}

where $\mathbf{g}_{f}$ is a surface term determined by either the neighboring cell or the boundary state as:

\begin{equation*}
    \mathbf{g}_{f}^{(\mathrm{I,ALE}),q} = \widetilde{\mathbf{f}}_{f,\kappa}^{\,\mathrm{ALE},-} - \widetilde{\mathbf{f}}_{f}^{\,\mathrm{ALE,ec}} \left( \mathbf{q}_{f}^{-}, \mathbf{q}_{f}^{+} \right)
\end{equation*}

\begin{equation}
    \mathbf{g}_{f}^{(\mathrm{B,ALE}),q} = \widetilde{\mathbf{f}}_{f,\kappa}^{\,\mathrm{ALE},-} - \widetilde{\mathbf{f}}_{f}^{\,\mathrm{ALE,bc}} \left( \mathbf{q}_{f}^{-}, \mathbf{q}_{f}^{B} \right)
    \label{eq:SATboundaryALE}
\end{equation}

here, we define appropriate numerical flux and entropy-conservative 2-point fluxes,

\begin{equation*}
    \widetilde{\mathbf{f}}_{f,\kappa}^{\,\mathrm{ALE},-} := \sum_{m=1}^{3} n_{\xi_{l(f)},f}\, \operatorname{diag} \left[ \mathbf{R}_{f} \left( J\frac{\partial\xi_{l(f)}}{\partial x_m} \right)_{\kappa} \right] \left[ \mathbf{F}_{x_m}\left(\mathbf{q}_{f}^{-}\right) - V_{m,f}\mathbf{q}_{f}^{-} \right]
\end{equation*}

\begin{equation}
    \widetilde{\mathbf{f}}_{f}^{\,\mathrm{ALE,ec}} \left( \mathbf{q}_{f}^{-}, \mathbf{q}_{f}^{+} \right) := \sum_{m=1}^{3} n_{\xi_{l(f)},f}\, \operatorname{diag} \left[ \mathbf{R}_{f} \left( J\frac{\partial\xi_{l(f)}}{\partial x_m} \right)_{\kappa} \right] \mathbf{F}_{x_m}^{\,\mathrm{ALE,ec}} \left( \mathbf{q}_{f}^{-}, \mathbf{q}_{f}^{+}; \avg{V_m}_{f} \right)
    \label{eq:SurfaceALE2pointFlux}
\end{equation}

where $l(f)$ denotes the computational direction orthogonal to face $f$, and $n_{\xi_{l(f)},f}\in\{-1,+1\}$ is the outward reference-space normal component, and the corresponding physical scaled normal is $n_{\xi_{l(f)},f}J\nabla_x\xi_{l(f)}$.

It has been proved that this formulation allows the volume contributions to telescope to surface terms under contraction with the entropy variables, thereby reducing the current entropy analysis to interface and boundary-point relations.
The linear SBP identity \eqref{eq:MultidimensionalDiscreteIntegrationByParts} is used to integrate the discrete viscous divergence by parts. It separates the viscous contribution into a surface flux and a non-negative volume dissipation term. The nonlinear ALE convective flux-differencing operator does not satisfy the linear identity directly; its entropy analysis instead relies on the nonlinear SBP property, the symmetry and consistency of the two-point flux, the Tadmor shuffle condition, and the discrete geometric conservation laws. Together, these properties reduce the contracted volume operators to face contributions that can be combined with the interface and boundary SATs. Yamaleev \cite{yamaleev2019EntropyStableALE} shows a detailed derivation of the equations by applying these properties.

The viscous contribution is discretized with a local discontinuous Galerkin (LDG) type of approach. This specific discretization strategy has been successfully applied in many previous works \cite{carpenter2014SSDC,carpenter2015SStableStaggered,parsani2015DiscontinuousInterface,yamaleev2019EntropyStableALE}.

\begin{equation}
    \boldsymbol{\Theta}_{\xi_l,\kappa} - \mathbf{D}_{\xi_l}\mathbf{w}_{\kappa} = \mathbf{P}_{\kappa}^{-1} \sum_{\substack{f\in\Gamma_{\kappa}^{\mathrm{I}}\\l(f)=l}} \mathbf{R}_{f}^{T}\mathbf{P}_{f} \mathbf{g}_{f}^{(\mathrm{I}),\Theta} + \mathbf{P}_{\kappa}^{-1} \sum_{\substack{f\in\Gamma_{\kappa}^{\mathrm{B}}\\l(f)=l}} \mathbf{R}_{f}^{T}\mathbf{P}_{f} \mathbf{g}_{f}^{(\mathrm{B}),\Theta}, \qquad l=1,2,3.
    \label{eq:ViscousLDG}
\end{equation}

\begin{equation}
    \widetilde{\mathbf{f}}_{l,\kappa}^{\,\mathrm{V}} := \sum_{n=1}^{3} \widehat{\mathbf{C}}_{ln,\kappa} \boldsymbol{\Theta}_{\xi_n,\kappa}
    \label{eq:DiscreteViscousFlux}
\end{equation}

\begin{equation*}
    \widehat{\mathbf{C}}_{ln,\kappa} := \mathbf{J}_{\kappa}^{-1} \sum_{m=1}^{3}\sum_{j=1}^{3} \operatorname{diag} \left[ \left( J\frac{\partial\xi_l}{\partial x_m} \right)_{\kappa} \right] \mathbf{C}_{mj,\kappa} \operatorname{diag} \left[ \left( J\frac{\partial\xi_n}{\partial x_j} \right)_{\kappa} \right]
\end{equation*}

The discrete counterpart of the geometric conservation laws applies to space and time in the following form. As previously stated, these relations are crucial to ensure water-tightness (i.e. freestream preservation) and entropy conservation for moving and deforming elements. The discrete space-GCL \eqref{eq:DiscreteSpaceGCL} is enforced by building the mapping with the same SBP operators that are used for computing derivatives. This is equivalent to the non-ALE implementation. On the other hand, the discrete time-GCL is enforced by advancing in time the equation for each deforming cell, together with the system of equations.

\begin{equation}
    \sum_{l=1}^{3} \mathbf{D}_{\xi_l} \left( J\frac{\partial\xi_l}{\partial x_m} \right)_{\kappa} = \mathbf{0}, \qquad m=1,2,3.
    \label{eq:DiscreteSpaceGCL}
\end{equation}

\begin{equation}
    \frac{\mathrm{d}}{\mathrm{d}\tau} \left( \mathbf{J}_{\kappa}\mathbf{1} \right) - 2\sum_{l=1}^{3} \left[ \mathbf{D}_{\xi_l} \circ \left( \left[ \sum_{m=1}^{3} \avg{J\frac{\partial\xi_l}{\partial x_m}}_{ij} \avg{V_m}_{ij} \right] \right) \right] \mathbf{1} = \mathbf{0}.
    \label{eq:DiscreteTimeGCL}
\end{equation}

A fundamental piece for achieving entropy conservation is an appropriate 2-point flux function, that is symmetric, consistent, and satisfies Tadmor's shuffle conditions for both the inviscid and the ALE fluxes.
Considering the Navier--Stokes equations, a possible choice for a 2-point flux function is the Chandrasekhar flux modified by the convective relative velocity due to grid motion. In general, any entropy-conservative 2-point flux that satisfies the aforementioned properties can be chosen with appropriate modifications to the convection part.
In the following, we will consider:

\begin{equation}
    \mathbf{F}_{x_m}^{\,\mathrm{ALE,ec}} = \mathbf{F}_{x_m}^{\,\mathrm{ec}} - \avg{V_m}\mathbf{U}^{*} =
    \begin{bmatrix}
        \logavg{\rho} \left( \avg{\Q_{m+1}} -\avg{V_m} \right) \\[0.4em]
        \logavg{\rho} \left( \avg{\Q_{m+1}} -\avg{V_m} \right) \avg{\Q_{2}} + \dfrac{\avg{\rho}}{2\avg{\beta}} \delta_{m,1} \\
        \logavg{\rho} \left( \avg{\Q_{m+1}} -\avg{V_m} \right) \avg{\Q_{3}} + \dfrac{\avg{\rho}}{2\avg{\beta}} \delta_{m,2} \\
        \logavg{\rho} \left( \avg{\Q_{m+1}} -\avg{V_m} \right) \avg{\Q_{4}} + \dfrac{\avg{\rho}}{2\avg{\beta}} \delta_{m,3} \\
        \logavg{\rho} \left( \avg{\Q_{m+1}} -\avg{V_m} \right) \left( \dfrac{1}{2(\gamma-1)\logavg{\beta}} + \overline{k} \right) + \dfrac{\avg{\rho}\avg{\Q_{m+1}}}{2\avg{\beta}}
    \end{bmatrix}
\label{eq:2pointFluxCH}
\end{equation}

where the kinetic energy is defined as

\begin{equation*}
    \overline{k} := \dfrac{1}{2} \sum_{r=1}^{3} \left(2\avg{\Q_{r+1}}^{2} - \avg{\Q_{r+1}^{2}}\right)
\end{equation*}


\subsection{Entropy-stable interface flux dissipation}

It has been demonstrated that an interface flux constructed using \eqref{eq:SurfaceALE2pointFlux} and adhering to the definition in \eqref{eq:2pointFluxCH} is entropy conservative. In the current discontinuous Galerkin (DG) formulation, adjacent elements are inherently discontinuous. Introducing interface dissipation is a common practice aimed at maintaining stability in regions with steep gradients. This ensures physically meaningful solutions, but relaxes entropy conservation to entropy stability.

Merriam~\cite{merriam1989RemarkableIdentity,merriam1989TowardsaRigorousApproachtoArtificialDissipation} proposed an approach that is consistent with the second law of thermodynamics. This result was later generalized by Barth~\cite{barth1999SymmetrizationofCLaws} and is known as the eigenvector scaling theorem.
In this context, we will apply a similar approach that incorporates the contribution of grid motion to the inviscid fluxes.

The entropy-stable numerical fluxes are defined as:
\begin{equation*}
    \mathbf{F}_{x_m}^{\,\mathrm{ALE,es}} \left( \mathbf{q}_{f}^{-}, \mathbf{q}_{f}^{+}; \avg{V_m}_{f} \right) = \mathbf{F}_{x_m}^{\,\mathrm{ALE,ec}} \left( \mathbf{q}_{f}^{-}, \mathbf{q}_{f}^{+}; \avg{V_m}_{f} \right) + \frac{1}{2} \mathbf{Y}_m \lvert \mathbf{\Lambda}_m \rvert \mathbf{Y}\Tr_m \left( \mathbf{w}_{f}^{-} - \mathbf{w}_{f}^{+} \right)
\end{equation*}
where $\mathbf{\Lambda}_m$ and $\mathbf{Y}_m$ are the diagonal matrix of eigenvalues and the matrix of the right eigenvectors, following these relations:
\begin{align*}
    \frac{\partial \Q}{\partial \W} &= \mathbf{Y}_m\mathbf{Y}\Tr_m , \\
    \frac{\partial \FxmALE}{\partial \Q} &= \mathbf{Y}_m \lvert \mathbf{\Lambda}_m \rvert \mathbf{Y}\Tr_m \qquad m=1,2,3.
\end{align*}
The right eigenvector matrix is the flux Jacobian of the inviscid fluxes, and it depends on the Roe-averaged states of conservative variables. It can be proved that the eigenvectors do not explicitly depend on the grid velocity contribution. The corresponding eigenvalues are modified with the grid velocity, giving an effective dissipation different than the conventional fixed grid equation.
\begin{equation*}
    \lvert \mathbf{\Lambda}_m \rvert = \operatorname{diag} \left( \lvert \avg{\U_m - \Vm} - c \rvert,\, \lvert \avg{\U_m - \Vm} \rvert,\, \lvert \avg{\U_m - \Vm} \rvert,\, \lvert \avg{\U_m - \Vm} \rvert,\, \lvert \avg{\U_m - \Vm} + c \rvert \right)
\end{equation*}
The proof of entropy stability for these interface operators can be found in \cite{krais2020SplitFormALE,schnucke2020EntropyStableALE,winters2017Upwind,yamaleev2019EntropyStableALE}.

\section{Entropy-conservative and entropy-stable moving-wall boundary conditions}
\label{sec:moving_wall_bc}

This section develops entropy-conservative and entropy-stable moving-wall boundary conditions, building on the operators and conventions established in the previous section. The entropy analysis adheres to the standard entropy-variable framework for conservative discretizations. Its extension to moving meshes follows some of the steps shown in the continuous analysis. However, additional important details are necessary to maintain the same entropy bounds at the semi-discrete level.
Krais, Schn\"ucke, and Yamaleev \cite{krais2020SplitFormALE,schnucke2020EntropyStableALE,yamaleev2019EntropyStableALE} established entropy conservation and stability for entropy-consistent ALE discretizations, including the treatment of non-periodic boundaries.
For the volume terms, the presence of mesh motion does not alter the fundamental structure of the entropy argument.
Provided that the discrete temporal geometric conservation law (GCL) and the discrete spatial metric identities are satisfied, contraction with the entropy variables, together with the SBP property and the entropy-conservative two-point ALE flux, reduces the discrete volume contribution to surface terms.
In particular, the additional terms generated by the grid velocity cancel through the discrete GCL, so that no spurious volume entropy production is introduced by the mesh motion. 
For brevity, we omit this standard volume derivation here; a symbolic verification of the complete volume calculation is provided in ~\ref{app:SymbolicVolumeContribution}.

Consequently, once the volume terms have been reduced and the interior-face contributions cancel, the entropy balance is governed solely by the physical boundary terms.
The focus of the present work is therefore the treatment of a moving solid wall.
In contrast to a stationary-wall analysis, the ALE boundary flux contains the normal grid velocity explicitly, and the impermeability condition must be imposed in terms of the velocity relative to the moving wall.
Importantly, the argument does not require the wall motion to be tangential: the wall, and hence the boundary grid, may possess a nonzero normal velocity.
We show that, when the moving-wall boundary conditions are imposed consistently, these additional ALE terms cancel in the entropy balance.
The viscous wall contribution retains the same entropy structure as in the corresponding stationary-grid analysis, since mesh motion enters through the ALE advective flux rather than the physical viscous flux.
Thus, the novelty of the proof lies in establishing that the moving-boundary contribution preserves the desired entropy balance even for general wall motion.

Let $\mathbf{S}_{\kappa}$ denote the vector of mathematical entropy evaluated at the nodes of element $\kappa$, and let $\mathbf{w}_{\kappa}$ denote the corresponding entropy variables.
Furthermore, let $\widehat{\mathbf{P}}_{\kappa}$ and $\widehat{\mathbf{P}}_{f}$ denote the scalar volume and surface SBP quadrature matrices, respectively, with their block extensions to the system of equations denoted by $\mathbf{P}_{\kappa}$ and $\mathbf{P}_{f}$.
Contracting \eqref{eq:SemidiscreteALENavierStokes} from the left with $\mathbf{w}_{\kappa}^{T}\mathbf{P}_{\kappa}$, using the auxiliary gradient equation \eqref{eq:ViscousLDG}, applying the SBP properties and invoking the discrete geometric conservation laws \eqref{eq:DiscreteSpaceGCL} and \eqref{eq:DiscreteTimeGCL} yields:

\begin{equation}
    \begin{aligned}
        \frac{\mathrm{d}}{\mathrm{d}\tau} \left( \mathbf{1}^{T} \widehat{\mathbf{P}}_{\kappa} \mathbf{J}_{\kappa} \mathbf{S}_{\kappa} \right) &+ \sum_{f\in\Gamma_{\kappa}^{\mathrm{I}}\cup\Gamma_{\kappa}^{\mathrm{B}}} \mathbf{1}_{f}^{T} \widehat{\mathbf{P}}_{f} \widetilde{\mathbf{F}}_{S,f,\kappa}^{\,\mathrm{ALE},-} - \sum_{f\in\Gamma_{\kappa}^{\mathrm{I}}\cup\Gamma_{\kappa}^{\mathrm{B}}} \left( \mathbf{w}_{f}^{-} \right)^{T} \mathbf{P}_{f} \widetilde{\mathbf{f}}_{f,\kappa}^{\,\mathrm{V},-} + \mathbf{DT}_{\kappa}\\
        &= \sum_{f\in\Gamma_{\kappa}^{\mathrm{I}}} \left( \mathbf{w}_{f}^{-} \right)^{T} \mathbf{P}_{f} \left[ \mathbf{g}_{f}^{(\mathrm{I,ALE}),q} + \mathbf{g}_{f}^{(\mathrm{I,V}),q} \right] \\
        &\quad+ \sum_{f\in\Gamma_{\kappa}^{\mathrm{B}}} \left( \mathbf{w}_{f}^{-} \right)^{T} \mathbf{P}_{f} \left[ \mathbf{g}_{f}^{(\mathrm{B,ALE}),q} + \mathbf{g}_{f}^{(\mathrm{B,V}),q} \right] \\
        &\quad + \sum_{f\in\Gamma_{\kappa}^{\mathrm{I}}} \left( \mathbf{R}_{f} \widetilde{\mathbf{f}}_{l(f),\kappa}^{\,\mathrm{V}} \right)^{T}
        \mathbf{P}_{f} \mathbf{g}_{f}^{(\mathrm{I}),\Theta}
        + \sum_{f\in\Gamma_{\kappa}^{\mathrm{B}}} \left(\mathbf{R}_{f} \widetilde{\mathbf{f}}_{l(f),\kappa}^{\,\mathrm{V}} \right)^{T}
        \mathbf{P}_{f} \mathbf{g}_{f}^{(\mathrm{B}),\Theta}.
    \end{aligned}
    \label{eq:ALEentropyBalanceCell}
\end{equation}

The outward-oriented ALE entropy flux is

\begin{equation*}
    \widetilde{\mathbf{F}}_{S,f,\kappa}^{\,\mathrm{ALE},-} := \sum_{m=1}^{3} n_{\xi_{l(f)},f}\, \operatorname{diag} \left[ \mathbf{R}_{f} \left( J\frac{\partial\xi_{l(f)}}{\partial x_m} \right)_{\kappa} \right] \left[ \mathbf{F}_{S,x_m} \left( \mathbf{q}_{f}^{-} \right) - V_{m,f}\mathbf{S}_{f}^{-} \right],
\end{equation*}

where $\mathbf{F}_{S,x_m}$ is the physical entropy flux in the $x_m$ direction. The discrete viscous dissipation is

\begin{equation*}
    \mathbf{DT}_{\kappa} := \sum_{l=1}^{3} \sum_{n=1}^{3} \boldsymbol{\Theta}_{\xi_l,\kappa}^{T} \mathbf{P}_{\kappa} \widehat{\mathbf{C}}_{ln,\kappa} \boldsymbol{\Theta}_{\xi_n,\kappa}
    \geq 0,
\end{equation*}

where $\widehat{\mathbf{C}}_{ln,\kappa}$ are the transformed viscous coefficient matrices.
Assume now that all interface contributions cancel or are entropy stable, and consider a single physical wall face $f_{w}$, like for the proof at the continuous level. Equation \eqref{eq:ALEentropyBalanceCell} reduces to

\begin{equation}
    \begin{aligned}
        \frac{\mathrm{d}}{\mathrm{d}\tau} \left( \mathbf{1}^{T} \widehat{\mathbf{P}}_{\kappa} \mathbf{J}_{\kappa} \mathbf{S}_{\kappa} \right) &+ \mathbf{1}_{f_w}^{T} \widehat{\mathbf{P}}_{f_w} \widetilde{\mathbf{F}}_{S,f_w,\kappa}^{\,\mathrm{ALE},-} - \left( \mathbf{w}_{f_w}^{-} \right)^{T} \mathbf{P}_{f_w} \widetilde{\mathbf{f}}_{f_w,\kappa}^{\,\mathrm{V},-} + \mathbf{DT}_{\kappa} \\
        &= \left( \mathbf{w}_{f_w}^{-} \right)^{T} \mathbf{P}_{f_w} \left[ \mathbf{g}_{f_w}^{(\mathrm{B,ALE}),q} + \mathbf{g}_{f_w}^{(\mathrm{B,V}),q} \right] \\
        & \quad + \left( \mathbf{R}_{f_w} \widetilde{\mathbf{f}}_{l(f_w),\kappa}^{\,\mathrm{V}} \right)^{T}
        \mathbf{P}_{f_w} \mathbf{g}_{f_w}^{(\mathrm{B}),\Theta}.
    \end{aligned}
    \label{eq:ALEentropyBalanceWall}
\end{equation}

At an individual wall quadrature point, the corresponding boundary contribution is

\begin{equation*}
    \begin{aligned}
        \mathcal{B}_{f_w}^{\,\mathrm{ALE}} :&= \left( \mathbf{w}_{f_w}^{-} \right)^{T} \left[ \mathbf{g}_{f_w}^{(\mathrm{B,ALE}),q} + \mathbf{g}_{f_w}^{(\mathrm{B,V}),q} + \widetilde{\mathbf{f}}_{f_w,\kappa}^{\,\mathrm{V},-} \right] - \widetilde{F}_{S,f_w,\kappa}^{\,\mathrm{ALE},-} 
        + \left( \mathbf{R}_{f_w} \widetilde{\mathbf{f}}_{l(f_w),\kappa}^{\,\mathrm{V}} \right)^{T}
        \mathbf{g}_{f_w}^{(\mathrm{B}),\Theta} \\
        &= \widetilde{\Psi}_{f_w,\kappa}^{\,\mathrm{ALE},-} - \left( \mathbf{w}_{f_w}^{-} \right)^{T} \widetilde{\mathbf{f}}_{f_w}^{\,\mathrm{ALE,bc}} + \mathcal{B}_{f_w}^{\,\mathrm{V,static}},
    \end{aligned}
\end{equation*}

where we applied the definition of entropy flux, plugged in the boundary SAT \eqref{eq:SATboundaryALE}, and defined

\begin{equation*}
    \mathcal{B}_{f_w}^{\,\mathrm{V,static}} := \left( \mathbf{w}_{f_w}^{-} \right)^{T} \left[ \mathbf{g}_{f_w}^{(\mathrm{B,V}),q} + \widetilde{\mathbf{f}}_{f_w,\kappa}^{\,\mathrm{V},-} \right]
    + \left( \mathbf{R}_{f_w} \widetilde{\mathbf{f}}_{l(f_w),\kappa}^{\,\mathrm{V}} \right)^{T}
    \mathbf{g}_{f_w}^{(\mathrm{B}),\Theta},
\end{equation*}

and the outward-oriented ALE entropy potential is defined as

\begin{equation*}
    \widetilde{\Psi}_{f_w,\kappa}^{\,\mathrm{ALE},-} :=
    \sum_{m=1}^{3} n_{\xi_{l(f_w)},f_w} \left( J\frac{\partial\xi_{l(f_w)}}{\partial x_m} \right)_{\kappa} 
    \left( \psi_{x_m}^{-} - V_{m,f_w}\phi^{-} \right),
\end{equation*}

with potential and potential-flux defined from their continuous counterpart \eqref{eq:PotentialPotentialFlux}.

As a result, the moving-wall analysis separates into a moving grid contribution and a viscous contribution:

\begin{equation*}
    \mathcal{B}_{f_w}^{\,\mathrm{ALE}} = \underbrace{ \widetilde{\Psi}_{f_w,\kappa}^{\,\mathrm{ALE},-} - \left( \mathbf{w}_{f_w}^{-} \right)^{T} \widetilde{\mathbf{f}}_{f_w}^{\,\mathrm{ALE,bc}} }_{\text{ALE inviscid wall contribution}} + \underbrace{ \mathcal{B}_{f_w}^{\,\mathrm{V,static}} }_{\text{viscous wall contribution}}.
\end{equation*}

The boundary contribution is ultimately determined by the value of the ghost state. We define the conservative ghost state $\mathbf{q}_{f}^{B}$, with respect to the primitive variables as:

\begin{equation}
    \mathbf{q}_{f}^{B} =
    \begin{bmatrix}
        \rho_{f}^{B} \\
        \rho_{f}^{B}\mathbf{u}_{f}^{B} \\
        \rho_{f}^{B}E_{f}^{B}
    \end{bmatrix},
    \qquad E_{f}^{B} = c_{v}T_{f}^{B} + \frac{1}{2} \left| \mathbf{u}_{f}^{B} \right|^{2} = \frac{RT_{f}^{B}}{\gamma-1} + \frac{1}{2} \left| \mathbf{u}_{f}^{B} \right|^{2}.
    \label{eq:GhostState}
\end{equation}

\begin{theorem}
    The penalty ALE ﬂux contribution in Eq. \eqref{eq:SemidiscreteALENavierStokes} is entropy conservative if the boundary state \eqref{eq:GhostState} is defined by the following primitive variables:
    \begin{equation}
        \rho_{f}^{B} = \rho_{f}^{-}, \qquad
        \mathbf{u}_{f}^{B} = \mathbf{u}_{f}^{-} - 2 \left[ \left( \mathbf{u}_{f}^{-} - \mathbf{u}_{w,f} \right) \cdot \mathbf{n}_{f} \right] \mathbf{n}_{f}, \qquad
        T_{f}^{B}= T_{f}^{-}.
        \label{eq:ALEmovingWallGhostState}
    \end{equation}
\end{theorem}
\begin{proof}
Consider one quadrature point on a moving-wall face $f$. Starting from Eq. \eqref{eq:ALEentropyBalanceWall}. Define the outward-oriented scaled physical normal components by
\begin{equation*}
    \mathcal{M}_{m,f} := n_{\xi_{l(f)},f} \mathbf{R}_{f} \left( J\frac{\partial\xi_{l(f)}}{\partial x_m} \right)_{\kappa}, \qquad m=1,2,3,
\end{equation*}
and define the surface Jacobian and the physical outward unit normal as
\begin{equation}
    \mathcal{J}_{f} := \left( \sum_{m=1}^{3}\mathcal{M}_{m,f}^{2} \right)^{1/2}, \qquad n_{m,f} := \frac{\mathcal{M}_{m,f}}{\mathcal{J}_{f}}.
    \label{eq:UnitFaceNormals}
\end{equation}

At the moving wall, the wall velocity is assumed to coincide with the mesh velocity,
\begin{equation*}
    \mathbf{u}_{w,f} = \mathbf{V}_{f}.
\end{equation*}

The manufactured state used in the ALE inviscid boundary flux is defined in primitive variables by reflecting the fluid velocity relative to the wall in the normal direction as \eqref{eq:ALEmovingWallGhostState}.

The corresponding conservative ghost state is
\begin{equation*}
    \mathbf{q}_{f}^{B} =
    \begin{bmatrix}
        \rho_{f}^{B} \\
        \rho_{f}^{B}\mathbf{u}_{f}^{B} \\
        \rho_{f}^{B}E_{f}^{B}
    \end{bmatrix},
    \qquad E_{f}^{B} = c_{v}T_{f}^{B} + \frac{1}{2} \left| \mathbf{u}_{f}^{B} \right|^{2} = \frac{RT_{f}^{B}}{\gamma-1} + \frac{1}{2} \left| \mathbf{u}_{f}^{B} \right|^{2}.
\end{equation*}

The average velocity associated with the numerical and ghost states is
\begin{equation*}
        \avg{\mathbf{u}}_{f} = \frac{1}{2} \left( \mathbf{u}_{f}^{-} + \mathbf{u}_{f}^{B} \right) = \mathbf{u}_{f}^{-} - \left[ \left( \mathbf{u}_{f}^{-} - \mathbf{u}_{w,f} \right) \cdot \mathbf{n}_{f} \right] \mathbf{n}_{f}.
\end{equation*}

Consequently, since $\mathbf{u}_{w,f}=\mathbf{V}_{f}$,
\begin{equation}
    \sum_{m=1}^{3} \mathcal{M}_{m,f} \left( \avg{u_m}_{f} - V_{m,f} \right) = \mathcal{J}_{f} \mathbf{n}_{f} \cdot \left( \avg{\mathbf{u}}_{f} - \mathbf{V}_{f} \right) = \mathcal{J}_{f} \mathbf{n}_{f} \cdot \left( \avg{\mathbf{u}}_{f} - \mathbf{u}_{w,f} \right) = 0.
    \label{eq:ZeroRelativeNormalAverage}
\end{equation}

For an ideal gas,
\begin{equation}
    \beta = \frac{\rho}{2p} = \frac{1}{2RT}.
\end{equation}

From the boundary state definition, it follows that
\begin{equation*}
    \beta_{f}^{B} = \beta_{f}^{-}.
\end{equation*}

Therefore, the averages appearing in the Chandrashekar flux reduce to
\begin{equation*}
    \logavg{\rho}_{f} = \rho_{f}^{-}, \qquad \logavg{\beta}_{f} = \beta_{f}^{-} = \frac{1}{2RT_{f}^{-}}, \qquad \frac{\avg{\rho}_{f}} {2\avg{\beta}_{f}} = \rho_{f}^{-}RT_{f}^{-}.
\end{equation*}

Using \eqref{eq:ZeroRelativeNormalAverage}, the metric-contracted ALE boundary flux becomes
\begin{equation}
    \widetilde{\mathbf{f}}_{f}^{\,\mathrm{ALE,bc}}
    := \sum_{m=1}^{3} \mathcal{M}_{m,f} \mathbf{F}_{x_m}^{\,\mathrm{ALE,ec}} \left( \mathbf{q}_{f}^{-}, \mathbf{q}_{f}^{B}; V_{m,f} \right) =
    \begin{bmatrix}
        0 \\
        \rho_{f}^{-}RT_{f}^{-}\mathcal{M}_{1,f} \\
        \rho_{f}^{-}RT_{f}^{-}\mathcal{M}_{2,f} \\
        \rho_{f}^{-}RT_{f}^{-}\mathcal{M}_{3,f} \\
        \rho_{f}^{-}RT_{f}^{-} \displaystyle\sum_{m=1}^{3} \mathcal{M}_{m,f}V_{m,f}
    \end{bmatrix}.
    \label{eq:ALEmovingWallNumericalFlux}
\end{equation}

Given the entropy variables definition from the continuous analysis \eqref{eq:EntropyVariables}, and the entropy potential, potential flux pair \eqref{eq:PotentialPotentialFlux}, contracting \eqref{eq:ALEmovingWallNumericalFlux} with the entropy variables of the interior state gives
\begin{equation}
    \begin{aligned}
        \left( \mathbf{w}_{f}^{-} \right)^{T} \widetilde{\mathbf{f}}_{f}^{\,\mathrm{ALE,bc}} &= \sum_{m=1}^{3} \frac{u_{m,f}^{-}}{T_{f}^{-}} \left( \rho_{f}^{-}RT_{f}^{-} \right) \mathcal{M}_{m,f} - \frac{1}{T_{f}^{-}} \left( \rho_{f}^{-}RT_{f}^{-} \right) \sum_{m=1}^{3} \mathcal{M}_{m,f}V_{m,f} \\
        &= \rho_{f}^{-} \sum_{m=1}^{3} \mathcal{M}_{m,f} \left( u_{m,f}^{-}-V_{m,f} \right) R.
    \end{aligned}
    \label{eq:EntropyContractionALEwallFlux}
\end{equation}

The outward-oriented ALE entropy potential evaluated at the interior state is
\begin{equation}
    \widetilde{\Psi}_{f}^{\,\mathrm{ALE},-} := \sum_{m=1}^{3} \mathcal{M}_{m,f} \left( \psi_{x_m,f}^{-} - V_{m,f}\phi_{f}^{-} \right) = \rho_{f}^{-} \sum_{m=1}^{3} \mathcal{M}_{m,f} \left( u_{m,f}^{-} - V_{m,f} \right) R.
    \label{eq:ALEwallEntropyPotential}
\end{equation}

It follows from \eqref{eq:EntropyContractionALEwallFlux} and \eqref{eq:ALEwallEntropyPotential} that
\begin{equation*}
    \left( \mathbf{w}_{f}^{-} \right)^{T} \widetilde{\mathbf{f}}_{f}^{\,\mathrm{ALE,bc}} = \widetilde{\Psi}_{f}^{\,\mathrm{ALE},-}.
\end{equation*}

Hence, the point-wise ALE inviscid wall contribution satisfies
\begin{equation*}
    \mathcal{B}_{f}^{\,\mathrm{ALE,I}} := \widetilde{\Psi}_{f}^{\,\mathrm{ALE},-} - \left( \mathbf{w}_{f}^{-} \right)^{T} \widetilde{\mathbf{f}}_{f}^{\,\mathrm{ALE,bc}} = 0.
\end{equation*}

Substituting this into Eq.\eqref{eq:ALEentropyBalanceWall} gives
\begin{equation*}
        \frac{\mathrm{d}}{\mathrm{d}\tau} \left( \mathbf{1}^{T} \widehat{\mathbf{P}}_{\kappa} \mathbf{J}_{\kappa} \mathbf{S}_{\kappa} \right) + \mathbf{DT}_{\kappa} = \mathcal{B}_{f_w}^{\,\mathrm{V,static}}.
\end{equation*}
\end{proof}

The inviscid contribution due to grid motion is entropy preserving. For the Navier--Stokes equations, an appropriate treatment of the viscous terms $\mathcal{B}_{f_w}^{\,\mathrm{V,static}}$, like the ones proposed by \cite{parsani2015SStabilitySolidWallBC,dalcin2019WallBC}, ensures that entropy is preserved. In the following we show that no major changes are needed to the entropy analysis. In fact, no contribution of the ALE terms affects viscous fluxes or dissipation.
The only difference from the static counterpart is that the proof needs to account for the relative velocity given the no-slip condition.
We consider again a boundary face $f$ normal to the reference direction $\xi_l$.
The interior primitive state and the manufactured viscous boundary state are defined as
\begin{equation*}
    \mathbf{v}_{f}^{-} = \left( \rho_f^{-}, \mathbf{u}_f^{-}, T_f^{-} \right)^{T},
    \qquad
    \mathbf{v}_{f}^{B,\mathrm{V}} = \left( \rho_f^{-}, 2\mathbf{V}_f-\mathbf{u}_f^{-}, T_f^{-} \right)^{T},
\end{equation*}
where the wall velocity is assumed to coincide with the grid velocity, $\mathbf{u}_{w,f}=\mathbf{V}_f$. 
The corresponding entropy variables are $\mathbf{w}_f^{-}=\mathbf{w}(\mathbf{v}_f^{-})$ and $\mathbf{w}_f^{B,\mathrm{V}}=\mathbf{w}(\mathbf{v}_f^{B,\mathrm{V}})$.
We define the discrete counterpart to the mapped viscous matrices defined in \eqref{eq:mapped_viscous_matrix}.
Let
\begin{equation}
    \mathbf{L}_f :=
    -\frac{\sigma_f}{2}
    \left[ \widehat{\mathbf{C}}_{ll} \left(\mathbf{v}_f^{-}\right) +
    \widehat{\mathbf{C}}_{ll} \left(\mathbf{v}_f^{B,\mathrm{V}}\right) \right],
    \qquad
    \sigma_f>0,
    \label{eq:MovingWallPenaltyMatrix}
\end{equation}
and define the additional viscous wall penalty by
\begin{equation}
    \mathbf{M}_{f}^{(\mathrm{B,V})} :=
    \mathbf{L}_f \left( \mathbf{w}_f^{-} - \mathbf{w}_f^{B,\mathrm{V}} \right).
    \label{eq:MovingWallViscousPenalty}
\end{equation}
Since $\widehat{\mathbf{C}}_{ll}$ is symmetric positive semi-definite, $\mathbf{L}_f$ is symmetric negative semi-definite.

\begin{theorem}
    Let the viscous wall SATs  $\mathbf{g}_{f}^{(\mathrm{B,V}),q}$  and  $\mathbf{g}_{f}^{(\mathrm{B}),\Theta}$  be constructed according to the entropy consistent wall treatment of \cite{dalcin2019WallBC}, where
    \begin{equation*}
        \mathbf{g}_{f}^{(\mathrm{B,V}),q} = \mathbf{g}_{f}^{(\mathrm{B,V}),q,\mathrm{ec}} + \mathbf{M}_{f}^{(\mathrm{B,V})} + \mathbf{L}_{f}^{(\mathrm{B,V})}.
    \end{equation*}
    Let
    \begin{equation*}
        g_f(\tau) := \kappa_f \frac{1}{T_f} \frac{\partial T}{\partial n}\bigg|_f    
    \end{equation*}
    denote the prescribed heat entropy flow at the wall.
    In the absence of the additional dissipative penalty, $\mathbf{M}_{f}^{(\mathrm{B,V})}=\mathbf{0}$, the viscous contribution to the point-wise entropy balance satisfies 
    \begin{equation*}
        \mathcal{B}_{f}^{\,\mathrm{V}} = \mathcal{J}_{f} g_f(\tau),
    \end{equation*}
    where $\mathcal{J}_{f}$ is the surface Jacobian. 
    Consequently, the viscous wall treatment is entropy-conservative for an adiabatic wall ($g_f(\tau)=0$) and entropy-stable for a prescribed, bounded heat-entropy flow.
\end{theorem}

\begin{proof}
    See Theorem 3.2 of \cite{dalcin2019WallBC}.
\end{proof}

\begin{theorem}
The additional viscous wall penalty $\mathbf{M}_{f}^{(\mathrm{B,V})}$ defined in \eqref{eq:MovingWallViscousPenalty} is entropy dissipative, i.e.,
\begin{equation*}
    \left( \mathbf{w}_f^{-} \right)^{T} \mathbf{M}_{f}^{(\mathrm{B,V})} \leq 0 .
\end{equation*}
\end{theorem}

\begin{proof}
    Let
    \begin{equation*}
        \mathbf{r}_f := \mathbf{u}_f^{-} - \mathbf{V}_f
    \end{equation*}
    denote the velocity of the fluid relative to the moving wall.
    The interior and manufactured boundary velocities can then be written as
    \begin{equation*}
        \mathbf{u}_f^{-} = \mathbf{V}_f + \mathbf{r}_f,
        \qquad
        \mathbf{u}_f^{B,\mathrm{V}} = \mathbf{V}_f - \mathbf{r}_f.
    \end{equation*}
    Since the manufactured state leaves the density and temperature unchanged,
    \begin{equation*}
        \rho_f^{B,\mathrm{V}} = \rho_f^{-},
        \qquad
        T_f^{B,\mathrm{V}} = _f^{-},
    \end{equation*}
    the jump in entropy variables is
    \begin{equation}
        \Delta \mathbf{w}_f := \mathbf{w}_f^{-} - \mathbf{w}_f^{B,\mathrm{V}}
        = \frac{2}{T_f^{-}}
        \begin{bmatrix}
            - \mathbf{V}_f \cdot \mathbf{r}_f \\
            \mathbf{r}_f \\
            0
        \end{bmatrix}.
    \label{eq:MovingWallEntropyJump}
    \end{equation}

    To relate the mapped viscous matric to the normal viscous operator, consider the surface Jacobian associated with the face and the physical unit normal \eqref{eq:UnitFaceNormals}.
    As a result, using \eqref{eq:mapped_viscous_matrix_metrics},
    \begin{equation}
        \widehat{\mathbf{C}}_{ll} = J^{-1} \sum_{m,j=1}^3 \mathcal{M}_l^m \mathbf{C}_{mj} \mathcal{M}_l^j
        = \frac{\mathcal{J}_f^2}{J} \sum_{m,j=1}^3 n_{m,f} \mathbf{C}_{mj} n_{j,f}.
        \label{eq:MappedNormalViscousMatrix}
    \end{equation}

    Introduce a local orthonormal basis $[\mathbf{n}_f, \mathbf{t}_{1,f}, \mathbf{t}_{2,f}]$ and decompose the relative velocity as
    \begin{equation*}
        \mathbf{r}_f = r_n \mathbf{n}_f + r_{t_1} \mathbf{t}_{1,f} + r_{t_2} \mathbf{t}_{2,f} 
    \end{equation*}
    For the compressible Navier--Stokes viscous matrices, direct contraction with \eqref{eq:MovingWallEntropyJump} gives
    \begin{equation*}
            \left( \mathbf{w}_f^{-} \right)\Tr
            \left[ \sum_{m,j=1}^3 n_{m,f} \mathbf{C}_{mj} \left(\mathbf{v}_f^{-}\right)n_{j,f}
        + \sum_{m,j=1}^3 n_{m,f}\mathbf{C}_{mj} \left(\mathbf{v}_f^{B,\mathrm{V}}\right)n_{j,f} \right] \Delta \mathbf{w}_f 
        = \frac{4 \mu_f}{3 T_f^{-}} \left( 4 r_n^2 + 3 r_{t_1}^2 + 3 r_{t_2}^2 \right)
    \end{equation*}
    Using \eqref{eq:MovingWallPenaltyMatrix}, \eqref{eq:MovingWallViscousPenalty}, and \eqref{eq:MappedNormalViscousMatrix}, we obtain
    \begin{align*}
        \left( \mathbf{w}_f^{-} \right)\Tr \mathbf{M}_{f}^{(\mathrm{B,V})}
        &= -\frac{\sigma_f}{2} \left( \mathbf{w}_f^{-} \right)\Tr
        \left[ \widehat{\mathbf{C}}_{ll} \left(\mathbf{v}_f^{-}\right)
        + \widehat{\mathbf{C}}_{ll} \left(\mathbf{v}_f^{B,\mathrm{V}}\right) \right]
        \Delta\mathbf{w}_f \\
        &= - \frac{2 \sigma_f \mu_f}{3 T_f^{-}} \frac{\mathcal{J}_f^2}{J}
        \left( 4 r_n^2 + 3 r_{t_1}^2 + 3 r_{t_2}^2 \right) \leq 0.
    \end{align*}

    Equivalently, since
    \begin{equation*}
        4r_n^2+3r_{t_1}^2+3r_{t_2}^2 = 3\left|\mathbf{r}_f\right|^2 + \left( \mathbf{r}_f\cdot\mathbf{n}_f \right)^2,
    \end{equation*}
    
    the entropy production can be written in coordinate-independent form as
    \begin{equation*}
        \left( \mathbf{w}_f^{-} \right)\Tr \mathbf{M}_{f}^{(\mathrm{B,V})}
        = - \frac{2\sigma_f\mu_f}{3 T_f^{-}} \frac{\mathcal{J}_f^2}{J}
        \left[ 3 \left| \mathbf{u}_f^{-}-\mathbf{V}_f \right|^2
        + \left( \left( \mathbf{u}_f^{-}-\mathbf{V}_f \right) \cdot\mathbf{n}_f \right)^2 \right]
        \leq 0 .
    \end{equation*}

    Thus, the additional viscous wall penalty remains entropy dissipative for an arbitrarily moving wall. No additional ALE contribution arises in the viscous term. The grid velocity appears only because the physical no-slip condition is imposed relative to the moving wall, $\mathbf{u}_{w,f} = \mathbf{V}_f$.

\end{proof}

\section{Numerical tests}
\label{sec:numerical-tests}
The following section covers numerical tests that verify the continuous and semi-discrete proofs for the entropy-conservative moving-wall boundary conditions. First, the ALE formulation is verified with a canonical convergence study on the isentropic vortex test case. A rotating pipe flow ensures convergence for a test with wall boundary conditions applied. Then, a laminar flow about a heaving-pitching airfoil is computed. A 3D turbulent airfoil in a 2-DOF dynamical system is investigated near instability conditions. Finally, a supersonic bluff body in unsteady rotation ensures robustness and scalability of the implementation over long integration times.

\subsection{Isentropic vortex}
The isentropic vortex is a well-established benchmark problem for the compressible Euler equations. The analytical solution is given by
\begin{align*}
    \mathcal{G} &= 1 - \left\{ \left[ x_1 - x_{1,0} - \mathcal{U}_{\infty} \cos{(\alpha)} \, t \right]^2 + \left[ x_2 - x_{2,0} - \mathcal{U}_{\infty} \sin{(\alpha)} \, t \right]^2 \right\}, \\
    \rho &= \mathcal{T}^{\,\frac{1}{\gamma - 1}}, \qquad \mathcal{T} = \left[ 1 - \epsilon_{\nu}^2 M_{\infty}^2 \frac{\gamma -1}{8 \pi^2} \text{exp}(\mathcal{G})\right], \\
    \mathcal{U}_1 &= \mathcal{U}_{\infty} \cos{(\alpha)} - \epsilon_{\nu} \frac{x_2 - x_{2,0} - \mathcal{U}_{\infty} \sin{(\alpha)} \, t \, \text{exp}(\frac{\mathcal{G}}{2})}{2 \pi}, \\
    \mathcal{U}_2 &= \mathcal{U}_{\infty} \sin{(\alpha)} - \epsilon_{\nu} \frac{x_1 - x_{1,0} - \mathcal{U}_{\infty} \cos{(\alpha)} \, t  \,\text{exp}(\frac{\mathcal{G}}{2})}{2 \pi}. \\
\end{align*}
We consider a square domain with periodic boundary conditions. The prescribed rigid motion is a combination of translations and rotations oscillating in time.
This test case ensures that the ALE scheme is correctly implemented and delivers accurate results for moving grid computations. Given that ALE fluxes are a rectification of convective fluxes, achieving entropy conservation and convergence is fundamental regardless of boundary conditions.
Multiple grid motions have been tested, in particular, the functions considered in this test case are:

\begin{align*}
    x(t) &= \cos\!\bigl(\beta(t)\bigr)(x_0-C_x) - \sin\!\bigl(\beta(t)\bigr)(y_0-C_y) + C_x + A_x\sin(\omega t), \\
    y(t) &= \sin\!\bigl(\beta(t)\bigr)(x_0-C_x) + \cos\!\bigl(\beta(t)\bigr)(y_0-C_y) + C_y + A_y\sin(\omega t+\phi),
\end{align*}
with
\begin{equation*}
    \beta(t)=\beta_0\sin(\omega t).
\end{equation*}

In the following, motion 1 is a rototraslation with $A_x=0.1$, $A_y=-0.1$, $\omega=20$, $\beta_0=\SI{0.1}{\radian}$, and $\phi=\qty[parse-numbers=false]{\frac{\pi}{6}}{\radian}$, whereas motion 2 is a simple rotation with the translation parameters $A_x=A_y=0$.
We perform two different tests: first, we consider the entropy-conservative fluxes without added dissipation through upwinding, and we make sure that entropy conservation is achieved.
Fig.~\ref{fig:isentropicvortex-dsdt} shows the evolution of the entropy rate for two different motions of the isentropic vortex. Entropy conservation is achieved up to machine precision as expected.
In addition, we perform convergence tests with the entropy-stable flux by adding upwind dissipation as described in the previous sections. Fig.~\ref{fig:isentropicvortex-convergence} shows the $L^2$ error norm for different refinements and increasing polynomial degree. The expected $p+1$ rate is achieved in all computations; for $p=5$, convergence plateaus very close to machine precision and to the time integration tolerance.

\begin{figure}[H]
    \centering
    \includegraphics{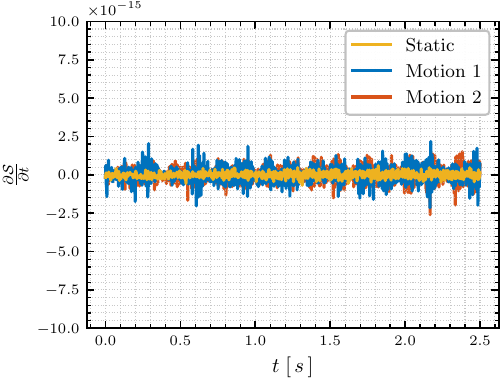}
    \caption{Entropy rate $\left( \frac{\partial \mathcal{S}}{\partial t}\right)$ as a function of time for a static grid and for two different motions.}
    \label{fig:isentropicvortex-dsdt}
\end{figure}

\begin{figure}[H]
    \centering
    \includegraphics{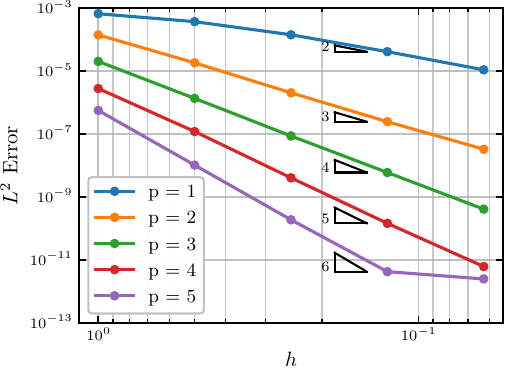}
    \caption{Convergence of the $L^2$ error of the isentropic vortex for different polynomial degrees.}
    \label{fig:isentropicvortex-convergence}
\end{figure}

\subsection{Annular Poiseuille flow}

We investigate the accuracy of the boundary conditions in a wall-bounded flow.
This test case has an analytical solution \cite{heaton2008HagenPoiseuille} for both the velocity distribution and the volume flux through an annular pipe. The expression for the axial velocity distribution, denoted as $\Uone$, as a function of the radial coordinate, $r$, is given by:
\begin{equation*}
	\ensuremath{\fnc{U}_{1}}(r) = \frac{G}{4\mu} \left[\left(R_{i}^2-r^2\right) + \left(R_{o}^2-R_{i}^2\right) \frac{ln(r/R_{i})}{ln(R_{o}/R_{i})} \right],  
\end{equation*}
where $G$ is the pressure gradient forcing term, and $R_{i}$ and $R_{o}$ are the inner and outer radii of the pipe, respectively. In this study, we set $R_{o} = 0.5$, $R_{o}/R_{i} = 4$ and $G/\mu=1$.
We chose this test case since it has an analytical solution that is not polynomial.
To approach the incompressible regime as closely as possible, we consider a Mach number of $\Mach = 10^{-5}$. Periodic boundary conditions are enforced in the axial direction, and a constant angular velocity of $2\pi$ is prescribed to the grid. The wall boundaries enforce the relative velocity according to the boundary condition.

We conduct a grid convergence study for polynomial degrees $p=2,3,4,5$, using a sequence of nested grids generated with rational Bezier basis functions, thereby preserving the exact geometric description of the pipe at each refinement level.

The error in the axial velocity profile is computed using discrete norms as follows:
\begin{equation*}
    \text{Discrete }L^{2}:\|\mathbf{u}\|_{L^{2}}^2 = \sum\limits_{\kappa=1}^{K} \mathbf{u}\Tr_{\kappa} \mathbf{P}_{\kappa} \mathbf{J}_{\kappa} \mathbf{u}\Tr_{\kappa}.
\end{equation*}

where $\mathbf{J}_{\kappa}$ is the metric Jacobian of the curvilinear transformation from physical
space to computational space of the $k$-th hexahedral element and $K$ is the
total number of non-overlapping hexahedral elements in the mesh.

The results of the grid convergence study are shown in Fig.~\ref{fig:pipe-convergence}, where we compare convergence rates for the static version of the computation and the one with a moving grid. It can be observed that the computed order of accuracy is very close to the formal value of approximately $(p+1)$, indicating that accuracy is preserved with the ALE formulation for non-periodic boundaries.

\begin{figure}[H]
    \centering
    \begin{subfigure}{0.475\linewidth}
        \centering
        \includegraphics[width=\linewidth]{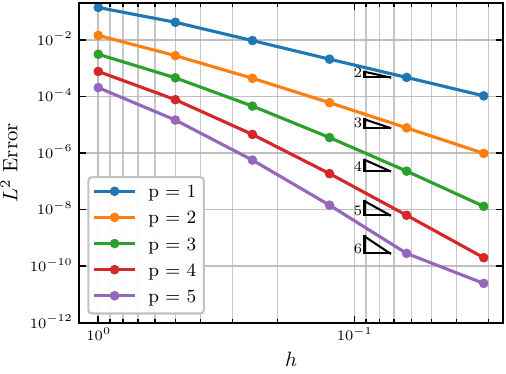}
        \caption{Static grid.}
        \label{fig:pipe-convergence-static}
    \end{subfigure}
    \hfill
    \begin{subfigure}{0.475\linewidth}
        \centering
        \includegraphics[width=\linewidth]{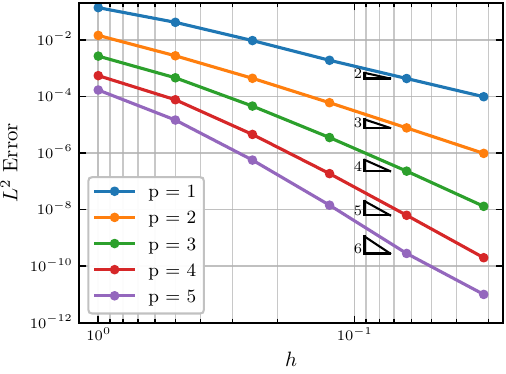}
        \caption{Moving grid.}
        \label{fig:pipe-convergence-rotating}
    \end{subfigure}
    \caption{Convergence of the $L^2$ error of the streamwise velocity for the Hagen--Poiseuille flow at different polynomial degrees. On the left, the default static grid solver, on the right, the ALE solver.}
    \label{fig:pipe-convergence}
\end{figure}

\subsection{Heaving-pitching NACA-0012}
This test case considers a NACA 0012 airfoil at $\Rey=1000$ and $\Mach=0.2$, undergoing combined heaving and pitching motions about the $c/3$ point.
The prescribed motion produces a rapid variation in the angle of attack together with a vertical translation of the airfoil.
The motion laws and computational reference data are taken from \cite{wukie2024HighFidelityCFDWorkshop}.
This configuration was considered by the Mesh Motion Working Group of the High-Fidelity CFD Workshop and therefore provides a well-established benchmark for assessing the present implementation against other high-order numerical methods.

The heaving displacement and pitching angle are prescribed by the following polynomial functions:
\begin{equation*}
    \begin{split}
        \Delta y(t) &= t^3 (8-3t)/16 \\
        \Delta \alpha(t) &= \frac{4 \pi}{9}(-t^6 + 6t^5 -12t^4 + 8t^3)
    \end{split}
\end{equation*}

\begin{figure}[H]
   \centering
   \begin{subfigure}{0.475\linewidth}
       \centering
       \includegraphics[width=\linewidth]{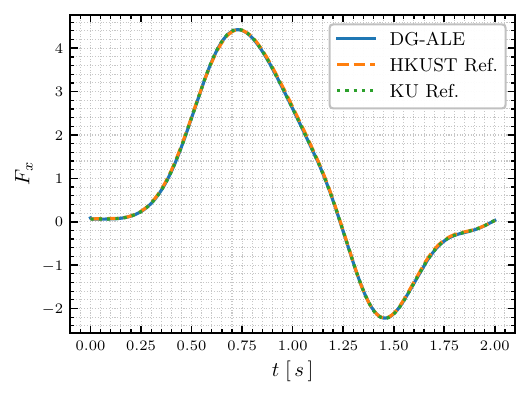}
       \caption{Time evolution of the drag force.}
       \label{fig:heavingpitching-forceX}
   \end{subfigure}
   \hfill
   \begin{subfigure}{0.475\linewidth}
       \centering
       \includegraphics[width=\linewidth]{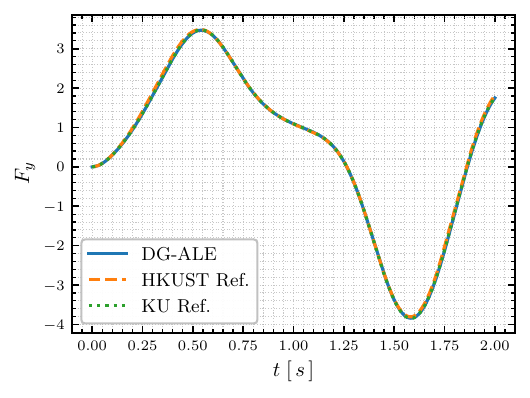}
       \caption{Time evolution of the lift force.}
       \label{fig:heavingpitching-forceY}
   \end{subfigure}
   \caption{Comparison of the aerodynamic forces in time with reference results from \cite{wukie2024HighFidelityCFDWorkshop}.}
   \label{fig:heavingpitching-force}
\end{figure}

\begin{figure}[H]
   \centering
   \begin{subfigure}{0.475\linewidth}
       \centering
       \includegraphics[width=\linewidth]{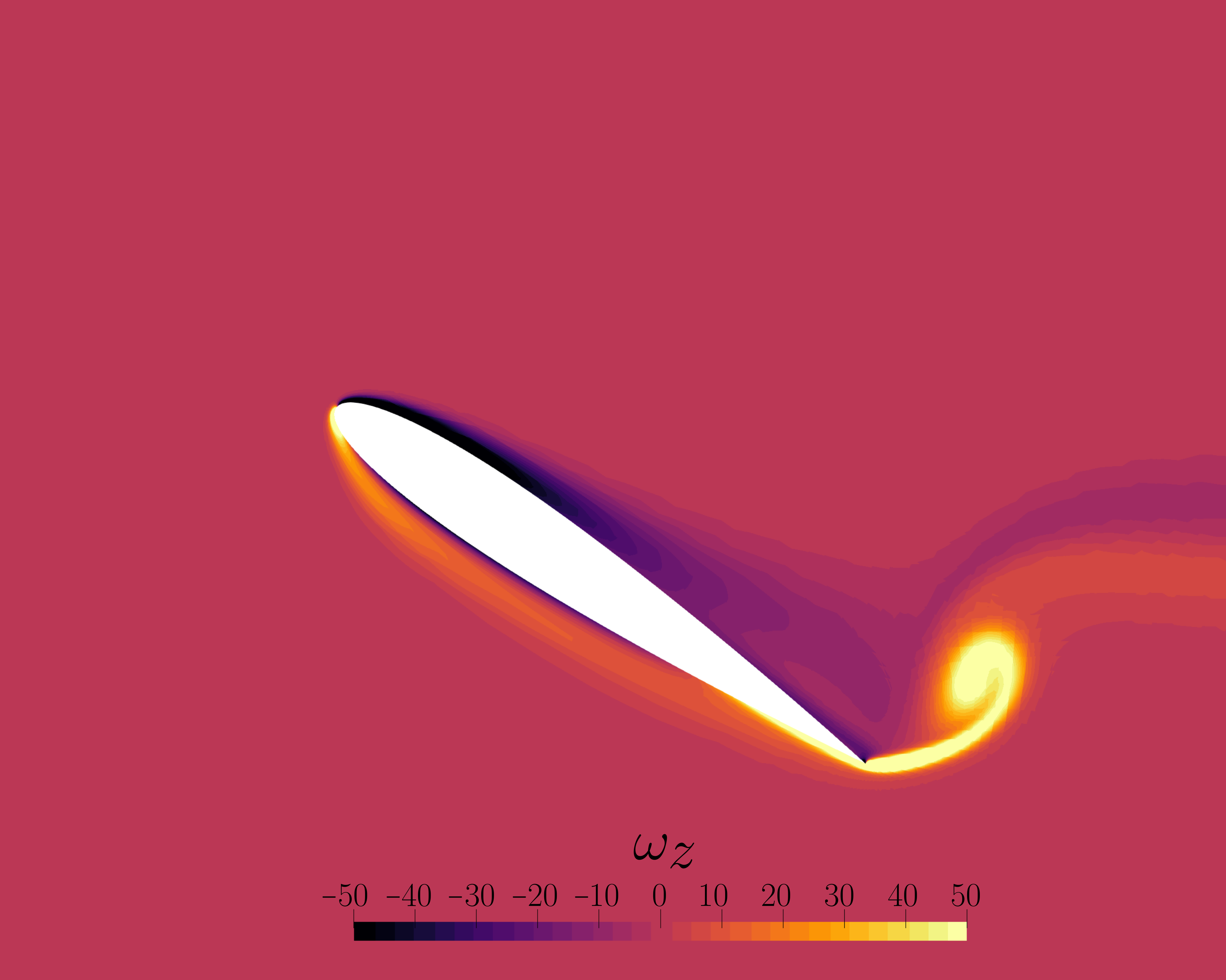}
       \caption{Vorticity plot at $t=\SI{0.5}{\second}$, compare to Fig. 3\cite{wukie2023HighFidelityCFDWorkshop}.}
       \label{fig:heavingpitching-vorticity-comp}
   \end{subfigure}
   \hfill
   \begin{subfigure}{0.475\linewidth}
       \centering
       \includegraphics[width=\linewidth]{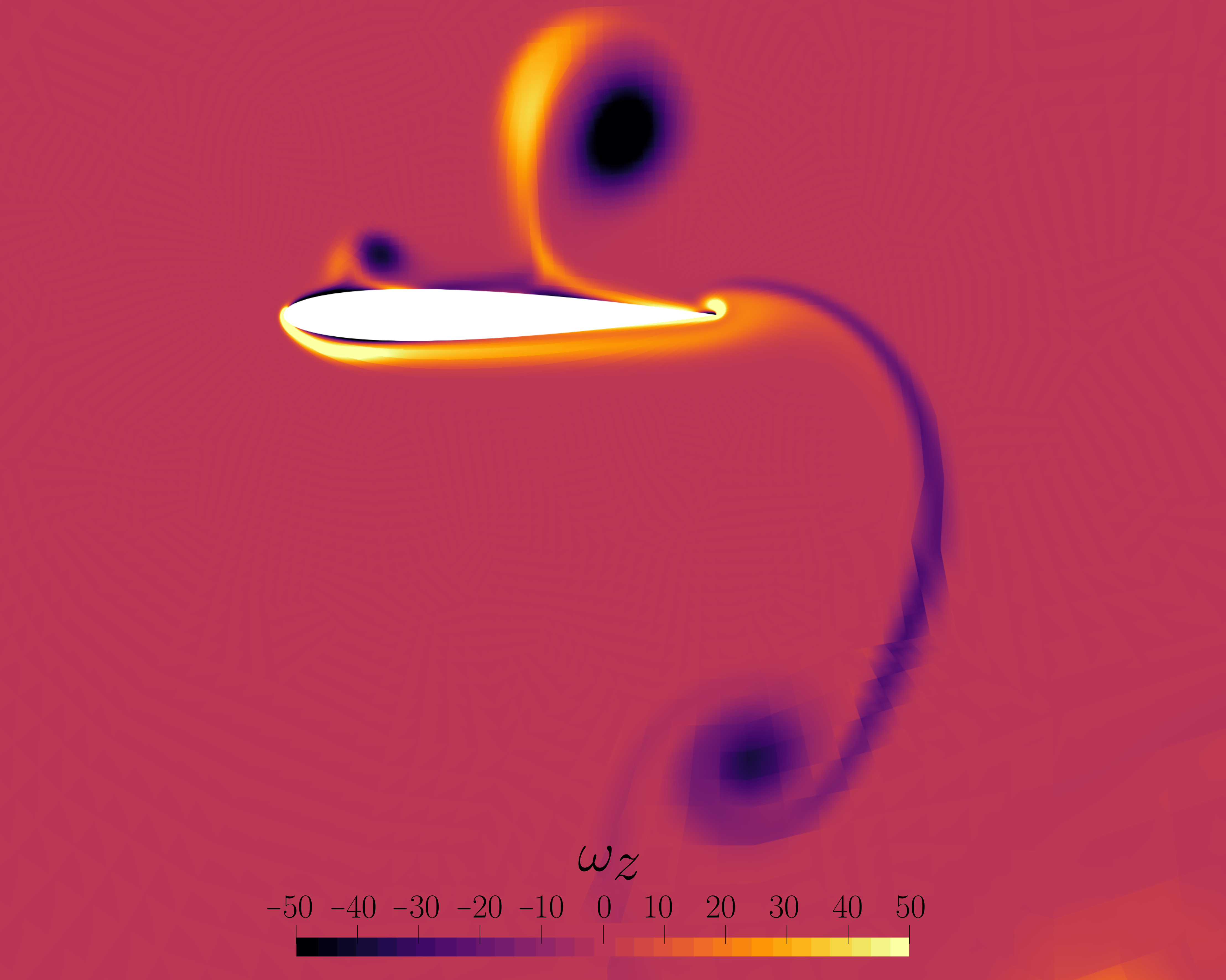}
       \caption{Vorticity plot at $t=\SI{2.0}{\second}$.}
       \label{fig:heavingpitching-vorticity-last}
   \end{subfigure}
   \caption{Comparison of the aerodynamic forces in time with reference results from \cite{wukie2024HighFidelityCFDWorkshop}.}
   \label{fig:heavingpitching-comp}
\end{figure}

The simulation is performed over two convective time units, corresponding to the time interval over which the prescribed polynomial motion is defined.
Validation is carried out by comparing both the aerodynamic forces and the resulting flow fields with the available reference solutions.
Fig.~\ref{fig:heavingpitching-forceX} and \ref{fig:heavingpitching-forceY} show the instantaneous drag and lift forces, respectively.
The predictions obtained with the present implementation exhibit excellent agreement with the reference results from the high-order solvers considered in the workshop, reproducing both the magnitude and temporal evolution of the aerodynamic loads.

A qualitative comparison is provided in Fig.~\ref{fig:heavingpitching-comp}, which shows the spanwise vorticity field at two representative instants of the motion.
The predicted flow structures and their temporal evolution are consistent with those reported in the reference computations.
Taken together, the agreement in both the aerodynamic loads and the resolved flow structures demonstrates the ability of the present mesh-motion implementation to accurately reproduce this strongly unsteady moving-boundary problem.

\subsection{Instability of a 2-DOF airfoil}
This test case considers the aeroelastic response of a two-degree-of-freedom NACA~0012 airfoil under transitional flow conditions.
The airfoil is allowed to translate in the vertical direction and to rotate about a fixed elastic axis located at $x_E = 0.4c$ from the leading edge.
The corresponding structural system is illustrated in Fig.~\ref{fig:2dof-illustration}.
The two structural degrees of freedom are the nondimensional vertical displacement $h=z/c$ and the pitch angle $\alpha$, with positive pitch defined as a counter-clockwise rotation.
The center of gravity is positioned downstream of the elastic axis by a distance $d=x_{EG}/c$, introducing an inertial coupling between the translational and rotational motions.

The governing equations are formulated directly in nondimensional form using the chord $c$, free-stream velocity $U_\infty$, and fluid density $\rho_\infty$ as reference quantities.
Time is therefore defined as $t^*=tU_\infty/c$, while the aerodynamic force and moment are normalized according to

\begin{equation*}
    F_y^* = \frac{F_y}{\rho_\infty U_\infty^2 c B},
    \qquad
    M_z^* = \frac{M_z}{\rho_\infty U_\infty^2 c^2 B},
\end{equation*}

where $B$ denotes the span represented by the aerodynamic loads. For clarity, the superscript $(\,\cdot\,)^*$ is omitted in the following.

For finite rotations, the coupled structural equations are written as

\begin{equation}
    \label{eq:2dof-dynamics}
    \begin{cases}
        m \ddot{h} &+ m d \cos(\alpha)\,\ddot{\alpha} - m d \sin(\alpha)\,\dot{\alpha}^{\,2} + b_h \dot{h} + k_h h = F_y, \\
        I_{\alpha}\ddot{\alpha} &+ m d \cos(\alpha)\,\ddot{h} + b_{\alpha}\dot{\alpha} + k_{\alpha}\alpha + k_{\alpha,3}\alpha^3 = M_z .
    \end{cases}
\end{equation}

Here, $m$ is the nondimensional structural mass, $I_{\alpha}$ is the mass moment of inertia about the elastic axis, $b_h$ and $b_{\alpha}$ are the translational and torsional damping coefficients, and $k_h$ and $k_{\alpha}$ are the corresponding linear stiffness coefficients.
The terms proportional to $m d \cos(\alpha)$ provide the inertial coupling between heave and pitch, whereas the term $-m d\sin(\alpha)\dot{\alpha}^{\,2}$ accounts for the finite-rotation centrifugal contribution associated with the offset center of gravity.
A cubic torsional stiffness $k_{\alpha,3}\alpha^3$ is additionally introduced. 
This term is weak for small pitch angles and therefore preserves the approximately linear small-amplitude dynamics, while providing a hardening restoring moment at large angles and preventing an unbounded rotational response.

The nondimensional structural parameters employed in the present test are summarized in Table~\ref{tab:2dof-parameters}.

\begin{figure}[H]
    \centering
    \includegraphics[width=0.5\linewidth]{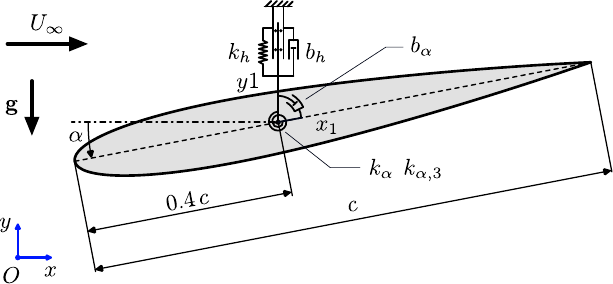}
    \caption{Illustration of the dynamical system considered.}
    \label{fig:2dof-illustration}
\end{figure}

\begin{table}[H]
    \centering
    \caption{Nondimensional structural parameters for the 2-DOF aeroelastic test case.}
    \label{tab:2dof-parameters}
    \begin{tabular}{c c c}
        \hline
        Parameter & Symbol & Value \\
        \hline
        Mass & $m$ & $12.0$ \\
        Pitch inertia & $I_{\alpha}$ & $0.35$ \\
        COG offset & $d=x_{EG}/c$ & $0.07$ \\
        Heave damping & $b_h$ & $0.066$ \\
        Heave stiffness & $k_h$ & $3.63$ \\
        Pitch damping & $b_{\alpha}$ & $\num{1.575e-3}$ \\
        Pitch stiffness & $k_{\alpha}$ & $0.196875$ \\
        Cubic pitch stiffness & $k_{\alpha,3}$ & $0.50$ \\
        \hline
    \end{tabular}
\end{table}

The coupled computation is initialized from a previously developed flow field obtained with the airfoil at a pitch angle of $\alpha=-5^\circ$, whereas the initial structural configuration is prescribed at $\alpha_0=2^\circ$ with zero heave and zero structural velocity.
Consequently, the initial fluid and structural states are intentionally not in equilibrium. The resulting difference in effective incidence, $\Delta\alpha=7^\circ$, generates a strong transient aerodynamic perturbation at the beginning of the coupled simulation.
The flow field subsequently adjusts to the new airfoil configuration while simultaneously exciting the heave--pitch structural system.

This initialization should not be interpreted as a prescribed physical gust.
Rather, it represents an impulsive perturbation of the coupled system, introduced to trigger the large-amplitude aeroelastic response.
Since the objective of the present test is to assess the robustness of the numerical method under large structural displacements and rotations, rather than to reproduce a specific experimental transient, the initial adjustment is considered part of the numerical stress test.

The fluid and structural solvers are coupled in a partitioned manner.
At the beginning of each time step, the aerodynamic force $F_y$ and moment $M_z$ are obtained by integrating the pressure and viscous stresses over the airfoil surface.
These loads are subsequently supplied to the structural solver, and Eq.~\eqref{eq:2dof-dynamics} is advanced in time using an implicit Newmark scheme.
The structural state at the beginning and end of the time step is then used to determine the motion at the intermediate fluid Runge--Kutta stages.
The grid undergoes a rigid-body translation and rotation; at this point, the ALE formulation as described by the paper is advanced in time with the explicit time stepper.

The flow conditions are defined by $\Rey_c = \num{4e4}$, $\Mach = 0.2$.
The purpose of this configuration is not to reproduce a particular experimental flutter boundary or aeroelastic benchmark quantitatively.
Instead, it is designed as a large-displacement fluid--structure interaction stress test.
The relatively low structural damping and the offset between the elastic axis and center of gravity promote a strongly coupled heave--pitch response, while the nonlinear torsional stiffness limits the growth of the pitch amplitude.
Consequently, the airfoil can undergo sustained large-amplitude oscillations involving large rotations, strongly separated flow, vortex shedding, and substantial mesh motion.

The test therefore targets the robustness of the complete coupled implementation, including aerodynamic load transfer, structural time integration, finite rigid-body motion, ALE grid velocities, and mesh quality under large displacements. Particular attention is given to the ability of the solver to maintain a stable and physically consistent solution when the airfoil experiences high instantaneous angles of attack and the surrounding flow becomes strongly separated.

Fig.~\ref{fig:2dof-time} shows the temporal evolution of the structural degrees of freedom. After the initial induced perturbation, the coupled system rapidly develops a large-amplitude aeroelastic response, with the pitch angle spanning approximately $\alpha\in[\SI{-60}{\degree},\SI{60}{\degree}]$ and the nondimensional heave displacement reaching values of approximately $h/c\in[-0.2,0.2]$.

During the initial transient, both the rotational and translational amplitudes grow rapidly as energy is transferred from the unsteady flow to the structural system. The rotational response is eventually bounded by the nonlinear hardening spring, whose cubic contribution becomes dominant at large values of $\alpha$ and provides an increasingly strong restoring moment. The resulting motion therefore remains confined to a finite range of large pitch angles rather than developing into an unbounded rotational divergence.

\begin{figure}[H]
   \centering
   \begin{subfigure}{0.475\linewidth}
       \centering
       \includegraphics[width=\linewidth]{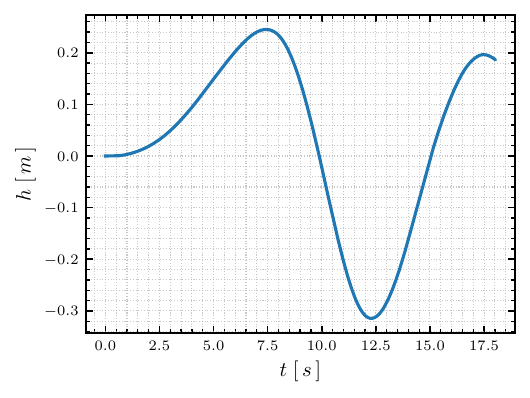}
       \caption{Vertical displacement $h$.}
       \label{fig:2dof-heave}
   \end{subfigure}
   \hfill
   \begin{subfigure}{0.475\linewidth}
       \centering
       \includegraphics[width=\linewidth]{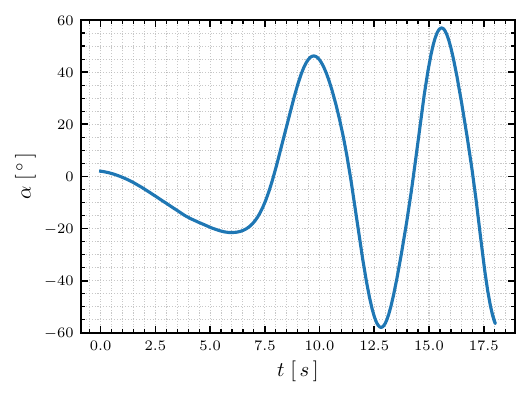}
       \caption{Angle of attack $\alpha$.}
       \label{fig:2dof-alpha}
   \end{subfigure}
   \caption{Time evolution of the free parameters driving the FSI system.}
   \label{fig:2dof-time}
\end{figure}

Such large-amplitude motion produces substantial variations in the instantaneous aerodynamic loading. In particular, the rapidly changing effective angle of attack promotes repeated separation and reattachment of the flow, accompanied by the formation and shedding of large coherent vortical structures. Consequently, pronounced fluctuations are observed in both the lift and drag coefficients. Their time histories are reported in Fig.~\ref{fig:2dof-loads}, where the strong correlation between the structural motion and the aerodynamic response is evident. The lift coefficient is closely related to the pitch angle, whereas the drag coefficient presents higher frequency content, due to the unsteadiness of vortices detaching from the surface.

Representative stages of the oscillation cycle are shown in Fig.~\ref{fig:2dof-QCrit} using the $Q$-criterion. The instantaneous flow field is characterized by extensive separation, strong vortex formation, and highly unsteady wake dynamics. These snapshots illustrate the severe aerodynamic conditions generated by the large structural excursions and highlight the ability of the coupled ALE--FSI formulation to remain stable in the presence of large translations, rotations, and strongly separated flow.

\begin{figure}[H]
   \centering
   \begin{subfigure}{0.475\linewidth}
       \centering
       \includegraphics[width=\linewidth]{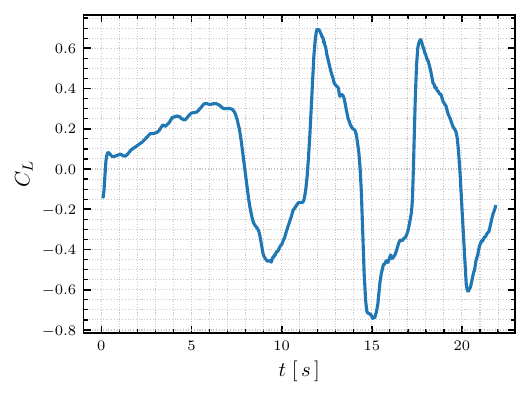}
       \caption{Lift coefficient $C_L$.}
       \label{fig:2dof-CL}
   \end{subfigure}
   \hfill
   \begin{subfigure}{0.475\linewidth}
       \centering
       \includegraphics[width=\linewidth]{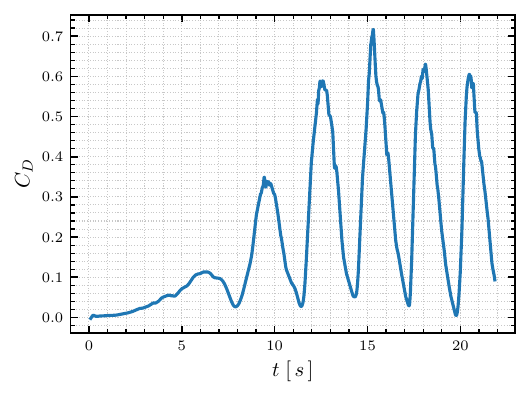}
       \caption{Drag coefficient $C_D$.}
       \label{fig:2dof-CD}
   \end{subfigure}
   \caption{Time evolution of the aerodynamic loads on the airfoil.}
   \label{fig:2dof-loads}
\end{figure}

\begin{figure}[H]
   \centering
   \begin{subfigure}{0.475\linewidth}
       \centering
       \includegraphics[width=\linewidth]{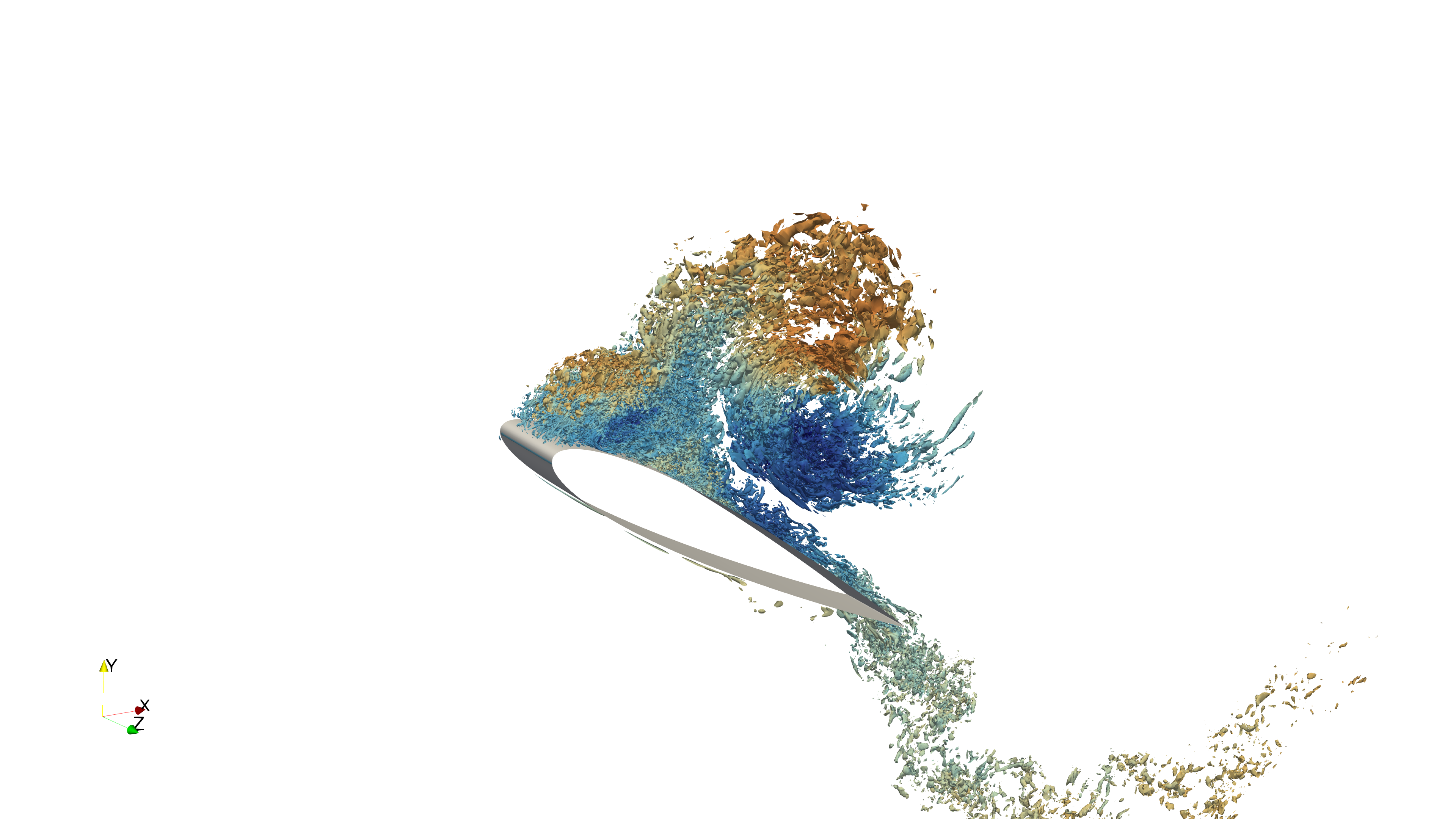}
       \caption{$t = 13.3$.}
       \label{fig:2dof-QCrit1}
   \end{subfigure}
   \hfill
   \begin{subfigure}{0.475\linewidth}
       \centering
       \includegraphics[width=\linewidth]{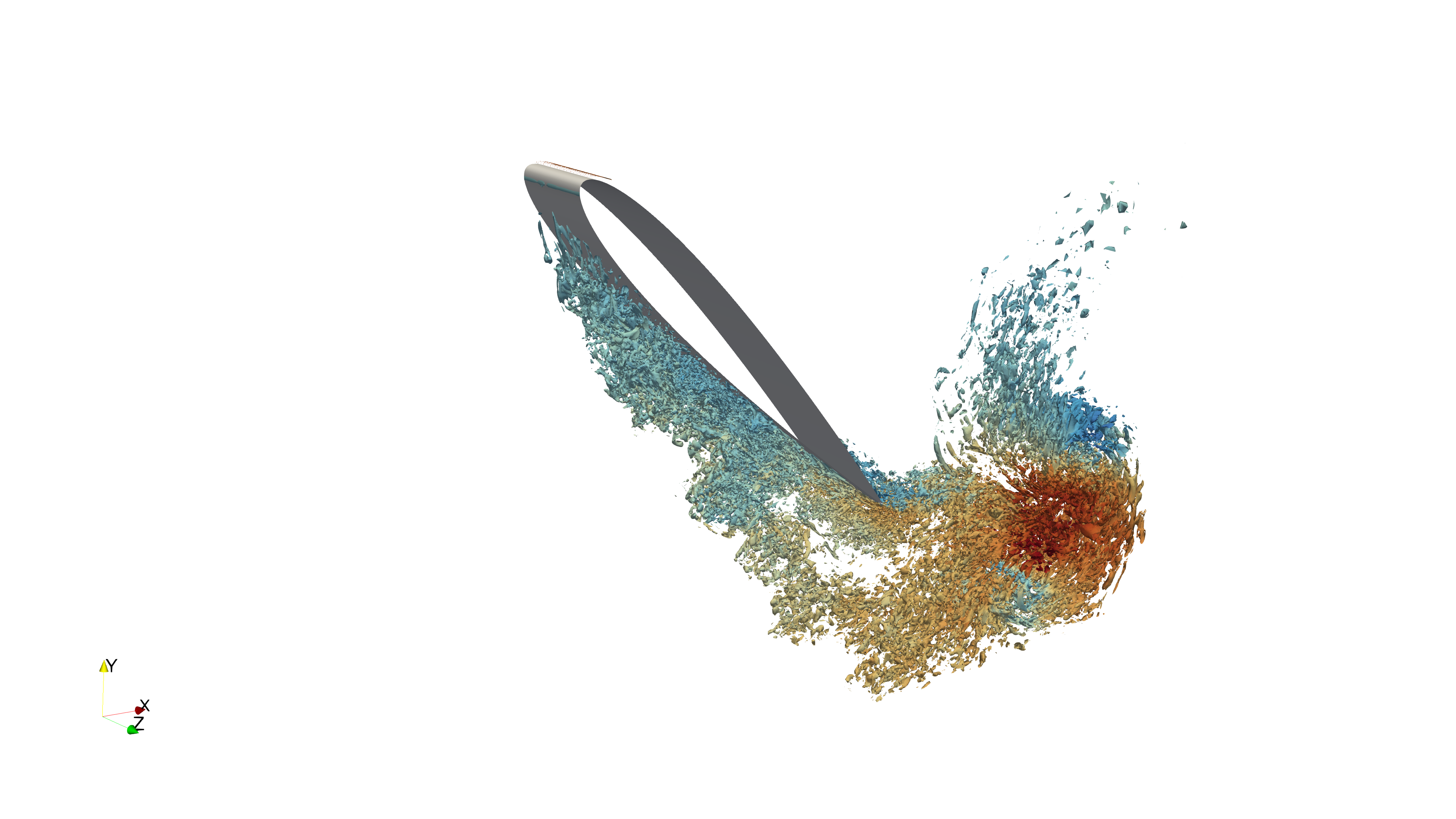}
       \caption{$t = 17.7$.}
       \label{fig:2dof-QCrit2}
   \end{subfigure}
   \\
   \begin{subfigure}{0.475\linewidth}
       \centering
       \includegraphics[width=\linewidth]{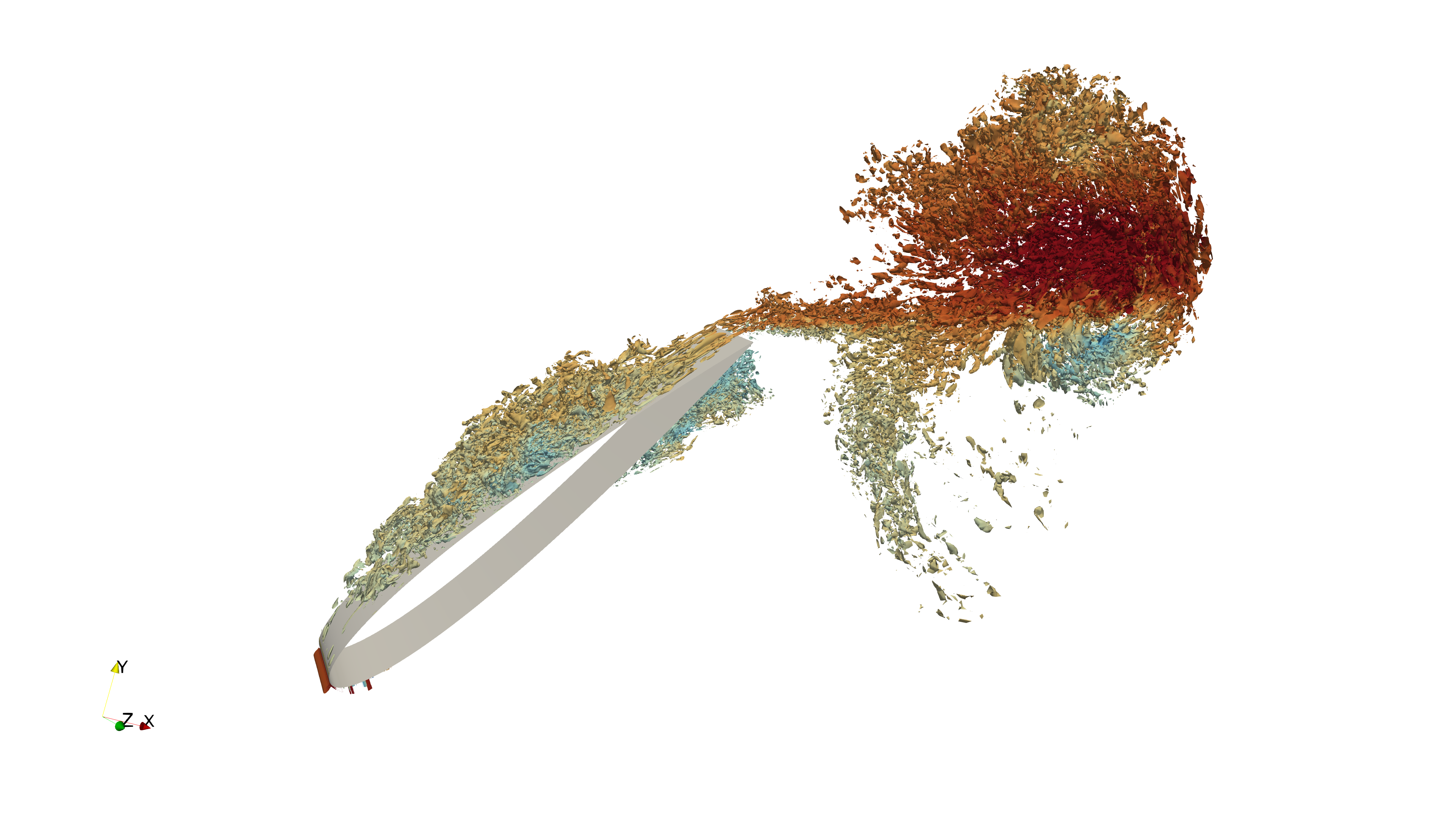}
       \caption{$t = 20.5$.}
       \label{fig:2dof-QCrit3}
   \end{subfigure}
   \hfill
   \begin{subfigure}{0.475\linewidth}
       \centering
       \includegraphics[width=\linewidth]{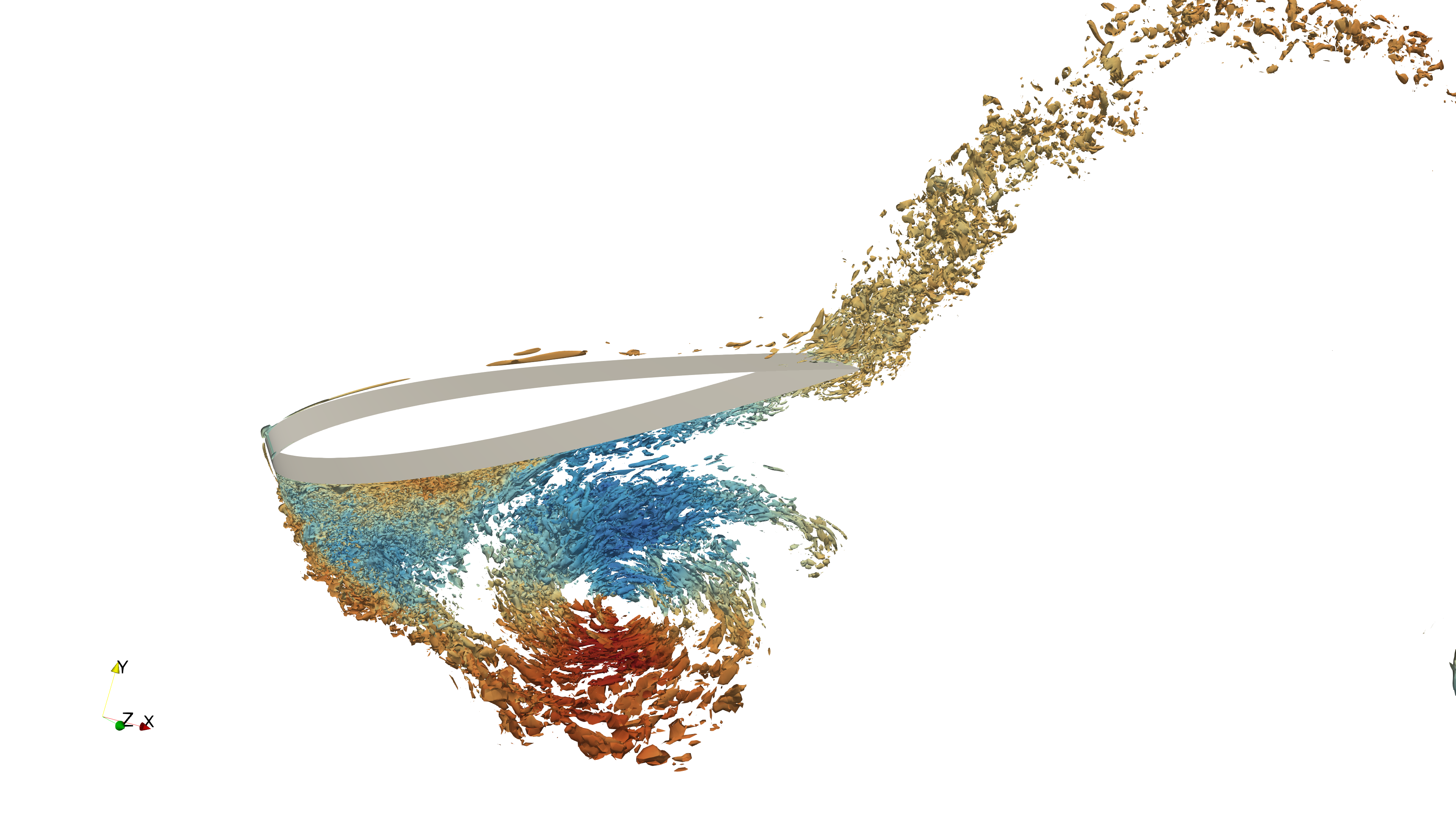}
       \caption{$t = 21.8$.}
       \label{fig:2dof-QCrit4}
   \end{subfigure}
   \caption{Iso-surfaces of the second velocity invariant $Q = 50$, colored by streamwise velocity $u$.}
   \label{fig:2dof-QCrit}
\end{figure}

\subsection{Supersonic bluff body}
This final test case considers a quintic polynomial motion of a bluff body in supersonic conditions. The body is a box resembling a satellite-like body, with length equal to $1$, width and height equal to $0.5$. The flow is in the $x$ direction at $\Mach=1.5$ and $\Rey=\num{1e4}$. The prescribed movement consists of a clockwise (CW) rotation about the $x$-axis. The angular velocity increases smoothly to a fixed value before going back to zero. The first maneuver is continuously followed by a reversal in the opposite direction, and a new rotation about the $z$-axis. After the values reach a prescribed value they continue the rotation with a constant rate in both axes.

The prescribed rigid-body motion is defined through the spatial angular velocity vector, with smooth transitions generated using the quintic smoothstep function as defined in Eq.~\eqref{eq:reentry-omega}. The motion is divided into four transient phases followed by a constant-rotation phase. The angular-velocity components are prescribed as Eq.~\eqref{eq:reentry-motion}. The motion is defined by the set of parameters reported in Tab.~\ref{tab:reentry-parameters}. Fig.~\ref{fig:reentry-motion} and \ref{fig:reentry-motion-3d} show the angular velocities as a function of time and a visualization of the attitude at different snapshots of the motion.

\begin{equation}
    \label{eq:reentry-omega}
    \boldsymbol{\omega}(t) =
    \begin{bmatrix}
        \omega_x(t)\\
        0\\ \omega_z(t)
    \end{bmatrix},
    \qquad \sigma(s)=10s^3-15s^4+6s^5,
\end{equation}

\begin{equation}
    \omega_x(t)=
    \begin{cases}
        -\Omega_{\mathrm{cw}}\, \sigma\!\left(\dfrac{t}{T_{\mathrm{acc}}}\right),                      & 0\le t<t_1, \\
        -\Omega_{\mathrm{cw}},                                                                              & t_1\le t<t_2, \\
        -\Omega_{\mathrm{cw}} \left[1-\sigma\!\left(\dfrac{t-t_2}{T_{\mathrm{dec}}}\right) \right],    & t_2\le t<t_3, \\
        \Omega_{\mathrm{ccw}}\, \sigma\!\left(\dfrac{t-t_3}{T_{\mathrm{rev}}}\right),                  & t_3\le t<t_4, \\
        \Omega_{\mathrm{ccw}},                                                                              & t\ge t_4,
    \end{cases}
    \qquad
    \omega_z(t)=
    \begin{cases}
        0, & 0\le t<t_3, \\
        \Omega_z\,\sigma\!\left(\dfrac{t-t_3}{T_{\mathrm{rev}}}\right), & t_3\le t<t_4, \\
        \Omega_z, & t\ge t_4,
    \end{cases}
    \label{eq:reentry-motion}
\end{equation}

\begin{equation*}
    t_1=T_{\mathrm{acc}},\qquad
    t_2=T_{\mathrm{acc}}+T_{\mathrm{hold}},\qquad
    t_3=T_{\mathrm{acc}}+T_{\mathrm{hold}}+T_{\mathrm{dec}},\qquad
    t_4=T_{\mathrm{acc}}+T_{\mathrm{hold}}+T_{\mathrm{dec}}+T_{\mathrm{rev}}.
\end{equation*}

\begin{table}[H]
    \centering
    \caption{Parameters defining the prescribed rigid-body rotational motion.}
    \label{tab:reentry-parameters}
    \begin{tabular}{ccc}
        \hline
        Parameter & Symbol & Value \\
        \hline
        Peak clockwise angular velocity about $x$
        & $\Omega_{\mathrm{cw}}$
        & $0.02618$ \\

        Final counterclockwise angular velocity about $x$
        & $\Omega_{\mathrm{ccw}}$
        & $0.04000$ \\

        Final angular velocity about $z$
        & $\Omega_z$
        & $0.03350$ \\

        Clockwise acceleration duration
        & $T_{\mathrm{acc}}$
        & $20.0$ \\

        Constant clockwise rotation duration
        & $T_{\mathrm{hold}}$
        & $40.0$ \\

        Clockwise deceleration duration
        & $T_{\mathrm{dec}}$
        & $20.0$ \\

        Reversal/combined-rotation ramp duration
        & $T_{\mathrm{rev}}$
        & $40.0$ \\
        \hline
    \end{tabular}
\end{table}

\begin{figure}[H]
    \centering
    \includegraphics[width=\linewidth]{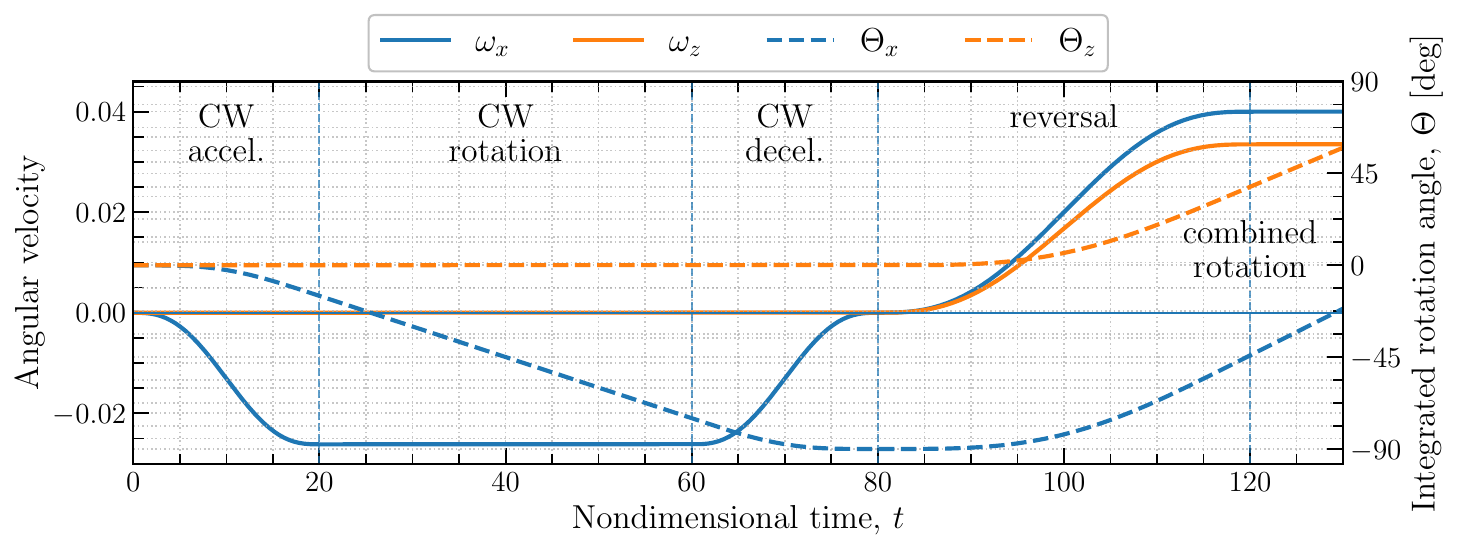}
    \caption{Angular velocity ($\omega$) and integrated rotation angle ($\Theta$) with respect to nondimensional time $t$.}
    \label{fig:reentry-motion}
\end{figure}

\begin{figure}[H]
    \centering
    \includegraphics[width=\linewidth]{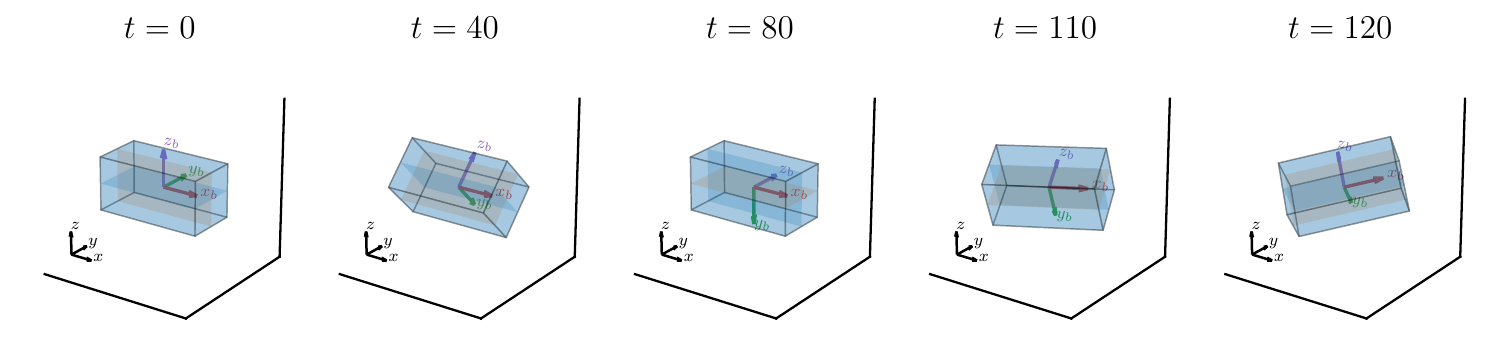}
    \caption{Attitude visualization at different times of the revolution.}
    \label{fig:reentry-motion-3d}
\end{figure}

We perform a large-eddy simulation (LES) of the bluff-body configuration over a total duration of $130$ convective time units.
The computational grid topology is shown in Fig.~\ref{fig:reentry-grid-domain}. Near the body, triangular prismatic elements are extruded in the wall-normal direction, while the remainder of the domain is initially filled with tetrahedral elements.
The resulting mesh is subsequently converted into a fully unstructured hexahedral grid. At the box surface, the first wall-normal spacing is selected to achieve approximately $y^+=1$, while the streamwise and spanwise resolutions are maintained below $10$ wall units.
Two refinement regions are positioned in the nearfield to allow a smoother transition to coarser elements in the farfield.
A uniform polynomial degree of $p=3$ is employed throughout the domain, resulting in approximately $\num{248e6}$ degrees of freedom.
Fig.~\ref{fig:reentry-flow} presents three-dimensional visualizations of the density and velocity fields at selected instants throughout the prescribed motion. The initial rotation about the streamwise axis primarily modifies the instantaneous shedding dynamics, while leaving the large-scale structure of the flow field largely unchanged. In contrast, once rotation about the transverse direction is initiated, a pronounced reorientation of the wake is observed, accompanied by a corresponding modification of the bow-shock structure and its interaction with the surrounding flow.

Overall, this test demonstrates the stability of the formulation in supersonic flow conditions and truly three-dimensional rigid-body motion.
The computation remains stable throughout all the phases of the movement, while accommodating the strong compressibility effects associated with the bow shock and the unsteady wake dynamics.
In particular, the successful treatment of simultaneous rotational components about different axes provides strong verification of the ALE formulation. It confirms its capabilities for simulations involving complex, time-dependent body kinematics.
\begin{figure}[H]
   \centering
   \begin{subfigure}{0.475\linewidth}
       \centering
       \includegraphics[width=\linewidth]{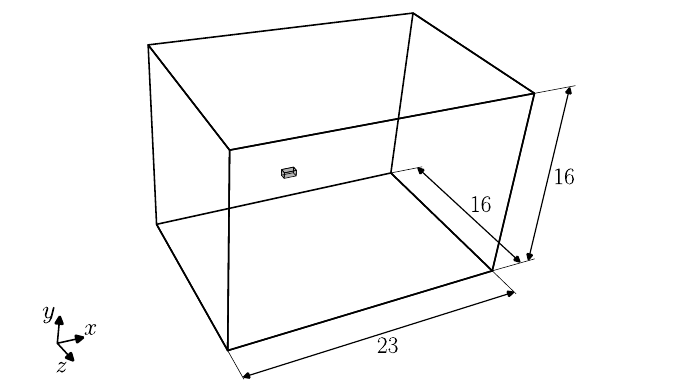}
       \label{fig:reentry-domain}
   \end{subfigure}
   \hfill
   \begin{subfigure}{0.475\linewidth}
       \centering
       \includegraphics[width=\linewidth]{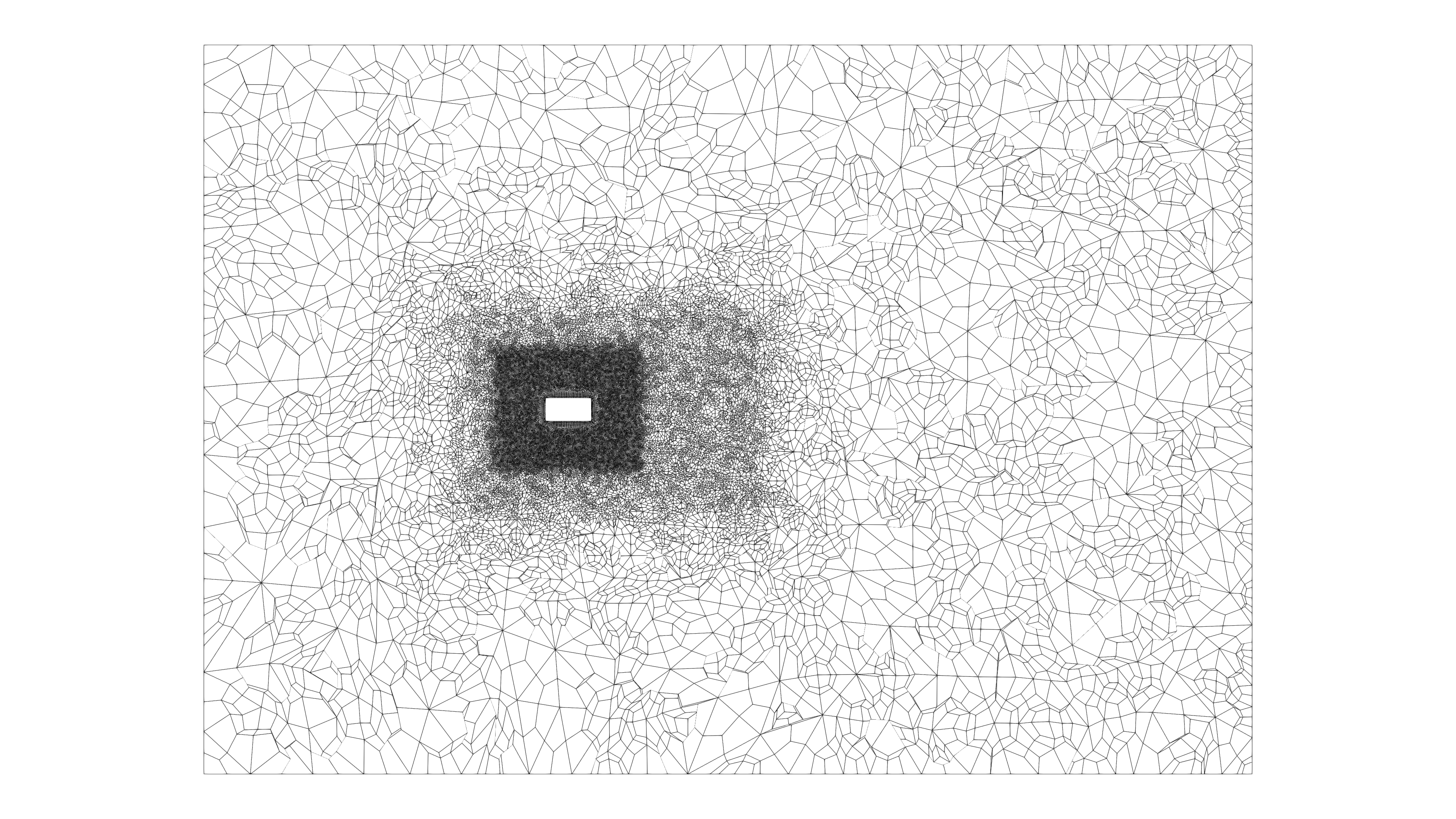}
       \label{fig:reentry-grid}
   \end{subfigure}
   \caption{Domain topology (left) and slice of the grid discretization (right).}
   \label{fig:reentry-grid-domain}
\end{figure}

\begin{figure}
    \centering
    \setlength{\tabcolsep}{3pt}
    \begin{tabular}{@{}c c c@{}}
        & \includegraphics[width=0.43\textwidth]{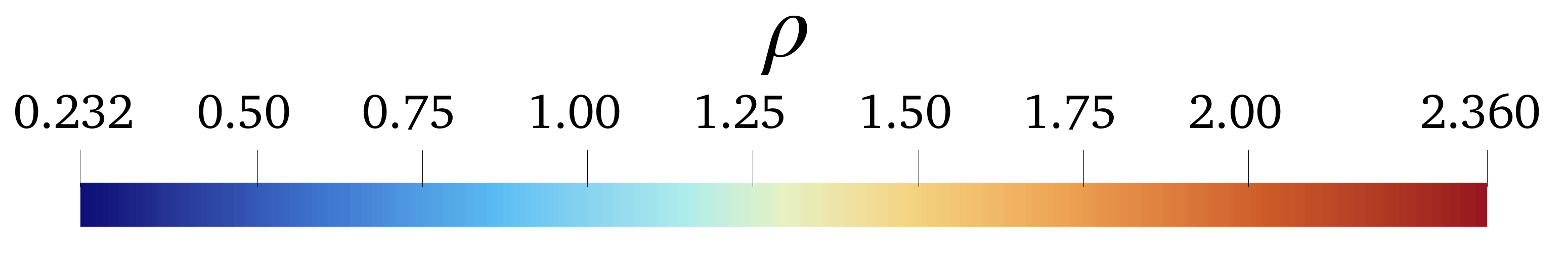}
        & \includegraphics[width=0.43\textwidth]{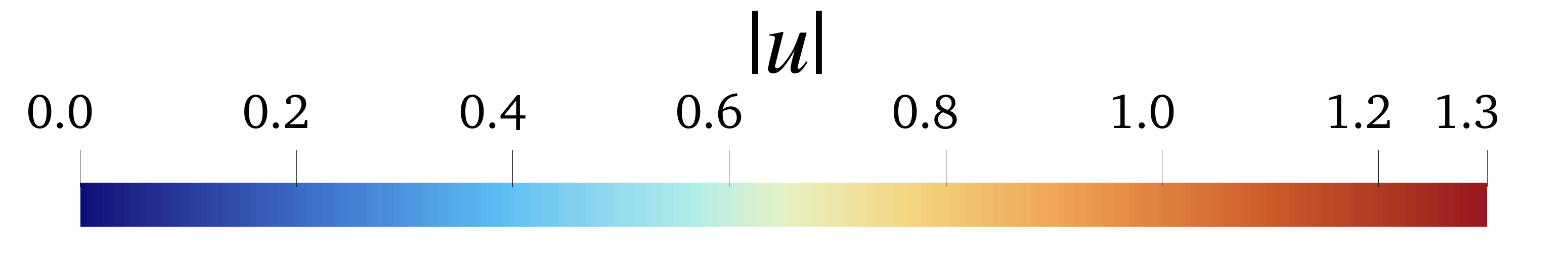} \\[1mm]
        $t=0$
        & \includegraphics[width=0.43\textwidth]{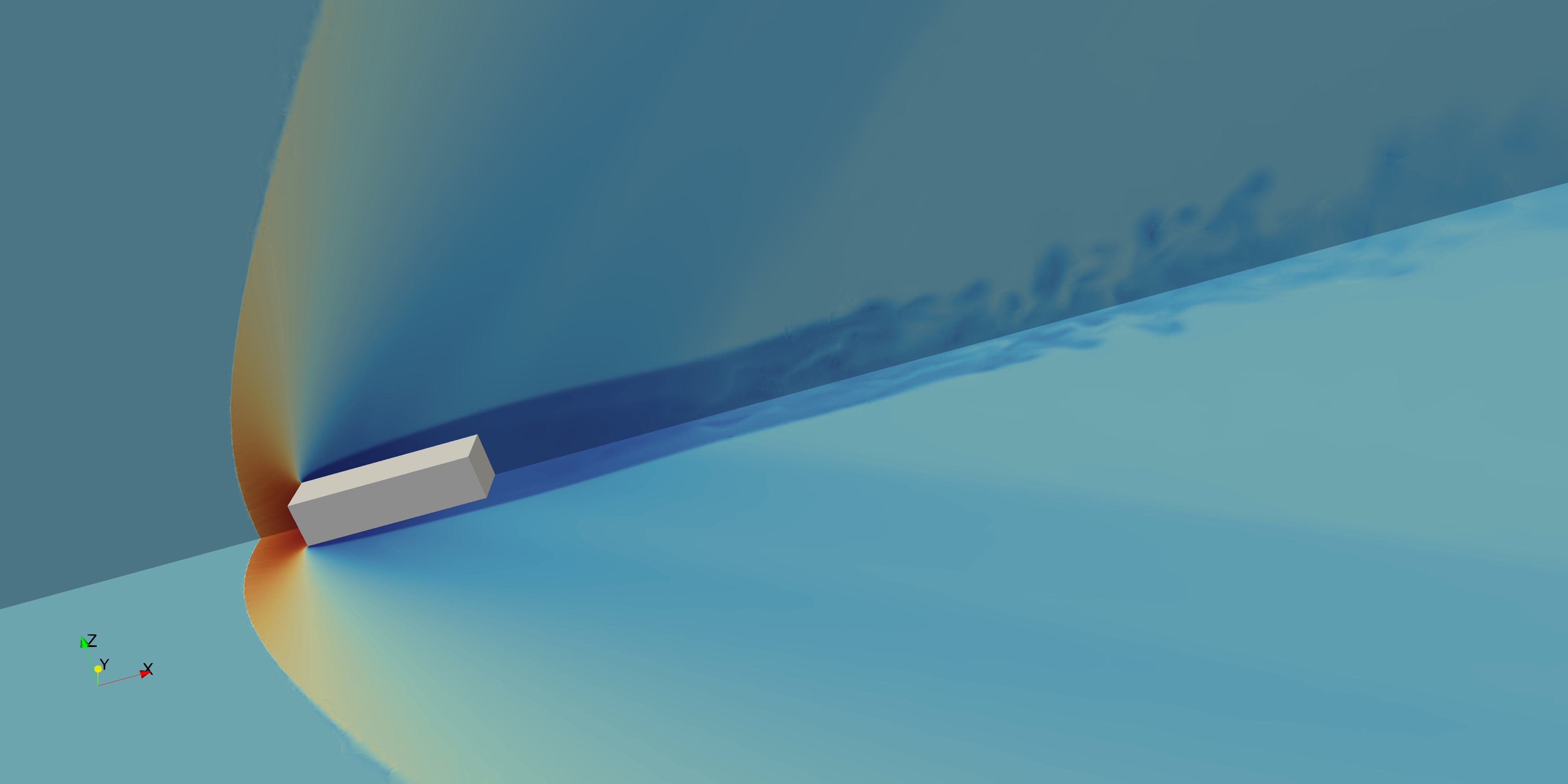}
        & \includegraphics[width=0.43\textwidth]{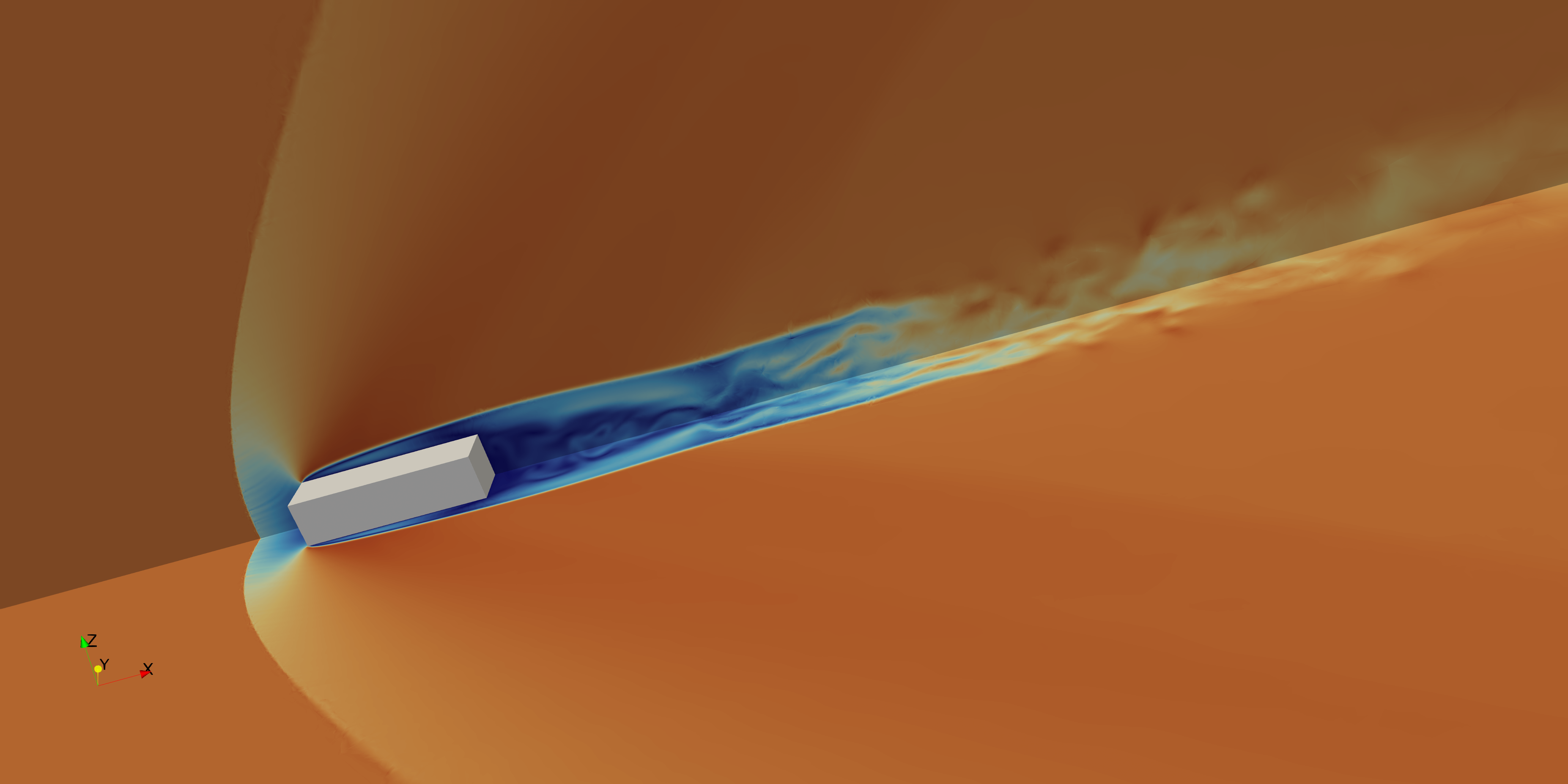} \\[1mm]
        $t=40$
        & \includegraphics[width=0.43\textwidth]{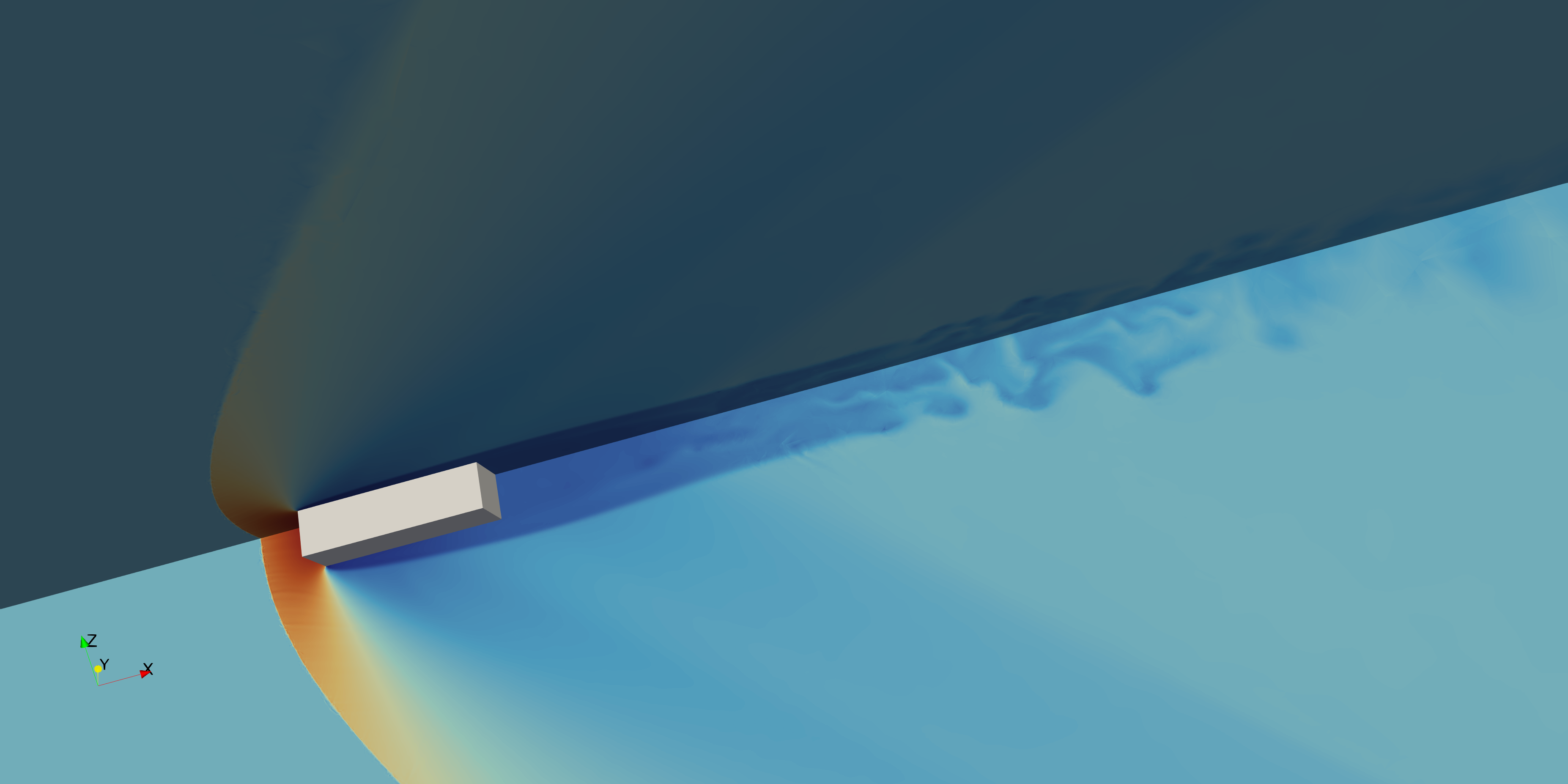}
        & \includegraphics[width=0.43\textwidth]{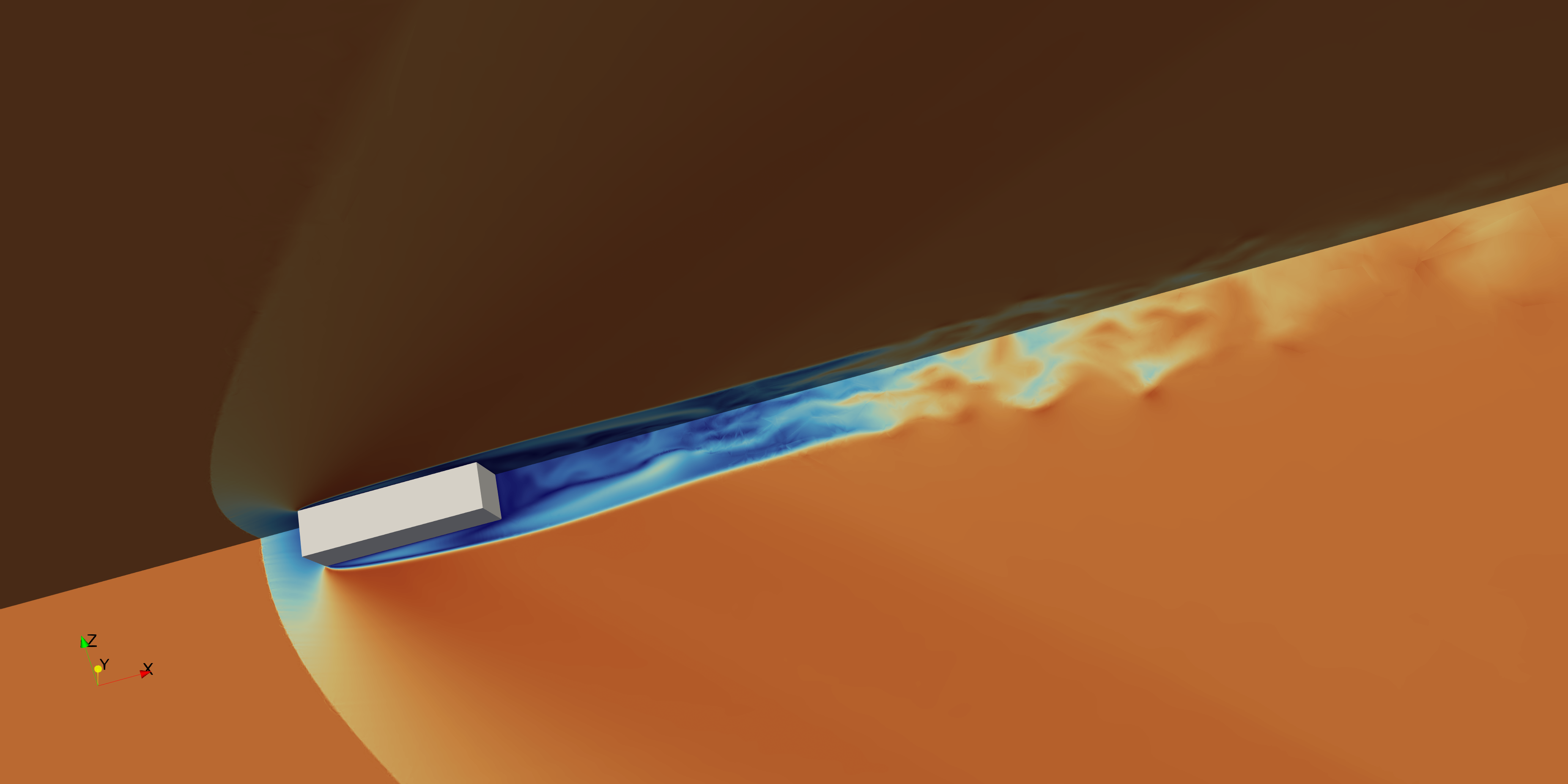} \\[1mm]
        $t=80$
        & \includegraphics[width=0.43\textwidth]{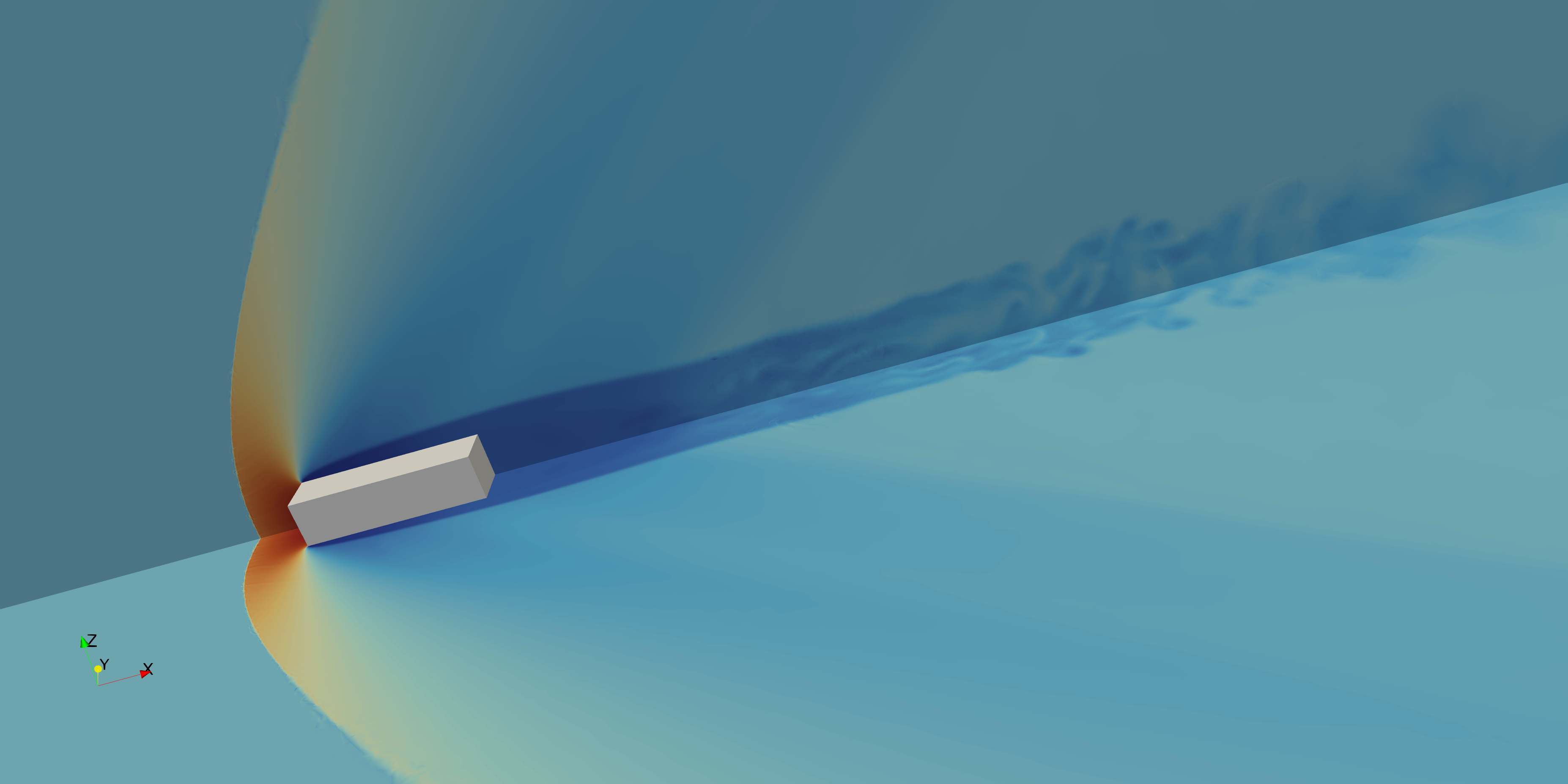}
        & \includegraphics[width=0.43\textwidth]{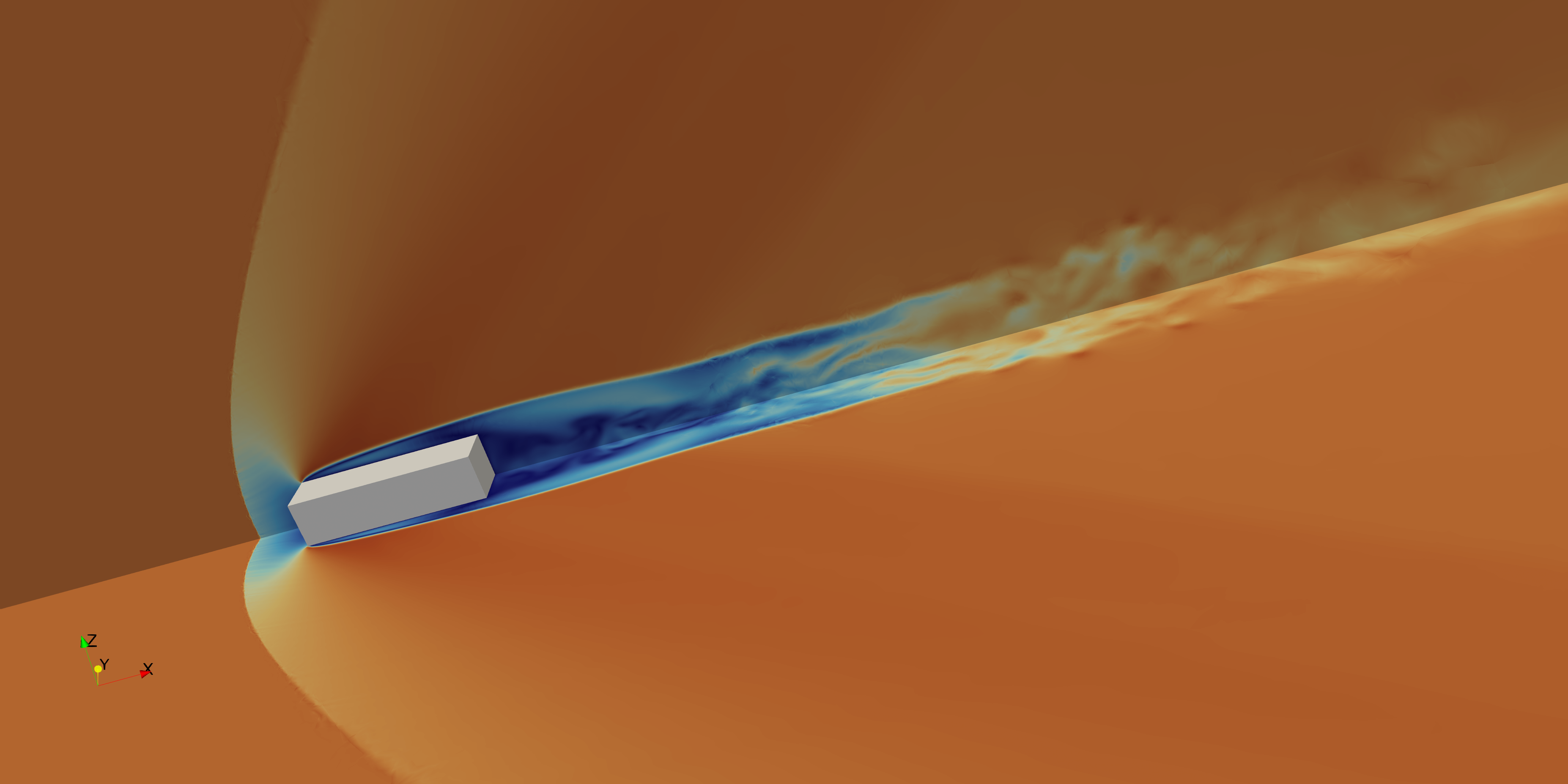} \\[1mm]
        $t=110$
        & \includegraphics[width=0.43\textwidth]{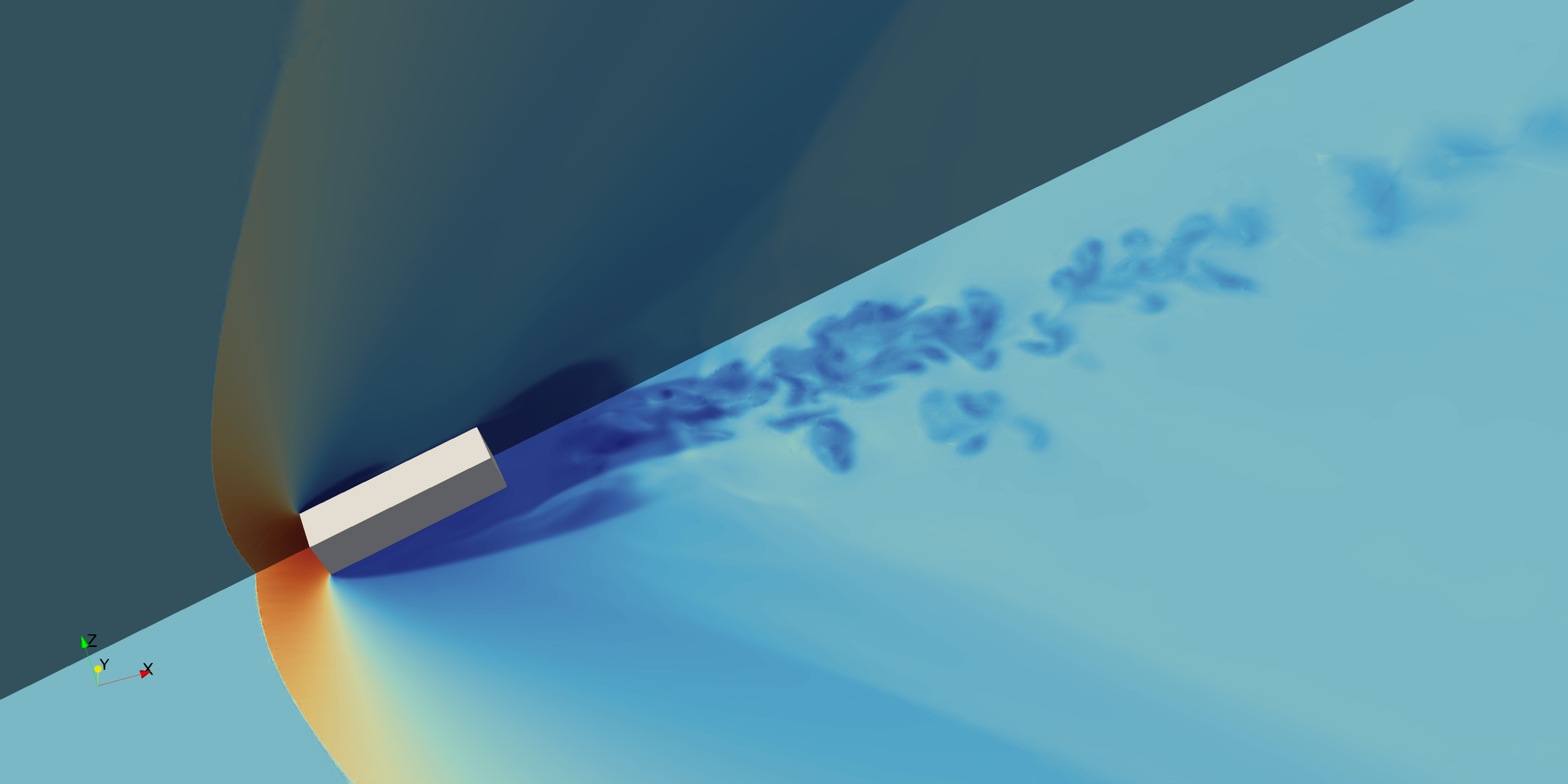}
        & \includegraphics[width=0.43\textwidth]{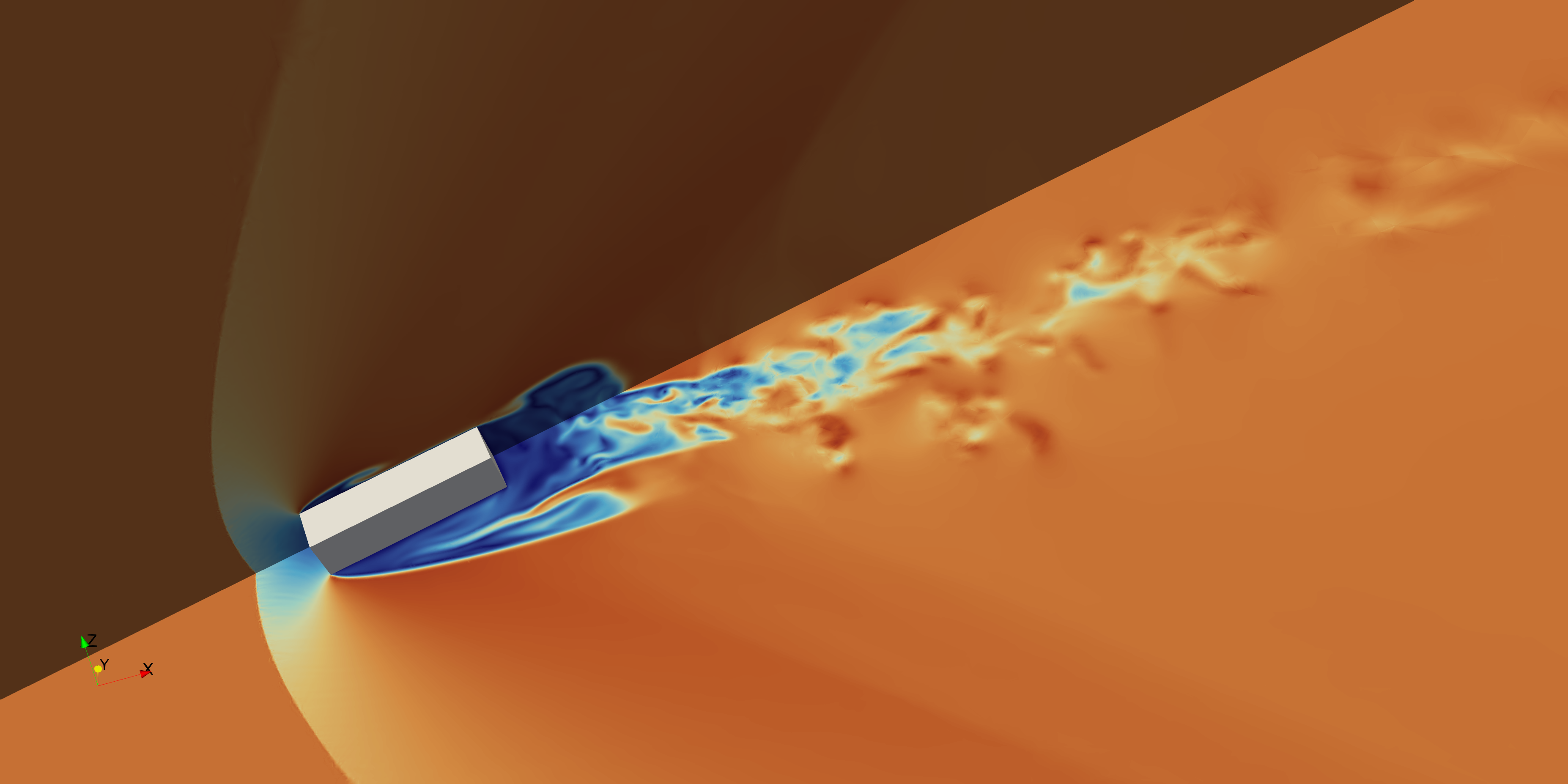} \\[1mm]
        $t=120$
        & \includegraphics[width=0.43\textwidth]{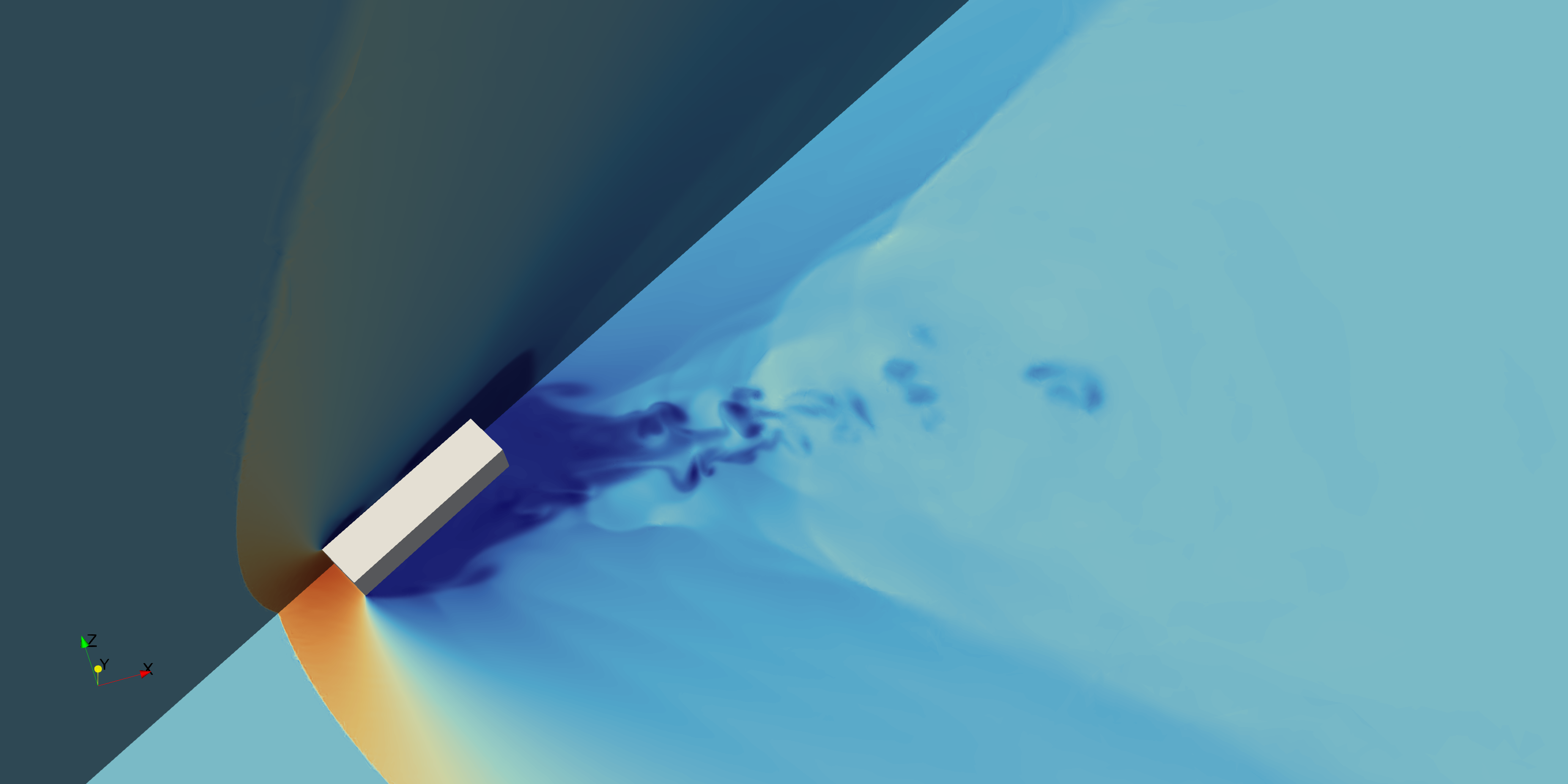}
        & \includegraphics[width=0.43\textwidth]{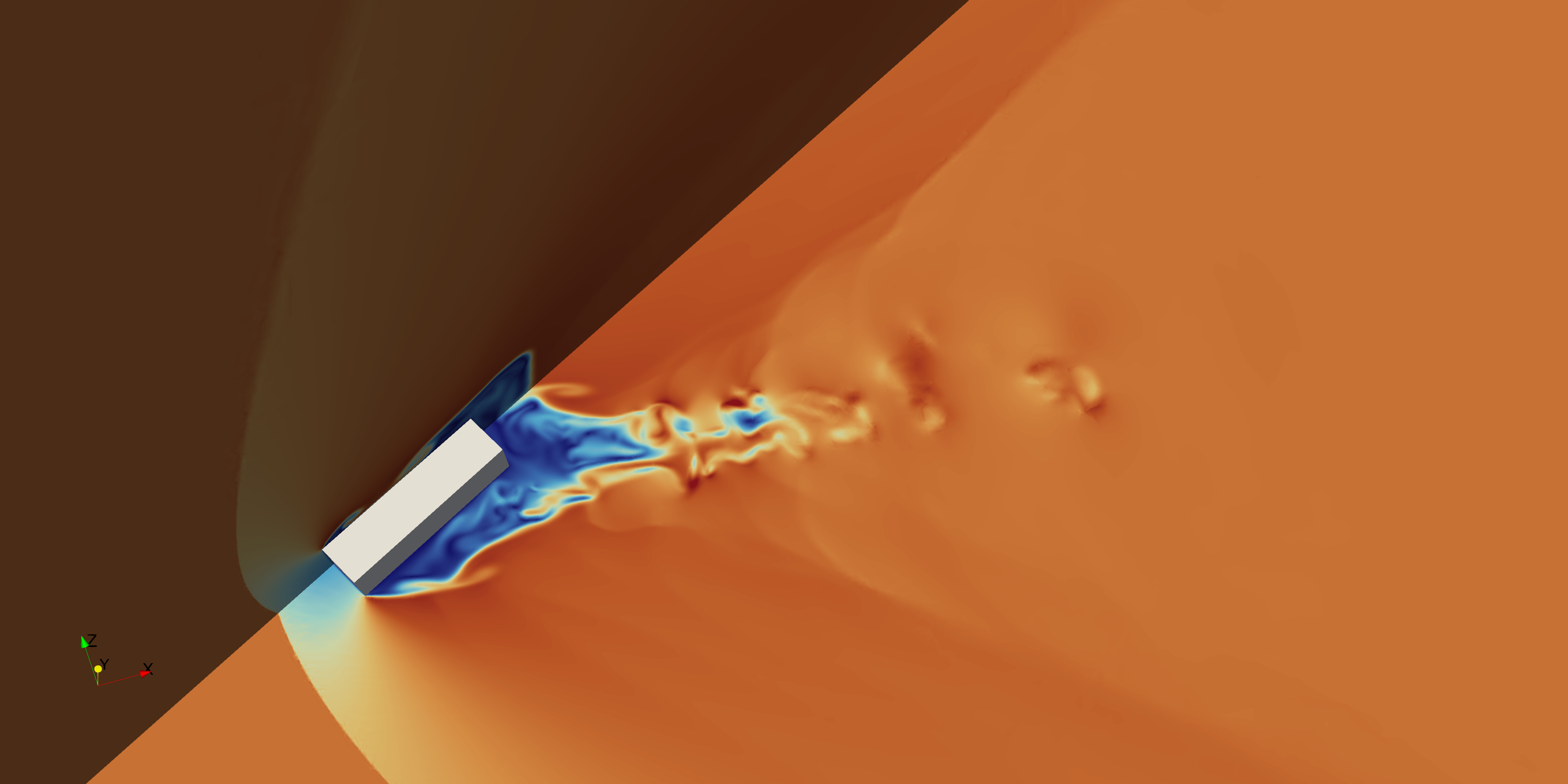}
    \end{tabular}
    \caption{Evolution of the density and velocity fields at selected timesteps.}
    \label{fig:reentry-flow}
\end{figure}

\section{Conclusions}
\label{sec:conclusions}
In this work, we demonstrate that moving-wall boundary conditions are provably entropy conservative for the compressible Euler equations and entropy stable for the compressible Navier–Stokes equations at the continuous level. The no-slip boundary conditions, when implemented with relative velocity, ensure that the inviscid contribution to the entropy balance remains bounded. A proper treatment of the viscous contributions, which results only in entropy dissipation, guarantees that the boundary contribution is entropy stable.

The proposed framework considers an arbitrary Lagrangian-Eulerian formulation of the nodal-discontinuous Galerkin method with penalty-type interface and boundary operators.
By utilizing summation-by-parts operators along with appropriate numerical fluxes, we extend the proof to the semi-discrete equations. As a result, the proposed framework ensures non-linear stability in the $L^2$ sense for both the continuous and semi-discrete equations.

An extensive set of numerical tests validates the accuracy, robustness, and scalability of the numerical implementation. Specifically, the isentropic vortex test case confirms convergence and verifies the implementation for canonical periodic boundaries. The annular pipe test demonstrates accuracy for wall-bounded flows. We achieve convergence of order $(p+1)$, as expected, for both of these tests. 

Furthermore, the 2D heaving-pitching NACA 0012 test validates forces and fluid structures against state-of-the-art high-order computational codes. The 2-DOF airfoil test showcases the scalability of the computation to 3D turbulent subsonic conditions, capturing significant unsteady motion and dynamic instabilities while remaining robust with fluid-structure interaction (FSI) coupling. Lastly, the bluff body test highlights the robustness of the implementation for supersonic large-scale computations over extended time integration.

Overall, this work shows that high-order entropy-stable numerical schemes can be effectively applied to complex moving-motion test cases across various flow regimes and multi-physics interactions.

\section*{Acknowledgments}
This work was supported by King Abdullah University of Science and Technology (KAUST) through grants no. BAS/1/1663-01-01 and REP/1/7200-01-01. The authors also gratefully acknowledge the KAUST Supercomputing Laboratory for providing computational resources on the CPU partition of the Shaheen III supercomputer. Finally, the authors acknowledge the use of Grammarly for the final check of the submitted manuscript.

\section*{CRediT authorship contribution statement}\noindent
\textbf{Luca Galimberti:} Conceptualization, Methodology, Software, Validation, Formal analysis, Investigation, Data Curation, Writing - Original Draft, Writing - Review \& Editing, Visualization. \\
\textbf{Roberto Nuca:} Methodology, Investigation, Formal Analysis, Writing - Review \& Editing, Visualization. \\
\textbf{Lisandro Dalcin:} Software, Validation, Writing - Review \& Editing, Supervision. \\
\textbf{Alberto Guardone:} Methodology, Writing - Review \& Editing, Supervision, Project administration, Funding acquisition. \\
\textbf{Matteo Parsani:}
Conceptualization, Methodology, Formal Analysis, Resources, Writing - Original Draft, Writing - Review \& Editing, Supervision, Project administration, Funding acquisition. \\

\section*{References}

\bibliographystyle{elsarticle-num}
\bibliography{refs}

\appendix

\section{Symbolic proof: moving-wall boundary condition}
\label{app:SymbolicMovingWallBC}
This appendix lists a symbolic proof of entropy conservation on arbitrary moving grids. 
The proof assumes a pointwise relation between the ghost node and the boundary node.
Physical quantities, fluxes, and 2-point fluxes match the definitions of the paper.
\begin{lstlisting}[language=Python, caption={Symbolic proof for the moving-wall boundary condition.}, label={lst:symbolic-proof-BC}]
import numpy as np
from sympy import *

init_printing()

# Define few utility routines

def zeros(*shape):
    return np.zeros(shape, dtype=object)

def dot(a, b):
    a = np.array(a)
    b = np.array(b)
    return np.dot(a, b)

def tensordot(a, b, axes=2):
    a = np.array(a)
    b = np.array(b)
    return np.tensordot(a, b, axes)

def matmul(a, b):
    a = np.array(a)
    b = np.array(b)
    return np.dot(a, b)

def outer(a, b):
    a = np.array(a)
    b = np.array(b)
    return np.outer(a, b)

# Viscosity and Conductivity are functions of temperature
Mu = symbols('mu', real=True, positive=True, cls=Function)
Kappa = symbols('kappa', real=True, positive=True, cls=Function)

# All these are constants for the analysis
R = symbols('R', real=True, positive=True)
gamma = symbols('gamma', real=True, positive=True)
cp = R*gamma/(gamma-1)
cv = R/(gamma-1)
rho_inf = symbols('rho_inf', real=True, positive=True)
T_inf = symbols('T_inf', real=True, positive=True)

def pressure(rho, T):
    return rho*R*T

def specific_energy(T):
    return cv*T

def specific_enthalpy(T):
    return cp*T

def specific_entropy(rho, T):
    s = cv*log(T/T_inf) - R*log(rho/rho_inf)
    return s

def kinetic_energy(u):
    return dot(u,u)/2

def EntropyVars(V):
    rho = V[0] # density
    u = V[1:4] # velocity
    T = V[4]   # temperature
    h = specific_enthalpy(T)
    s = specific_entropy(rho, T)
    k = kinetic_energy(u)
    W = zeros(5)
    W[0] = h/T - k/T - s
    W[1:4] = u/T
    W[4] = -1/T
    return W

def PrimitiveVars(W):
    T = -1/W[4]
    u = W[1:4]*T
    h = specific_enthalpy(T)
    k = kinetic_energy(u)
    s = h/T - k/T - W[0]
    rho = rho_inf*exp((cv*log(T/T_inf) - s)/R)
    V = zeros(5)
    V[0] = rho
    V[1:4] = u
    V[4] = T
    return V

def dWdV(Vs):
    V = zeros(5)
    V[0] = symbols('rho', real=True, positive=True)
    V[1:4] = symbols('u0:3', real=True)
    V[4] = symbols('T', real=True, positive=True)
    W = EntropyVars(V)
    mat = zeros(5, 5)
    for a in range(5):
        for b in range(5):
            mat[a,b] = W[a].diff(V[b]).subs(zip(V, Vs))
    return mat

def dVdW(Vs):
    mat = Matrix(dWdV(Vs)).inv()
    return np.array(mat.tolist())

def InviscidFluxALE(n, V, GridV):
    rho = V[0] # density
    u = V[1:4] # velocity
    T = V[4]   # temperature
    p = pressure(rho, T)
    e = specific_energy(T)
    h = specific_enthalpy(T)
    k = kinetic_energy(u)
    E = e + k # total energy
    H = h + k # total enthalpy
    un = dot(n, u)
    Vn = dot(n, GridV)
    fALE = zeros(5)
    fALE[0] = rho*(un - Vn)
    fALE[1:4] = fALE[0]*u + p*n
    fALE[4] = rho*un*H - rho*Vn*E
    return fALE

def LogAverage(a, b):
    if a == b: return a
    return (b-a)/log(b/a)

def EntropyConsistentFluxALE(n,Vl,Vr,GridV):
    # Chandrashekar, doi:10.4208/cicp.170712.010313a
    # Yamaleev, doi:10.1016/j.jcp.2019.108897 (Eq. 5.5)
    gm1 = (gamma-1)
    Bl = 1/Vl[4]
    Br = 1/Vr[4]
    p_avg = (Vl[0] + Vr[0]) / (Bl + Br)
    u_avg = 0.5*(Vl[1:4] + Vr[1:4])
    B_log = LogAverage(Bl,Br)
    rho_log = LogAverage(Vl[0],Vr[0])
    Kl = kinetic_energy(Vl[1:4]); Kr = kinetic_energy(Vr[1:4])
    K_avg = 0.5 * (Kl + Kr)
    mdot = rho_log * dot(u_avg,n)
    mdotALE   = -rho_log * dot(GridV,n)
    fALE = zeros(5)
    fALE[0]   = mdot
    fALE[1:4] = mdot * u_avg + R * p_avg * n
    fALE[4]   = mdot * (R / (gm1 * B_log) - K_avg) + dot(fALE[1:4],u_avg)
    fALE[0]   = fALE[0] + mdotALE
    fALE[1:4] = fALE[1:4] + mdotALE * u_avg
    fALE[4]   = fALE[4] + mdotALE * (R / (gm1 * B_log) + 0.5*dot(Vl[1:4],Vr[1:4]))
    return fALE

def ViscousFlux(n, V, Theta):
    u = V[1:4] # velocity
    T = V[4]   # temperature
    mu, kappa = Mu(T), Kappa(T) # mu and kappa depend on temperature
    Pi = matmul(dVdW(V), Theta) # rotate gradient to primitive vars
    grad_u = Pi[1:4,:] # gradient of velocity
    grad_T = Pi[4,:] # gradient of temperature
    div_u = np.trace(grad_u) # divergence of velocity
    I = eye(3) # identity tensor
    epsilon = zeros(3,3) # rate-of-strain tensor
    for i in range(3):
        for j in range(3):
            epsilon[i,j] = (grad_u[i,j] + grad_u[j,i])/2
    tau = mu * (2*epsilon - Rational(2,3)*div_u*I) # Cauchy stress tensor
    traction = dot(tau, n)
    fV = zeros(5) # normal viscous flux
    fV[0] = sympify(0)
    fV[1:4] = traction
    fV[4] = dot(u, traction) + kappa * dot(n, grad_T)
    return fV

# From now on, we use the following conventions:
# * (-) and '_m' refers to internal state at a point within a cell boundary face
# * (w) and '_w' refers to given boundary values, superscript (wall) in the paper
# * (+) and '_p' refers to manufactured state, superscript (B) in the paper

# Unit face normal
nrm = symbols('n0:3', real=True)
nrm = np.array(nrm)
nrm = nrm / sqrt(dot(nrm, nrm))

# Index 'k' identifying the boundary face.
# If k is None, use an arbitrary face normal.
# Otherwise, set k from 1 to 6 to assume a regular hexahedron [0,1]^3
# (unit volume, unit face areas). In the paper we pick k = 1 for the sake
# of simplicity, in here we default to k = None to proof the general case.
k = None
if k == 1: nrm[:] = sympify([-1,0,0])
if k == 2: nrm[:] = sympify([+1,0,0])
if k == 3: nrm[:] = sympify([0,-1,0])
if k == 4: nrm[:] = sympify([0,+1,0])
if k == 5: nrm[:] = sympify([0,0,-1])
if k == 6: nrm[:] = sympify([0,0,+1])

# Arbitrary wall velocity
u_w = symbols('u0:3^(wall)', real=True)
u_w = np.array(u_w)

# Wall entropy flux function
g = symbols('g', real=True, cls=Function)
t = symbols('t', real=True) # time

# State (-) at a point in the face
V_m = zeros(5)
V_m[0] = symbols('rho', real=True, positive=True)
V_m[1:4] = symbols('u0:3', real=True)
V_m[4] = symbols('T', real=True, positive=True)

# Gradient (-) at a point in the face
Theta_m = Matrix(MatrixSymbol('Theta', 5, 3))

# --------------------------------
# --- Proof for Inviscid Terms ---
# --------------------------------

# Manufactured (+) state at a point in the face
V_p = zeros(5)
u_m = V_m[1:4] # velocity at (-)
wn = dot(u_w, nrm) # arbitrary normal wall/mesh velocity
un = dot(u_m, nrm) # normal flow velocity
V_p[0] = V_m[0]                          # \
V_p[1] = V_m[1] + 2 * (wn - un) * nrm[0] # |
V_p[2] = V_m[2] + 2 * (wn - un) * nrm[1] # | Equation 32
V_p[3] = V_m[3] + 2 * (wn - un) * nrm[2] # |
V_p[4] = V_m[4]                          # /

W_m = EntropyVars(V_m) # entropy vars at (-)

fALE_m = InviscidFluxALE(nrm, V_m, u_w)
fALE_s = EntropyConsistentFluxALE(nrm, V_m, V_p, u_w)
sat_ALE = fALE_m - fALE_s # inviscid penalty term, Equation 22

# Verify consistency of the two-point ALE flux at equal states.
fALE_consistent = EntropyConsistentFluxALE(nrm, V_m, V_m, u_w)
for i in range(5):
    assert simplify(expand(fALE_consistent[i] - fALE_m[i])) == 0
 
rho_m, u_m = V_m[0], V_m[1:4] # density and velocity at (-)
Psi_m = R * rho_m * dot(nrm, u_m)
phi_m = R * rho_m

PsiALE_m = Psi_m - phi_m * dot(nrm, u_w)
F_mALE = dot(W_m, fALE_m) - PsiALE_m # entropy flux

# Verify the physical ALE entropy flux independently.
s_m = specific_entropy(V_m[0], V_m[4])
F_mALE_expected = -rho_m * s_m * dot(nrm, u_m - u_w)
assert simplify(expand(F_mALE - F_mALE_expected)) == 0

# Inviscid contributions to the RHS of Equation 26
RHS_I = (- F_mALE + dot(W_m, sat_ALE))

RHS_I = ratsimp(RHS_I)  # simplify expressions
RHS_I = expand(RHS_I)   # trigger cancellations
assert RHS_I == 0       # Entropy-conservative BC

# --------------------------------
# --- Proof for Viscous Terms ----
# --------------------------------

# Manufactured (+) state at a point in the face
# Note that this way u^(w) = 1/2*(u^(-) + u^(+))
V_p = zeros(5)
V_p[0] = +V_m[0]              # \
V_p[1] = -V_m[1] + 2 * u_w[0] # |
V_p[2] = -V_m[2] + 2 * u_w[1] # |
V_p[3] = -V_m[3] + 2 * u_w[2] # |
V_p[4] = +V_m[4]              # /

# Manufactured (+) gradient at a point in the face
Pi_p = zeros(5, 3)
Pi_m = matmul(dVdW(V_m), Theta_m)
Pi_p[0,:] = -Pi_m[0,:]            # \
Pi_p[1,:] = +Pi_m[1,:]            # |
Pi_p[2,:] = +Pi_m[2,:]            # |
Pi_p[3,:] = +Pi_m[3,:]            # |
Pi_p[4,:] = -Pi_m[4,:]            # /
Theta_p = matmul(dWdV(V_p), Pi_p) #

W_m = EntropyVars(V_m) # entropy vars at (-)
W_p = EntropyVars(V_p) # entropy vars at (+)

fV_m = ViscousFlux(nrm, V_m, Theta_m)
fV_p = ViscousFlux(nrm, V_p, Theta_p)
sat_V = -(fV_m-fV_p)/2 # viscous penalty term

# Source heat flux corresponding to g(t)
src_L = zeros(5)
T_m = V_m[4] # temperature at (-)
src_L[4] = -T_m * g(t)

# Volume viscous flux to contract with gradient
fV = zeros(5,3)
fV[:,0] = ViscousFlux([1,0,0], V_m, Theta_m)
fV[:,1] = ViscousFlux([0,1,0], V_m, Theta_m)
fV[:,2] = ViscousFlux([0,0,1], V_m, Theta_m)
# Gradient penalty term
sat_Theta = -Rational(1,2) * outer(W_m - W_p, nrm)

RHS_V = (
    + dot(W_m, fV_m)
    + dot(W_m, sat_V)
    + dot(W_m, src_L)
    + tensordot(fV, sat_Theta)
)

RHS_V = ratsimp(RHS_V) # simplify expressions
RHS_V = expand(RHS_V)  # trigger cancellations
assert RHS_V == g(t)   # Entropy-conservative BC

# --------------------------------
# --- Proof for IP dissipation ---
# --------------------------------
# This constant is the IP strength factor

sigma_IP = symbols('sigma_IP', real=True, positive=True)

# Proof using viscous flux evaluations and normal entropy jumps,
# this is the practical way to implement the IP dissipation.
# This proof holds for any normal vector and wall velocity. Note
# that in the paper, and for the sake of simplicity, we express the
# IP dissipation in terms of the # viscous C_jj, j=1,2,3 and we
# assume normal vectors aligned with the Cartesian coordinate directions.

dWn = outer(W_m - W_p, nrm)
fm = ViscousFlux(nrm, V_m, dWn)
fp = ViscousFlux(nrm, V_p, dWn)
M = -sigma_IP * (fm + fp)/2
RHS_IP = dot(W_m, M)
RHS_IP = expand(RHS_IP)
RHS_IP = simplify(RHS_IP)

# Now we check that the IP term is non-positive
u_m = V_m[1:4]
T_m = V_m[4]
du = u_m - u_w
N = outer(nrm, nrm) # rank 1, symmetric, semi-PD matrix
RHS_IP_explicit = (
    -2*sigma_IP*Mu(T_m)/(3*T_m) * ( # this factor is negative
      + dot(du, matmul(N, du))      # this term is non-negative
      + 3*dot(du,du)*dot(nrm,nrm)   # this term is non-negative
    ) # then this is non-positive and thus dissipative
)
RHS_IP_explicit = expand(RHS_IP_explicit)
RHS_IP_explicit = simplify(RHS_IP_explicit)
assert expand(RHS_IP - RHS_IP_explicit) == 0

# Proof using the C_jj, j=1,2,3 viscous matrices,
# this is the usual form we use in proofs.
# This form assumes that the normal vector is aligned with a Cartesian
# coordinate direction. The wall velocity may have an arbitrary normal
# component.

def C_11(V):
    rho,u1,u2,u3,T = V
    mu, kappa = Mu(T), Kappa(T)
    C = zeros(5,5)
    C[1,1] = Rational(4,3)*T*mu
    C[1,4] = Rational(4,3)*T*mu*u1
    C[2,2] = T*mu
    C[2,4] = T*mu*u2
    C[3,3] = T*mu
    C[3,4] = T*mu*u3
    C[4,1] = C[1,4]
    C[4,2] = C[2,4]
    C[4,3] = C[3,4]
    C[4,4] = T**2*kappa+Rational(1,3)*T*mu*(4*u1**2+3*u2**2+3*u3**2)
    return C

def C_22(V):
    rho,u1,u2,u3,T = V
    mu, kappa = Mu(T), Kappa(T)
    C = np.zeros([5,5], dtype=object)
    C[1,1] = T*mu
    C[1,4] = T*mu*u1
    C[2,2] = Rational(4,3)*T*mu
    C[2,4] = Rational(4,3)*T*mu*u2
    C[3,3] = T*mu
    C[3,4] = T*mu*u3
    C[4,1] = C[1,4]
    C[4,2] = C[2,4]
    C[4,3] = C[3,4]
    C[4,4] = T**2*kappa+Rational(1,3)*T*mu*(3*u1**2+4*u2**2+3*u3**2)
    return C

def C_33(V):
    rho,u1,u2,u3,T = V
    mu, kappa = Mu(T), Kappa(T)
    C = zeros(5, 5)
    C[1,1] = T*mu
    C[1,4] = T*mu*u1
    C[2,2] = T*mu
    C[2,4] = T*mu*u2
    C[3,3] = Rational(4,3)*T*mu
    C[3,4] = Rational(4,3)*T*mu*u3
    C[4,1] = C[1,4]
    C[4,2] = C[2,4]
    C[4,3] = C[3,4]
    C[4,4] = T**2*kappa+Rational(1,3)*T*mu*(3*u1**2+3*u2**2+4*u3**2)
    return C

if k is not None:
    if k in (1, 2):
        C_m = C_11(V_m)
        C_p = C_11(V_p)
    if k in (3, 4):
        C_m = C_22(V_m)
        C_p = C_22(V_p)
    if k in (5, 6):
        C_m = C_33(V_m)
        C_p = C_33(V_p)
    L = -sigma_IP * (C_m + C_p)/2
    M = dot(L, W_m - W_p)
    RHS_IP_v2 = dot(W_m, M)
    RHS_IP_v2 = expand(RHS_IP_v2)
    RHS_IP_v2 = simplify(RHS_IP_v2)
    assert RHS_IP == RHS_IP_v2
\end{lstlisting}

\section{Symbolic proof: volume contribution}
\label{app:SymbolicVolumeContribution}
This appendix presents a symbolic verification of the entropy-conservation analysis for the semi-discrete formulation on arbitrary moving grids.
The derivation checks the SBP identities, the ALE volume entropy reduction with the spatial and temporal geometric conservation-law residuals retained explicitly, and the entropy-conservative coupling across interior interfaces.
\begin{lstlisting}[language=Python, caption={Symbolic proof for the semi-discrete entropy analysis.}, label={lst:symbolic-proof-volume}]
import argparse
from functools import reduce
import sympy as sp

def gll_sbp_operator(degree):
    # Construct the exact diagonal-norm LGL SBP operator 
    # of degree p
    if degree < 1:
        raise ValueError("The polynomial degree must be at least one.")

    x = sp.Symbol("x", real=True)
    Pp = sp.legendre(degree, x)

    # LGL nodes: endpoints and roots of P_p'(x).
    interior_nodes = sp.Poly(sp.diff(Pp, x), x).all_roots()
    nodes = [sp.Integer(-1), *interior_nodes, sp.Integer(1)]
    nodes = sorted(nodes, key=lambda z: float(sp.N(z, 30)))

    # LGL quadrature weights.
    weights = [ sp.simplify( sp.Rational(2, degree * (degree + 1)) / sp.legendre(degree, xi) ** 2 ) for xi in nodes ]

    n = degree + 1
    D1 = sp.zeros(n, n)

    # LGL collocation differentiation matrix.
    for i in range(n):
        for j in range(n):
            if i != j:
                D1[i, j] = sp.simplify( sp.legendre(degree, nodes[i]) / ( sp.legendre(degree, nodes[j]) * (nodes[i] - nodes[j])))

    D1[0, 0] = -sp.Rational(degree * (degree + 1), 4)
    D1[-1, -1] = sp.Rational(degree * (degree + 1), 4)

    P1 = sp.diag(*weights)
    Q1 = sp.simplify(P1 * D1)
    B1 = sp.diag(-1, *([0] * (n - 2)), 1)

    return nodes, P1, D1, Q1, B1

def kron(*args):
    return reduce(lambda a, b: sp.kronecker_product(a, b), args)

def vec(prefix, n):
    return sp.Matrix(sp.symbols(f"{prefix}0:{n}", real=True))

def scalar(x):
    if isinstance(x, sp.MatrixBase):
        if x.shape != (1, 1):
            raise ValueError(f"Expected a 1x1 matrix, got shape {x.shape}.")
        return x[0, 0]
    return x

def is_zero(expr):
    expanded = sp.expand(expr)
    return expanded == 0 or sp.simplify(expanded) == 0

def flux_reduce(expr):
    if isinstance(expr, sp.MatrixBase):
        return expr.applyfunc(lambda e: sp.factor(sp.together(e)))
    return sp.factor(sp.together(expr))

def pair_average(v, i, j):
    return sp.Rational(1, 2) * (v[i] + v[j])

def pair_metric_velocity(M, V, l, i, j):
    # K_l(i,j) = sum_m {{M_lm}}_ij {{V_m}}_ij
    return sum( pair_average(M[l][m], i, j) * pair_average(V[m], i, j) for m in range(3) )

def pairwise_gcl_divergence(D, M, V, one):
    # Return sum_l 2 (D_l o K_l) 1 without forming 
    # dense K_l matrices
    N = one.rows
    result = sp.zeros(N, 1)

    for l in range(3):
        for i in range(N):
            acc = sp.Integer(0)

            for j in range(N):
                dij = D[l][i, j]

                if dij != 0:
                    acc += ( 2 * dij * pair_metric_velocity(M, V, l, i, j) )

            result[i] += acc

    return result

# 3D tensor-product diagonal-norm SBP operator

parser = argparse.ArgumentParser()
parser.add_argument( "-p", "--degree", type=int, default=1, help="Polynomial degree of the 1-D LGL SBP operator." )
args = parser.parse_args()

nodes_1d, P1, D1, Q1, B1 = gll_sbp_operator(args.degree)
I1 = sp.eye(args.degree + 1)

# The selected polynomial degree p gives p + 1 LGL nodes 
# in each coordinate direction.

P = kron(P1, P1, P1)

D = [kron(D1, I1, I1),
     kron(I1, D1, I1),
     kron(I1, I1, D1)]

Q = [P * Dl for Dl in D]

# Construct the physical boundary matrices independently
#  of Q so the SBP identity below is a genuine check 
# rather than B := Q + Q^T.
B = [kron(B1, P1, P1),
     kron(P1, B1, P1),
     kron(P1, P1, B1)]

N = P.rows
one = sp.ones(N, 1)

# Verify the one-dimensional SBP identity Q_1 + Q_1^T = B_1.
assert all(is_zero(e) for e in Q1 + Q1.T - B1)

# Cache nonzero Q entries. 
# The tensor-product Q_l are sparse, and the nonlinear
# shuffle sums only need these pairs.
q_pairs = []

for l in range(3):
    q_pairs.append([(i, j, Q[l][i, j])
                    for i in range(N)
                    for j in range(N)
                    if Q[l][i, j] != 0])

for l in range(3):
    # Verify the tensor-product SBP identity in direction l.
    assert all(is_zero(e) for e in Q[l] + Q[l].T - B[l])

    # Verify exact differentiation of constants.
    assert all(is_zero(e) for e in D[l] * one)

    # Verify the discrete telescoping relation 1^T Q_l = 1^T B_l.
    assert all(is_zero(e) for e in one.T * Q[l] - one.T * B[l])

# Generic nonlinear SBP telescoping identity for 
# a symmetric two-point flux.
w = vec("w", N)

# Only create flux symbols actually used by Q, 
# plus the diagonal symbols needed by the boundary term.
flux_symbols = {}

def symmetric_flux_symbol(i, j):
    key = (min(i, j), max(i, j))

    if key not in flux_symbols:
        flux_symbols[key] = sp.Symbol( f"F_{key[0]}_{key[1]}", real=True)

    return flux_symbols[key]

for l in range(3):
    direct = 2 * sum( qij * w[i] * symmetric_flux_symbol(i, j) for i, j, qij in q_pairs[l] )
    boundary = sum( B[l][i, i] * w[i] * symmetric_flux_symbol(i, i) for i in range(N) if B[l][i, i] != 0 )
    shuffle = sum( qij * (w[i] - w[j]) * symmetric_flux_symbol(i, j) for i, j, qij in q_pairs[l] )

    # Verify the nonlinear SBP telescoping identity 
    # for a symmetric two-point flux.
    assert is_zero(direct - boundary - shuffle)

# Explicit ALE entropy-conservative two-point flux 
# and volume identity. We first check pointwise physical
# and two-point flux identities. Primitive variables are
# V = [rho, u_1, u_2, u_3, T]. The logarithmic mean is
# written as (b-a)/(log(b)-log(a)), which is equivalent
# to (b-a)/log(b/a) for positive states but is easier 
# for exact symbolic reduction.

R = sp.Symbol("R", real=True, positive=True)
gamma = sp.Symbol("gamma", real=True, positive=True)
rho_inf = sp.Symbol("rho_inf", real=True, positive=True)
T_inf = sp.Symbol("T_inf", real=True, positive=True)

cv = R / (gamma - 1)
cp = R * gamma / (gamma - 1)

def LogAverage(a, b):
    if a == b:
        return a
    return (b - a) / (sp.log(b) - sp.log(a))

def pressure(V):
    return V[0] * R * V[4]

def specific_entropy(V):
    rho = V[0]
    T = V[4]
    return cv * (sp.log(T) - sp.log(T_inf)) - R * (sp.log(rho) - sp.log(rho_inf))

def kinetic_energy(u):
    return scalar(u.T * u) / 2

def ConservativeVars(V):
    rho = V[0]
    u = sp.Matrix(V[1:4])
    T = V[4]
    E = cv * T + kinetic_energy(u)
    return sp.Matrix([rho, rho * u[0], rho * u[1], rho * u[2], rho * E])

def EntropyVars(V):
    rho = V[0]
    u = sp.Matrix(V[1:4])
    T = V[4]
    h = cp * T
    s = specific_entropy(V)
    k = kinetic_energy(u)
    return sp.Matrix([h / T - k / T - s, u[0] / T, u[1] / T, u[2] / T, -1 / T])

def MathematicalEntropy(V):
    return -V[0] * specific_entropy(V)

def EntropyPotential(V):
    return R * V[0]

def EntropyFlux(m, V):
    return MathematicalEntropy(V) * V[m + 1]

def EntropyFluxPotential(m, V):
    return R * V[0] * V[m + 1]

def InviscidFlux(m, V):
    rho = V[0]
    u = sp.Matrix(V[1:4])
    T = V[4]
    p = pressure(V)
    H = cp * T + kinetic_energy(u)
    em = sp.eye(3)[:, m]
    momentum = rho * u[m] * u + p * em
    return sp.Matrix([rho * u[m], momentum[0], momentum[1], momentum[2], rho * u[m] * H])

def EntropyConservativeState(Vl, Vr):
    rhol = Vl[0]
    rhor = Vr[0]
    ul = sp.Matrix(Vl[1:4])
    ur = sp.Matrix(Vr[1:4])
    betal = 1 / Vl[4]
    betar = 1 / Vr[4]

    rho_log = LogAverage(rhol, rhor)
    beta_log = LogAverage(betal, betar)
    u_avg = (ul + ur) / 2

    energy = rho_log * ( R / ((gamma - 1) * beta_log) + sp.Rational(1, 2) * scalar(ul.T * ur) )

    return sp.Matrix([rho_log, rho_log * u_avg[0], rho_log * u_avg[1], rho_log * u_avg[2], energy])

def EntropyConservativeFlux(m, Vl, Vr):
    rhol = Vl[0]
    rhor = Vr[0]
    ul = sp.Matrix(Vl[1:4])
    ur = sp.Matrix(Vr[1:4])
    betal = 1 / Vl[4]
    betar = 1 / Vr[4]

    rho_log = LogAverage(rhol, rhor)
    beta_log = LogAverage(betal, betar)
    p_avg = (rhol + rhor) / (betal + betar)
    u_avg = (ul + ur) / 2
    K_avg = (kinetic_energy(ul) + kinetic_energy(ur)) / 2

    mdot = rho_log * u_avg[m]
    em = sp.eye(3)[:, m]
    momentum = mdot * u_avg + R * p_avg * em
    energy = mdot * ( R / ((gamma - 1) * beta_log) - K_avg ) + scalar(momentum.T * u_avg)

    return sp.Matrix([mdot, momentum[0], momentum[1], momentum[2], energy])

def InviscidFluxALE(n, V, GridV):
    return sum(( n[m] * InviscidFlux(m, V) for m in range(3) ), sp.zeros(5, 1)) - scalar(n.T * GridV) * ConservativeVars(V)

def EntropyConservativeFluxALE(n, Vl, Vr, GridV):
    return sum(( n[m] * EntropyConservativeFlux(m, Vl, Vr) for m in range(3) ), sp.zeros(5, 1)) - scalar(n.T * GridV) * EntropyConservativeState(Vl, Vr)

def EntropyPotentialALE(n, V, GridV):
    return sum( n[m] * ( EntropyFluxPotential(m, V) - GridV[m] * EntropyPotential(V) ) for m in range(3) )

rhoL, rhoR = sp.symbols("rhoL rhoR", real=True, positive=True)
TL, TR = sp.symbols("TL TR", real=True, positive=True)
uL = sp.symbols("uL1:4", real=True)
uR = sp.symbols("uR1:4", real=True)

VL = sp.Matrix([rhoL, uL[0], uL[1], uL[2], TL])
VR = sp.Matrix([rhoR, uR[0], uR[1], uR[2], TR])

WL = EntropyVars(VL)
WR = EntropyVars(VR)
dW = WR - WL

# Verify the physical entropy potential phi = w^T q - S.
assert is_zero(flux_reduce( scalar(WL.T * ConservativeVars(VL)) - MathematicalEntropy(VL) - EntropyPotential(VL) ))

# Verify the physical entropy-flux potentials 
# psi_m = w^T f_m - F_m^S.
for m in range(3):
    assert is_zero(flux_reduce( scalar(WL.T * InviscidFlux(m, VL)) - EntropyFlux(m, VL) - EntropyFluxPotential(m, VL) ))

# Verify consistency of the temporal two-point state U^#.
assert all(is_zero(e) for e in flux_reduce( EntropyConservativeState(VL, VL) - ConservativeVars(VL) ))

# Verify symmetry of the temporal two-point state U^#.
assert all(is_zero(e) for e in flux_reduce( EntropyConservativeState(VL, VR) - EntropyConservativeState(VR, VL) ))

# Verify the temporal Tadmor shuffle condition [w]^T U^# = [phi].
assert is_zero(flux_reduce( scalar(dW.T * EntropyConservativeState(VL, VR)) - (EntropyPotential(VR) - EntropyPotential(VL)) ))

# Verify consistency, symmetry, and the Tadmor shuffle
# condition for each physical EC flux.
for m in range(3):
    assert all(is_zero(e) for e in flux_reduce( EntropyConservativeFlux(m, VL, VL) - InviscidFlux(m, VL) ))
    assert all(is_zero(e) for e in flux_reduce( EntropyConservativeFlux(m, VL, VR) - EntropyConservativeFlux(m, VR, VL) ))
    assert is_zero(flux_reduce( scalar(dW.T * EntropyConservativeFlux(m, VL, VR)) - (EntropyFluxPotential(m, VR) - EntropyFluxPotential(m, VL)) ))

# An arbitrary normal/scaled metric vector. 
# It need not be normalized, so the same formula applies 
# to physical normal fluxes and contravariant volume fluxes.
n = sp.Matrix(sp.symbols("n1:4", real=True))
GridV = sp.Matrix(sp.symbols("vg1:4", real=True))

# Verify consistency of the normal/scaled ALE two-point flux.
assert all(is_zero(e) for e in flux_reduce( EntropyConservativeFluxALE(n, VL, VL, GridV) - InviscidFluxALE(n, VL, GridV) ))

# Verify symmetry of the normal/scaled ALE two-point flux.
assert all(is_zero(e) for e in flux_reduce( EntropyConservativeFluxALE(n, VL, VR, GridV) - EntropyConservativeFluxALE(n, VR, VL, GridV) ))

# Verify the ALE Tadmor shuffle condition.
assert is_zero(flux_reduce( scalar(dW.T * EntropyConservativeFluxALE(n, VL, VR, GridV)) - (EntropyPotentialALE(n, VR, GridV) - EntropyPotentialALE(n, VL, GridV)) ))

# We now verify the element ALE volume identity with the GCL 
# residuals left explicit. The pointwise flux identities have now 
# been verified explicitly. To keep the tensor-product proof compact,
# the element calculation uses only their scalar consequences 
# instead of rebuilding the logarithmic-mean flux at every SBP pair.
#
# S          : mathematical entropy at the nodes
# phi        : entropy potential, phi = w^T q - S
# FS[m]      : Cartesian entropy flux in x_m
# psi[m]     : Cartesian entropy-flux potential, psi_m = w^T F_m - FS_m
# M[l][m]    : J dxi_l/dx_m
# V[m]       : Cartesian mesh velocity

S = vec("S", N)
phi = vec("phi", N)
FS = [vec(f"FS{m + 1}_", N) for m in range(3)]
psi = [vec(f"psi{m + 1}_", N) for m in range(3)]
M = [[vec(f"M{l + 1}{m + 1}_", N) for m in range(3)] for l in range(3)]
V = [vec(f"V{m + 1}_", N) for m in range(3)]

# d/dt (1^T P J S)
Jdot = vec("Jdot_", N)
dE = sp.Symbol("dE_dt", real=True)

# The explicit pointwise verification above establishes:
#
#   (w_i-w_j)^T F_m^ec = psi_m_i-psi_m_j,
#   (w_i-w_j)^T U^#    = phi_i-phi_j.
#
# The following is the resulting entropy contraction of the ALE volume 
# flux.
volume_ALE = sp.Integer(0)
boundary_entropy_ALE = sp.Integer(0)

for l in range(3):
    diagonal_contraction = []
    diagonal_entropy_flux = []

    for i in range(N):
        wiFi = sp.Integer(0)
        entropy_flux_i = sp.Integer(0)

        for m in range(3):
            wiFi += M[l][m][i] * ( FS[m][i] + psi[m][i] - V[m][i] * (S[i] + phi[i]) )
            entropy_flux_i += M[l][m][i] * ( FS[m][i] - V[m][i] * S[i] )

        diagonal_contraction.append(wiFi)
        diagonal_entropy_flux.append(entropy_flux_i)

    boundary_part = sum( B[l][i, i] * diagonal_contraction[i] for i in range(N) if B[l][i, i] != 0 )
    boundary_entropy_ALE += sum( B[l][i, i] * diagonal_entropy_flux[i] for i in range(N) if B[l][i, i] != 0 )

    jump_part = sp.Integer(0)

    for i, j, qij in q_pairs[l]:
        spatial_shuffle = sum( pair_average(M[l][m], i, j) * (psi[m][i] - psi[m][j]) for m in range(3) )
        temporal_shuffle = ( pair_metric_velocity(M, V, l, i, j) * (phi[i] - phi[j]) )
        jump_part += qij * ( spatial_shuffle - temporal_shuffle )

    volume_ALE += boundary_part + jump_part

# The chain rule for the time derivative contributes
#
#   dE/dt + phi^T P Jdot.
#
time_contraction = ( dE + scalar(phi.T * P * Jdot) )

spatial_gcl_residual = []

for m in range(3):
    spatial_gcl_residual.append( sum(( D[l] * M[l][m] for l in range(3) ), sp.zeros(N, 1)))

temporal_gcl_residual = ( Jdot - pairwise_gcl_divergence( D, M, V, one ))

expected_ALE = ( dE + boundary_entropy_ALE + scalar( phi.T * P * temporal_gcl_residual ) + sum( scalar( psi[m].T * P * spatial_gcl_residual[m] ) for m in range(3)))

# Verify the ALE element entropy identity with the spatial and temporal 
# GCL residuals left explicit.
assert is_zero(time_contraction + volume_ALE - expected_ALE)

# If the spatial and temporal discrete GCL residuals vanish, the ALE volume
# contraction reduces to dE/dt plus the ALE surface entropy flux.

# ALE interface: common metric/grid velocity and entropy stability
nvar = 5

wL = vec("wL_", nvar)
wR = vec("wR_", nvar)
dw = ( wR - wL )

phiL, phiR = sp.symbols( "phiL phiR", real=True )
psiL = sp.symbols( "psiL1:4", real=True )
psiR = sp.symbols( "psiR1:4", real=True )
Aface = sp.symbols( "Af1:4", real=True )
Vface = sp.symbols( "Vf1:4", real=True )

PsiL = sum( Aface[m] * ( psiL[m] - Vface[m] * phiL ) for m in range(3) )
PsiR_left_orientation = sum( Aface[m] * ( psiR[m] - Vface[m] * phiR ) for m in range(3) )
PsiR_outward = sum( (-Aface[m]) * ( psiR[m] - Vface[m] * phiR ) for m in range(3) )

# Verify cancellation from opposite outward metrics with a common face 
# grid velocity.
assert is_zero(PsiR_outward + PsiR_left_orientation)

# The first proof has verified the entropy-conservative ALE shuffle 
# condition for the explicit physical two-point flux. 
# Add positive-semidefinite dissipation:
#
#     f_es = f_ec - 1/2 H^T H (w_R-w_L),
#
# and verify that the resulting interface entropy production is 
# non-positive.

h_symbols = sp.symbols( "h0:25", real=True )
H = sp.Matrix( 5, 5, h_symbols )
Diss = ( H.T * H )

jump_Psi = ( PsiR_left_orientation - PsiL )
contracted_f_es = ( jump_Psi - sp.Rational(1, 2) * scalar( dw.T * Diss * dw ) )
interface_ALE = ( contracted_f_es - jump_Psi )
interface_ALE_squares = ( -sp.Rational(1, 2) * scalar( (H * dw).T * (H * dw) ) )

# Verify that the entropy-stable ALE interface contribution is a non-positive
# quadratic form, -1/2 ||H (w_R-w_L)||^2.
assert is_zero(interface_ALE - interface_ALE_squares)

# The spatial and temporal GCLs remain explicit assumptions/residuals 
# rather than being proved by a simplified mapping construction. 
# The remaining physical-boundary input can be verified by the separate 
# moving-wall ghost-state script.
\end{lstlisting}

\end{document}